\documentclass[11pt]{article}%
\usepackage{amssymb,amsmath,amsfonts,amsthm,array,bm,bbm,color}%
\usepackage[T1]{fontenc}
\usepackage[utf8]{inputenc}
\usepackage[backend=biber,style=numeric,sorting=nty,natbib=true,isbn=false,url=false,date=year,maxbibnames=99]{biblatex}
\renewbibmacro*{volume+number+eid}{%
  \printfield{volume}%
  \iffieldundef{number}{}{%
    \printtext{\mkbibparens{\printfield{number}}}%
  }%
  \setunit{\addcomma\space}%
  \printfield{eid}%
}

\usepackage{mathrsfs}
\usepackage{graphicx}
\usepackage{comment}
\usepackage{enumerate}

\providecommand{\U}[1]{\protect\rule{.1in}{.1in}}

\usepackage[pdfstartview=FitH]{hyperref}
\usepackage{xcolor}
\usepackage{orcidlink}

\numberwithin{equation}{section}

\newtheorem{theorem}{Theorem}[section]
\newtheorem{lemma}[theorem]{Lemma}

\newtheorem{proposition}[theorem]{Proposition}
\newtheorem{remark}[theorem]{Remark}

\newtheorem{definition}[theorem]{Definition}

\newcommand{\ind}{\mathbf{1}}

\def\scalV#1{ {_{V^{\ast}}}\langle #1 \rangle_{V}}

\def\scalH#1{\langle #1 \rangle_{H}}
\def\bigscalH#1{\big\langle #1 \big\rangle_{H}}
\def\scalU#1{\langle #1 \rangle_{U}}

\def\<{\langle}
\def\>{\rangle}

\def\d{{\rm d}}

\def\E{\mathbb{E}}
\def\N{\mathbb{N}}
\def\P{\mathbb{P}}
\def\R{\mathbb{R}}
\def\T{\mathbb{T}}
\def\W{\mathbb{W}}

\def\cF{\mathcal{F}}
\def\cI{\mathcal{I}}

\def\cY{\mathcal{Y}}

\def\sC{\mathscr{C}}

\begin{document}

\title{Large deviation principles for stationary solutions and invariant measures of a class of SPDE with locally monotone coefficients}
\author{Yong Liu$^1$ \footnote{Email: liuyong@math.pku.edu.cn}~, \quad Bin Tang$^2$ \footnote{Email: tangbin@njust.edu.cn} ~\orcidlink{0000-0003-0367-7271}, \quad Rangrang Zhang$^3$ \,\footnote{Email: rrzhang@bit.edu.cn}
\bigskip \\
{\small $^1$LMAM, School of Mathematical Sciences, Peking University, } \\
{\small Beijing 100871, China} \\
{\small $^2$School of Mathematics and Statistics, Nanjing University of Science and Technology, } \\
{\small Nanjing, Jiangsu, 210094, China} \\
{\small $^3$School of Mathematics and Statistics, Beijing Institute of Technology, } \\
{\small Beijing, 100081, China}
}

\maketitle

\begin{abstract}
We establish the well-posedness of stationary solutions for a class of SPDEs with locally monotone coefficients, and prove the Freidlin--Wentzell large deviation principle (LDP) for these stationary solutions. The LDP for the associated invariant measures then follows via the contraction principle, avoiding the need to construct the quasi-potential and verify the Dembo--Zeitouni uniform LDP over bounded sets. By working directly with stationary solutions, we bypass these technical difficulties, thereby providing a more general and flexible framework that is adapted to additive noise, multiplicative noise, and transport-type noise. As applications, our results cover a range of SPDEs, including the stochastic reaction-diffusion equations, stochastic 1D viscous Burgers equation, stochastic 2D Navier--Stokes equations, stochastic 2D magneto-hydrodynamic equations and stochastic 3D hyper-dissipative Navier--Stokes equations.
\end{abstract}



\section{Introduction}


Stochastic partial differential equations (SPDEs) are a mainstay of modern probability theory and arise naturally in a wide range of applications, such as engineering, finance, biology, and physics. In this paper, we study the Freidlin--Wentzell large deviation principal (LDP) of the stationary solutions and invariant measures for a large class of SPDEs with locally monotone coefficients within the variational framework,
\begin{equation*}
    \d X^{\epsilon}_t = A^{\epsilon}(X^{\epsilon}_t) \, \d t + \sqrt{\epsilon} B(X^{\epsilon}_t) \, \d W_t,
\end{equation*}
where $\epsilon>0$, operators $A^{\epsilon} : V \rightarrow V^{\ast}$ and $B: V \rightarrow L_{2}(U;H)$ are measurable and $W_{t}$ is $H$-valued double-side Wiener process, see details in Section \ref{subsec: setting}. It covers some important models such as the stochastic 2D Navier--Stokes equations, the stochastic 2D magneto-hydrodynamic equations,  the stochastic 3D hyper-dissipative Navier--Stokes equations, the stochastic 1D viscous Burgers equation, the stochastic Ginzburg--Landau equation, and the stochastic reaction-diffusion equations, see details in Section \ref{sec: application}.

The well-posedness of SPDEs in the variational framework has been extensively studied under monotonicity-type conditions. Starting from the pioneering works of Pardoux \cite{Pardoux_1974_Equations, Pardouxt_1980_Stochastic}, the monotonicity method has become a fundamental tool for establishing existence and uniqueness of variational solutions to nonlinear SPDEs with multiplicative noise. This framework was subsequently extended to locally monotone conditions, allowing the treatment of a broader class of models; see \cite{Liu_2013_Wellposedness, Liu_Stochastic_2015, Krylov_1981_Stochastic, Ren_2007_Stochastic, Brzezniak_2014_Strong, Neelima_2020_Coercivity, Nguyen_2021_Nonlinear} and the references therein. Further refinements incorporating pseudomonotone operators and relaxed continuity assumptions have led to well-posedness results for fully locally monotone SPDEs, even in the presence of multiplicative and gradient-dependent noise, see \cite{Kumar_2024_Wellposedness, Rockner_2024_Wellposedness, Pan_Large_2026}.


The large deviation principle (LDP) is a fundamental tool in probability theory that quantifies the exponential decay rates of probabilities of rare events. Its modern formulation was established by Varadhan \cite{Varadhan_1966_Asymptotic} in 1966. A powerful approach is the weak convergence method, developed by Bou\'{e}, Dupuis, Ellis, Budhiraja et al. \cite{Boue_variational_1998, Budhiraja_variational_2000, Budhiraja_Large_2008, Budhiraja_2019_Analysis,Dupuis_1997_Weak}. This approach reformulates the LDP via a variational representation formula for functionals of infinite-dimensional Brownian motion, avoids the difficult exponential probability estimates, and has proved effective for a broad class of SPDEs. To further describe transitions between stationary behaviors of dynamical systems under small noise, Freidlin and Wentzell \cite{Freidlin_1988_Random, Freidlin_2012_Random} developed a systematic framework for studying long-time behavior under small random perturbations, and introduced the quasi-potential as a key tool for analyzing the LDP of associated invariant measures.

The LDP for invariant measures of SPDEs has attracted considerable attention. For reaction--diffusion equation, Freidlin and Wentzell \cite{Freidlin_1988_Random} first established this with small additive Gaussian noise and used it to describe transitions between stationary behaviors; subsequent works progressively relaxed the assumptions on the noise and nonlinearity \cite{Sowers_1992_Large, Cerrai_Large_2005, Wang_Large_2024}. For fluid models, Brzezniak and Cerrai \cite{Brzezniak_Large_2017} established the LDP for invariant measures of the 2D stochastic Navier--Stokes equations with additive noise, Cerrai and Paskal \cite{Cerrai_2022_Large} extended this to vanishing noise correlation, and \cite{Bai_2026_Large} obtained related results for 1D stochastic Burgers equations. Zhang \cite{Zhang_2012_Large} studied the LDP for invariant measures of a 1D SPDE driven by multiplicative noise with two reflecting walls. In a different direction, Klose and Mayorcas \cite{Klose_2024_Large} derived the LDP for the $\Phi^4_3$ measures via stochastic quantisation. However, all these approaches rely on constructing the quasi-potential and verifying the Dembo--Zeitouni uniform LDP over bounded sets. The upper bound for the latter requires exponential estimates for invariant measures on compact set of the state space, which is notoriously difficult to establish in infinite-dimensional settings. To our knowledge, \cite{Zhang_2012_Large} is the only infinite-dimensional result treating multiplicative noise in this framework, made possible by the two reflecting walls that keep the solution bounded and reduce the multiplicative noise to additive-like behavior. Consequently, genuinely multiplicative noise in infinite dimensions remains out of reach within this framework, and no unified approach is currently available for broader applicability.

The concept of stationary solutions for stochastic dynamical systems dates back to the foundational work of It\^{o} and Nisio \cite{Ito_Stationary_1964}, where stationary solution is defined through the consistency of finite-dimensional distributions. An alternative and more dynamical formulation defines stationary solutions via the cocycle property of the associated random dynamical system \cite{Arnold_Random_2003, Arnold_Perfect_1995}. In this paper, we adopt a definition based on the crude cocycle, adapted from \cite{Mohammed_Stable_2008} to the two-sided time axis $\R$; see Section \ref{sec: stationary solution} for the precise formulation. A key observation is that the distribution of a stationary solution at any fixed time yields an invariant measure of the equation, while the stationary solution itself carries strictly more pathwise information. The existence of stationary solutions has been studied for a variety of models, including stochastic Burgers equations \cite{Bakhtin_2014_Spacetime, Dunlap_2021_Stationary, Liu_Representation_2009}, stochastic compressible Navier--Stokes systems \cite{Breit_2019_Stationary} and stochastic delay system \cite{Jiang_2023_Global,Liu_Large_2025}. A standard construction of stationary solution is the pull-back approach \cite{Flandoli_1999_Weak, Liu_Representation_2009}, which constructs the stationary solution as the limit of solutions started at time $-n$ as $n \rightarrow +\infty$. In this paper, we employ this approach to construct the pathwise unique stationary solution for SPDEs with locally monotone coefficients. A natural further question is how these stationary behaviors change under small random perturbations, which is precisely the subject of the LDP for stationary solutions and invariant measures.

Regarding the LDP for stationary solutions, Gao et al. \cite{Gao_Large_2022} established the LDP for stationary solutions of a class of SPDEs admitting explicit representations, with stochastic Burgers equations \cite{Liu_Representation_2009} as a primary example. Liu and Tang \cite{Liu_Large_2025} proved the LDP for stationary solutions of stochastic functional differential equations with infinite delay and derived the LDP for invariant measures via the contraction principle, circumventing the quasi-potential which is difficult to define in this setting. In this paper, we extend this approach to SPDEs with locally monotone coefficients, where stationary solutions are generally not available in explicit form. Without requiring the verification of the Dembo--Zeitouni uniform LDP over bounded sets, our approach provides a general framework applicable to additive noise, multiplicative noise, and transport-type noise, and in particular yields the LDP for invariant measures of equations whose exponential estimates of invariant measures on compact sets are difficult to establish, such as the stochastic 3D hyper-dissipative Navier--Stokes equations.

\subsection{Setting and Concept} \label{subsec: setting}

Let $V \subset H \equiv H^{\ast} \subset V^{\ast}$ be a Gelfand triple. $V$ is a reflexive and separable Banach space and $V^{\ast}$ is the dual space of $V, (H, \<\cdot, \cdot \>_H )$ is a separable Hilbert space identified with its dual space by the Riesz isomorphism, and $V$ is continuously and densely embedded in $H$. The dualization between $V^{\ast}$ and $V$ is denoted by $_{V^{\ast}} \< \cdot, \cdot \>_V$ ; we obviously have
\begin{equation*}
    _{V^{\ast}} \< u, v\>_V = \<u, v\>_H , \quad u \in H, \, v \in V.
\end{equation*}
Throughout the whole paper, we assume that there exists a constant $\lambda_1 >0 $ such that
\begin{equation}\label{eq: embedding}
    \lambda_1 \|v\|_{H}^2 \leq \| v \|_{V}^{2} , \quad \forall \, v \in V.
  \end{equation}
In particular, if $H= L^2 (\T^d)$ and $V = H_0^1 (\T^d)$, then the embedding inequality \eqref{eq: embedding} with the first eigenvalues $\lambda_1$ of $\Delta$ in $\T^d$ is  the well-known Poincar\'{e} inequality. Since all the examples we apply fulfill such an inequality, we will not list it as a condition.

Let $\{W_t\}_{t \in \R }$ be a double-side cylindrical Wiener process on a separable Hilbert space $U$ with respect to a complete filtered probability space $(\Omega, \cF , \cF_t, \P)$. Let $(L_2(U; H), \| \cdot \|_2)$ denote the space of all Hilbert-Schmidt operators from $U$ to $H$. The path of $W$ lies on the space:
\begin{equation} \label{def: Wiener space}
 \W=\Big\{ \theta\in C(\mathbb{R}; \, U); \; \theta(0)=0\Big\},\quad \|\theta\|_{\W}:= \sum_{n=1}^{+\infty} \frac{ (\sup_{|t| \leq n} \|\theta(t)\|_{U}) \wedge 1}{2^{n}}.
\end{equation}

For some $\epsilon>0$, we consider the equation
\begin{equation}\label{eq: SPDE}
    \d X^{\epsilon}_t = A^{\epsilon}(X^{\epsilon}_t) \, \d t + \sqrt{\epsilon} B(X^{\epsilon}_t) \, \d W_t,
\end{equation}
where $A^{\epsilon} : V \rightarrow V^{\ast}$ and $B: V \rightarrow L_{2}(U;H)$ are measurable. For the well-posedness and LDP of stationary solutions to \eqref{eq: SPDE}, we need the following assumptions.

Fix a constant $\beta \geq 0$, there exist a constant $\gamma_0>0$ and a positive constant $0<\epsilon \ll \gamma_0$ such that the following conditions hold for all $v, v_1, v_2 \in V$:
\begin{enumerate}
    \renewcommand{\labelenumi}{[A\theenumi].}
    \item  (Hemicontinuity) The map $s \mapsto _{V^{\ast}}\< A^{\epsilon}(v_1 + s v_2), v\>_V$ is continuous on $\R$.
    \item (Local monotonicity)  \begin{equation}
        \begin{aligned}
            2 _{V^{\ast}}\<A^{\epsilon}(v_1)-A^{\epsilon}(v_2) & , v_1 - v_2 \>_V + \epsilon \, \| B(v_1) - B(v_2) \|_2^2 \\
            & \leq  -2 \gamma_0 \|v_1 - v_2\|_{V}^{2} +  \rho(v_1) \| v_1 - v_2 \|_H^2,
        \end{aligned}
    \end{equation}
    where $\rho : V \rightarrow [0,+\infty)$ is a measurable and locally bounded function in $V$ satisfying
    \begin{equation}\label{rrr-1}
        \rho (v)  \leq  C_{\rho_1}  \big( 1+  \| v \|_{H}^{\beta} \big) \| v \|_{V}^{2} + C_{\rho_2},
    \end{equation}
    with constants $C_{\rho_1}, C_{\rho_2} \geq 0$.
    \item (Growth) There are positive constants $C_A,C_B$, $L_B$ and $\kappa\geq 0$ such that
    \begin{align*}
        & \|B(v)\|_2^2 \leq  C_{B}  \,  \big(1 +\|v\|_H^2 \big) +  L_{B}  \| v\|_{V}^{2} , \\
        &  \|B(v_1) - B(v_2) \|_2^2  \leq  C_{B}  \,  \|v_1- v_2\|_H^2 + L_B  \|v_1- v_2\|_V^{2},  \\
        & \|A^{\epsilon}(v)\|^{2}_{V^{\ast}} \leq C_{A} \big(1+ \| v\|_{V}^{2} \big) \big(1+ \| v\|_{H}^{\kappa} \big), \\
        & \|A^{\epsilon}(v)-A^{0}(v)\|^{2}_{V^{\ast}} \leq C_{A}\epsilon^2 \big(1+ \| v\|_{V}^{2} \big) \big(1+ \| v\|_{H}^{\kappa} \big) .
    \end{align*}
    \item (Noise boundness) $\scalH{v,B(v) \cdot}$ is a Hilbert--Schmidt operator that fulfills
    \begin{align*}
       \max \{  1, \|v\|_{H}^{\beta} \} \sum_{k} \big| \scalH{v,B(v) \, \mathcal{U}_{k}}  \big|^2 & \leq  C_{B}  \,  \|v\|_{H}^2, \quad \forall \, v \in V , \\
        \sum_{k} \Big| \bigscalH{u-v, \big(B(u)-B(v) \big) \, \mathcal{U}_{k}}  \Big|^2 & \leq  C_{B}  \,  \|u-v\|_{H}^4, \quad \forall \, u, \, v \in V,
    \end{align*}
    where $\{ \mathcal{U}_{k} \}_{k }$ is a complete orthonormal basis of $U$ and $C_B$ is given in {\rm{[A3]}}.
    \item (Coercivity) There are nonnegative constants $C_{A,\rho,\epsilon}$ and $C_{A,\rho}$ such that
    \begin{equation}  \label{eq: coercivity conditions}
        \max \{ \| v\|_{H}^{\beta} ,1 \} \, \Big( 2 _{V^{\ast}}\<A^{\epsilon}(v)  , v\>_V + \epsilon  \, \| B(v) \|_2^2 +  \gamma_0 \|v\|_{V}^{2} \Big) \leq C_{A,\rho,\epsilon} \leq C_{A, \rho}.
    \end{equation}
\end{enumerate}

We make some comments on the above conditions. The upper bound \eqref{rrr-1} on $\rho$ differs from that in \cite{Liu_Large_2020}, where the condition $\rho(v) \leq C(1 + \|v\|_V^{\alpha})(1 + \|v\|_H^{\beta})$ is proposed with general $\alpha > 1$, in two aspects. First, we explicitly separate the dependence on the
parameters $C_{\rho_1}$ and $C_{\rho_2}$, since they play different roles in establishing the LDP for stationary solutions and must be treated separately. Second, we fix $\alpha = 2$, which is necessary to ensure the exponential dissipation of the solution, see Section \ref{subsec: technical comment} for a detailed discussion. We consider a family of operators $\{A^{\epsilon}\}_{\epsilon>0}$ subject to condition {\rm [A3]}. This setup is intended to include transport noise with Stratonovich integral, whose It\^o-Stratonovich correction term is a higher-order infinitesimal in $\epsilon$ and thus vanishes as $\epsilon \rightarrow 0$. Regarding the other conditions, when $\epsilon$ is sufficiently small, the coercivity condition {\rm [A5]} can be derived from {\rm [A1--A3]}; see Lemma \ref{lem-1}. The noise boundedness condition {\rm [A4]} is introduced to ensure  exponential estimates of solutions; see Section \ref{subsec: technical comment} for further discussions.

It should be noted that the set of noises satisfying condition {\rm [A4]} is non-empty. A prominent example is transport-type noise, such as the Kraichnan noise (see \cite{Galeati_2023_Ldp, Flandoli_2024_Quantitative, Gess_2025_Stabilization} for more details), which takes the form $B(v)\,\mathcal{U}_k := \sigma_k \cdot \nabla v$ with $\{\sigma_k\}$ being divergence-free vector fields. We remark that pure transport-type noise may render the stationary solution trivial (i.e., the zero solution), so in practice it is typically combined with an additive or multiplicative component to ensure nontrivial dynamics. In particular, when $\beta = 0$, condition {\rm [A4]} holds for additive noise and bounded multiplicative noise. When $\beta > 0$, condition {\rm [A4]} is satisfied if the noise has special structure, for instance if $B(v)\,\mathcal{U}_k$ is orthogonal to $v$ in $H$, or if the Hilbert--Schmidt norm $\|B(v)\|_2^2 \leq C_{B} \| v \|_{H}^{-\beta}$, i.e., the noise intensity decays as $\|v\|_H$ grows.


Under assumptions {\rm [A1-A3]} and {\rm [A5]}, Liu et al. \cite{Liu_Stochastic_2015} and Pan et al. \cite{Pan_Large_2026} establish the existence and pathwise uniqueness of the solution $X^{\epsilon}_{t;t_0,\xi}$ to Eq. \eqref{eq: SPDE}  on any finite time interval $[t_0,T]$. These solutions naturally give rise to a family of solution maps $\{ U^{\epsilon}_{t_0}\}_{t_0 \in \R}$, which form a crude cocycle $(U^{\epsilon}, \theta)$ over the Wiener shift $\theta$; see Section \ref{sec: stationary solution} for the precise construction. Define $\sC_{-\infty}$ as the space $C(\R;H) \cap  \{ \cap_{N \geq 1} L^2([-N,N];V) \}$ equipped with
    \begin{equation*}
         d_{-\infty} (X,Y) := \Big(  \sum_{N=1}^{+\infty} \frac{ 1 \wedge \sup_{s\in [-N,N] } \| X_s -Y_s \|_{H}^2 \wedge \int_{-N }^N  \| X_s -Y_s \|_{V}^2 \, \d s}{2^{N}} \Big)^{1/2},
\end{equation*}
for all $X,Y \in \sC_{-\infty} $. A stationary solution for $(U^{\epsilon}, \theta)$ is an $\mathcal{F}_t$-progressively measurable process $\mathcal{Y}^{\epsilon}: \mathbb{R} \times \Omega \rightarrow H$ lying in the path space $\sC_{-\infty}$ and satisfying
\begin{equation*}
    U^{\epsilon}_0\big(t, \mathcal{Y}^{\epsilon}(s,\omega), \theta(s,\omega)\big) = \mathcal{Y}^{\epsilon}(t+s, \omega)=\mathcal{Y}^{\epsilon}(t,\theta(s, \omega)), \quad \forall\, t \geq 0, \, s \in \mathbb{R},
    \quad \mathbb{P}\text{-a.s.} .
\end{equation*}
This means that $\mathcal{Y}^{\epsilon}(t+s,\omega)$ is the solution of Eq.~\eqref{eq: SPDE} starting from $\mathcal{Y}^{\epsilon}(s,\omega)$ at any starting time $s \in \R$; see Definition~\ref{def: stationary solution} and the subsequent remark for further details. In particular, $\nu^{\epsilon} := \mathbb{P} \circ (\mathcal{Y}^{\epsilon}(0))^{-1}$ defines an invariant measure of Eq.~\eqref{eq: SPDE}.

\subsection{Main results}

Before stating the main results, we introduce two parameters used throughout the paper. Define $\tilde{\gamma}_0$ and $\tilde{\epsilon}$ as follows:
\begin{align}
  \tilde{\gamma}_0:=&\frac{C_{\rho_2}}{\lambda_1} \vee \big( \frac{(4+\beta) C_{\rho_1} C_{A,\rho}}{\lambda_1} \big)^{1/2},  \label{def: gamma}\\
   \tilde{\epsilon}:=&\frac{ (1+\beta)\lambda_1 \gamma_0^2}{8 (2+\beta)^2  C_{\rho_1} C_{B}} \wedge \frac{\gamma_0}{(18 \gamma_0 +  \beta  (2+\beta)  C_{\rho_1} )C_B} \Big( 2 \lambda_1 \gamma_0-\frac{   (4+\beta) C_{\rho_1} C_{A,\rho}   }{\gamma_0} -C_{\rho_2}  \Big), \label{def: epsilon}
\end{align}
where $C_{A,\rho}$ is given by the coercivity condition {\rm [A5]} or Lemma \ref{lem-1}. Our first main result concerns the well-posedness of stationary solutions to Eq. \eqref{eq: SPDE}.

\begin{theorem} [Theorem \ref{thm: stationary solution}] \label{thm: stationary solution intro}
 Assume that the conditions {\rm [A1-A5]} hold. If $  \tilde{\gamma}_0<\gamma_0$, then for any $0 \leq \epsilon< \tilde{\epsilon}$, there is a pathwise unique stationary solution $\cY^{\epsilon} \in  L^{\infty}(\R; L^{2}(\Omega; H))$ for cocycle $(U^{\epsilon},\theta)$.
For any $-\infty<t_0<t<+\infty$, the stationary solution $\cY^{\epsilon} \in L^{2}(\Omega;\sC_{-\infty})$ $\P$-a.s. satisfies
\begin{equation}\label{rrr-7}
    \scalV{\cY^{\epsilon}(t),v} = \scalV{\cY^{\epsilon}(t_0),v} + \int_{t_0}^{t} \scalV{A^{\epsilon}(\cY^{\epsilon}(s)),v} \d s + \sqrt{\epsilon} \int_{t_0}^{t} \scalH{B (\cY^{\epsilon}(s)) \, \d W_s,v},
\end{equation}
for any $v \in V$. Conversely, if a $\cF_{t}$-adapted process $\cY^{\epsilon}$ satisfies Eq. \eqref{eq: stationary solution} and $\sup_{t \in \R} \E \|\cY^{\epsilon}(t) \|_{H}^2< +\infty$, then it is indeed the stationary solution for cocycle $(U^{\epsilon},\theta)$. There exists a Borel measurable map $\mathcal{G}^{\epsilon}_{-\infty}: \W \rightarrow \sC_{-\infty}$ such that $\mathcal{G}^{\epsilon}_{-\infty}(W)(\cdot)$ is an $\mathcal{F}_{t}$-adapted solution to Eq. \eqref{eq: stationary solution} and $\sup_{t \in \R} \E \|\mathcal{G}^{\epsilon}_{-\infty}(W)(t) \|_{H}^2< +\infty$, called the strong stationary solution for cocycle $(U^{\epsilon},\theta)$.
\end{theorem}

A few remarks on the above theorem are in order. The test functions in \eqref{rrr-7} are taken from $V$ and the starting time $t_0$ is arbitrary; in this sense, the stationary solution satisfying \eqref{rrr-7} is referred to as the ``very weak solution'' in \cite[Definition3.4]{Brzezniak_Large_2017}. The existence and uniqueness of the stationary solution are established via the pull-back approach, relying on the exponential estimates in Propositions \ref{prop: basic energy estimate beta} and \ref{prop: basic energy estimate}. The existence of the Borel measurable map $\mathcal{G}^{\epsilon}_{-\infty}: \W \rightarrow \sC_{-\infty}$ follows from pathwise uniqueness combined with a modified Yamada--Watanabe--Engelbert theorem (Lemma \ref{lem: Yamada-Watanabe}), which is adapted to handle stationary solutions.

The following results  are concerned with LDP, so it is necessary to recall some definitions from \cite{Budhiraja_Large_2008}. Let $\mathcal{E}$ be a Polish space. The rate function and LDP on $\mathcal{E}$ are defined as follows.
\begin{definition}[Rate functions on compact set]
A family of rate functions $\cI_{\xi}$ on $\mathcal{E}$, parametrized by $\xi \in H$, is said to have compact level sets on compacts if for all compact subsets $K \subset \mathcal{E}$ and each $M<+\infty$, $\Lambda_{M,K}= \cup_{\xi \in K} \{\Phi \in \mathcal{E}; \, \cI_{\xi}(\Phi) \leq M\}$ is a compact subset of $\mathcal{E}$.
\end{definition}
\begin{definition}[LDP]
    We say a sequence $\{ X^{\epsilon}\}_{\epsilon>0} \subset \mathcal{E}$ satisfies the large deviation principle with rate function $\cI$, if the following two conditions hold:
    \begin{itemize}
        \item [1.] Large deviation upper bound. For each closed subset $F$ of $\mathcal{E}$,
            \begin{equation*}
                \limsup_{\epsilon \rightarrow 0} \epsilon\,\log \P(X^{\epsilon}\in F)\leq - \cI(F).
            \end{equation*}
        \item [2.] Large deviation lower bound. For each open subset $G$ of $\mathcal{E}$,
            \begin{equation*}
                \liminf_{\epsilon \rightarrow 0} \epsilon\,\log \P(X^{\epsilon}\in G) \geq -\cI(G).
            \end{equation*}
    \end{itemize}
\end{definition}

\begin{definition}[Freidlin-Wentzell uniform LDP]
    Let $\mathcal{K}$ be the family of all compactness subset of $H$. We say sequence $\{X^{\epsilon}_{\cdot;t_0,\xi}\}_{\epsilon>0}\subset \mathcal{E}$ indexed by $\xi \in H$ satisfies the uniform large deviation principle with rate function $\cI_{\xi}$, uniformly in $\mathcal{K}$ if the following two conditions hold:
    \begin{itemize}
        \item [1.] Upper bound. For any $K \in \mathcal{K}$, $\delta>0$ and $s_0>0$,
            \begin{equation*}
                \limsup_{\epsilon \rightarrow 0} \sup_{\xi \in K} \sup_{s \leq s_0} \Big\{ \epsilon \log \P \big( \inf_{\Phi \in \{\Phi \in \mathcal{E}; \, \cI_{\xi}(\Phi) \leq s\} } d_{\mathcal{E}}(X^{\epsilon}_{\cdot;t_0,\xi},\phi) \geq \delta \big) + s \Big\} \leq 0.
            \end{equation*}
        \item [2.] Lower bound. For any $K \in \mathcal{K}$, $\delta>0$ and $s_0>0$,
            \begin{equation*}
                \liminf_{\epsilon \rightarrow 0} \inf_{\xi \in K} \inf_{\{\Phi \in \mathcal{E}; \, \cI_{\xi}(\Phi) \leq s_0\}} \Big\{ \epsilon\,\log \P \big( d_{\mathcal{E}}(X^{\epsilon}_{\cdot;t_0,\xi},\Phi) <\delta \big) + \cI_{\xi}(\Phi) \Big\} \geq 0.
            \end{equation*}
    \end{itemize}
\end{definition}

Taking $\mathcal{E}:=C([t_0,T];H)\cap L^2([t_0,T];V)$ equipped with distance $d_{\mathcal{E}}(X,Y):= \sup_{s \in [t_0,T]} \| X(s)-Y(s)\|_{H}^2 \wedge \int_{t_0}^T \| X(s)-Y(s) \|_{V}^2 \, \d s$, the uniform LDP of $X^{\epsilon}_{\cdot;t_0,\xi}$ states as follows.
\begin{theorem}[Theorem \ref{thm: ULDP}] \label{thm: ULDP intro}
 Under conditions {\rm [A1]-[A5]}, if $\gamma_0 \geq  \tilde{\gamma}_0$, for each $t_0\in \mathbb{R}$ and any $T>t_0$, the solution $X^{\epsilon}_{\cdot;t_0,\xi}$ of Eq. \eqref{eq: SPDE} satisfies the uniform Freidlin-Wentzell LDP over compact sets on $C([t_0,T];H)\cap L^2([t_0,T];V)$ with the rate function $\mathcal{I}_{t_0,\xi}$ defined by \eqref{def: rate func 1}.
\end{theorem}

The LDP for solutions within the variational framework on finite time intervals is relatively standard and can be found in \cite{Liu_Large_2020, Pan_Large_2026}. The uniform LDP over compact sets of initial data does not introduce substantial additional difficulty. We establish it here for two reasons: it serves as a key intermediate step in proving the LDP for stationary solutions, and the estimates developed in its proof
are also needed in the proof of this LDP. For completeness, we therefore prove it directly under assumptions {\rm [A1--A5]}.

\begin{theorem}[Theorem \ref{thm: LDP for stationary solution}] \label{thm: LDP for stationary solution intro}
Assume that the conditions {\rm [A1-A5]} hold and $\gamma_0 \geq \tilde{\gamma}_0$. The stationary solution $\{\cY^{\epsilon}\}_{\epsilon>0}$ satisfies Freidlin-Wentzell LDP on the space $\sC_{-\infty}$ with the good rate function $\mathcal{I}_{-\infty}$ given by (\ref{def: rate func 2}).
\end{theorem}

The proof follows the weak convergence approach. The key difficulty lies in establishing the convergence of controlled equations corresponding to stationary solutions on $\R$. This is achieved through the finite-interval convergence of Theorem \ref{thm: ULDP} together with uniform-in-$\epsilon$ approximation estimates between controlled equations on $[-n,T]$ and those corresponding to stationary solutions (Lemmas~\ref{lem: skeleton} and~\ref{lem: controlled eq diff}). A detailed discussion is given in Section~\ref{subsec: technical comment}.

Let $\Pi: \sC_{-\infty}\rightarrow H$ be the projection operator that maps $\Phi\in \sC_{-\infty}$ to $\Phi(0)\in H$. Due to $\nu^{\epsilon}:=\mathbb{P}\circ (\cY^{\epsilon}(0))^{-1}$ is an invariant measure of Eq. \eqref{eq: SPDE}, as a result of the contraction principle, we get the LDP for invariant measure $\nu^{\epsilon}$ with rate function
\begin{align}\label{rrr-12}
  \mathcal{I}^{\prime} (\phi):=\inf_{\Phi\in \sC_{-\infty}}\{\mathcal{I}_{-\infty}(\Phi): \Pi\Phi=\phi\}, \quad \forall \, \phi\in H.
\end{align}

\begin{theorem}[Theorem \ref{thm: LDP for invariant measures}] \label{thm: LDP for invariant measures intro}
Assume that the conditions {\rm [A1-A5]} hold and $\gamma_0 \geq \tilde{\gamma}_0$. The family of invariant measure $\{\nu^{\epsilon}\}_{\epsilon>0}$ satisfies LDP on $H$ with the good rate function $\mathcal{I}^{\prime}$.
\end{theorem}
The rate function $\mathcal{I}^{\prime}$ is obtained via the contraction principle from the LDP for stationary solutions, rather than through the quasi-potential. By the uniqueness of the rate function, both approaches yield the same rate function $\mathcal{I}^{\prime}$.

\subsection{Key arguments and technical remarks} \label{subsec: technical comment}

We comment on the key technical steps and limitations underlying the main results, organized around the following three aspects: exponential estimates uniform in $\epsilon > 0$ and bounded initial data, the existence and pathwise uniqueness of stationary solutions, and the LDP for stationary solutions and invariant measures.

\begin{itemize}
\item \textbf{Exponential estimates (Proposition \ref{prop: basic energy estimate beta} and \ref{prop: basic energy estimate}).}
In the variational framework, the difference $\|X_1 - X_2\|_H$ between two solutions typically admits an energy estimate of the form $\E\big[\exp\{-\int_0^t \rho(X_2)\,\mathrm{d}s\} \|X_1 - X_2\|_H^2\big] \leq C$, where $\rho$ is defined in the local monotonicity condition {\rm [A2]}. While this suffices for uniqueness on finite time intervals, it is inadequate for establishing the existence and uniqueness of stationary solutions on the infinite time axis. Two ingredients are needed to resolve this. First, since $\rho(X_2)$ varies along the trajectory in the pull-back approach, one needs a uniform bound on $\E\exp\{\delta\int_0^t \rho(X_2)\,\mathrm{d}s\}$ for some $\delta > 0$. Following \cite[Lemma 3.5]{Odasso_2008_Exponential}, we construct an exponential martingale and apply the supermartingale inequality to derive exponential upper bounds, for $\|\cdot\|_H^2 + \int_0^t \|\cdot\|_V^2$ and $\|\cdot\|_H^{2+\beta} + \int_0^t \|\cdot\|_H^{\beta}\|\cdot\|_V^2$, that are uniform in $\epsilon > 0$ and initial data, where $\beta \geq 0$ is the parameter in \eqref{rrr-1} characterizes the strength of the nonlinearity. The construction of the exponential martingale requires the noise to satisfy the boundedness condition {\rm [A4]} and the strengthened coercivity condition {\rm [A5]}. As a limitation, condition {\rm [A4]} with $\beta > 0$ excludes additive noise, restricting the admissible noise in the strongly nonlinear regime. Second, to counterbalance the exponential growth of $\exp\{\int_0^t \rho(X_2)\,\mathrm{d}s\}$, the solution itself must have sufficiently strong exponential dissipation. This necessitates taking $\alpha = 2$ in {\rm [A2]}, which excludes equations with only polynomial decay such as stochastic $p$-Laplace equations, and further requires the dissipation coefficient $\gamma_0 \geq \tilde{\gamma}_0$ to dominate the nonlinearity with $\epsilon$ sufficiently small.
\item \textbf{Well-posedness of stationary solutions (Theorem \ref{thm: stationary solution intro}).}
We construct the stationary solution via the pull-back approach: the sequence $\{X^\epsilon_{\cdot;-n,\xi}\}_n$ is shown to be a Cauchy sequence in $L^{2/p_0}(\Omega; \sC_{-\infty})$ using the exponential estimates of Propositions \ref{prop: basic energy estimate beta} and \ref{prop: basic energy estimate}, and its limit element $\cY^{\epsilon}$ satisfies the ``very weak solution'' formulation \eqref{rrr-7} and can be verified to be a stationary solution via Definition \ref{def: stationary solution}. For pathwise uniqueness, it suffices to assume that another stationary solution $\tilde{\mathcal{Y}}^{\epsilon}$ satisfies $\sup_{t \in \R} \E\|\tilde{\mathcal{Y}}^{\epsilon}(t)\|_H^2 < +\infty$, which is natural since the marginal distributions of stationary solution coincide at all times. In the fully locally monotone case, however, $\rho(\cdot)$ depends on both $v_1$ and $v_2$ simultaneously, which would require exponential estimates on $\tilde{\mathcal{Y}}^{\epsilon}$ analogous to those in Propositions \ref{prop: basic energy estimate beta} and \ref{prop: basic energy estimate}; such an assumption is unnatural and lies beyond the scope of the present work. Finally, the existence of a Borel measurable map $\mathcal{G}^\epsilon: \W \to \sC_{-\infty}$ such that $\mathcal{G}^\epsilon(W)(\cdot)$ is an $\mathcal{F}_t$-adapted strong stationary solution relies on a modified Yamada--Watanabe--Engelbert theorem (Lemma~\ref{lem: Yamada-Watanabe}), which combines \cite{Rockner_YamadaWatanabe_2008} and \cite{Kurtz_YamadaWatanabeEngelbert_2007} to handle $\mathcal{F}_t$-adapted solutions on the infinite time axis, since the former does not apply over infinite time and the latter is restricted to compatible solutions.
\item \textbf{LDP for stationary solutions (Theorem \ref{thm: LDP for stationary solution intro} and \ref{thm: LDP for invariant measures intro}).}
We employ the weak convergence approach to prove the LDP for stationary solutions. The key step is to establish the convergence in distribution of
\begin{equation*}
  \mathcal{Y}^{\epsilon,v^{\epsilon}}:=\mathcal{G}^{\epsilon}_{-\infty} \Big( W + \frac{1}{\sqrt{\epsilon}} \int_{0}^{\cdot} v^{\epsilon} \, \mathrm{d}s \Big) \longrightarrow \mathcal{Y}^{0,v}: = \mathcal{G}^{0}_{-\infty} \Big( \int_{0}^{\cdot} v \, \mathrm{d}s \Big) \quad \text{on} \quad \sC_{-\infty},
\end{equation*}
as $\epsilon \to 0$ and the $S_M$-valued random elements $v^{\epsilon}$ converge to $v$ in distribution; see condition (S1) in Section~\ref{sec: LDP for stationary solution} for the precise formulation. Since verifying (S1) directly is difficult, we introduce the solution $X^{\epsilon,v^{\epsilon}}_{\cdot;-n,\xi}$ of controlled equation on $[-n,T]$ as an auxiliary process. The proof of the LDP for $X^{\epsilon}_{\cdot;-n,\xi}$ on  $[-n,T]$ yields the convergence of $X^{\epsilon,v^{\epsilon}}_{\cdot;-n,\xi}$ to $X^{0,v}_{\cdot;-n,\xi}$, which solves the skeleton equation. Using Lemmas \ref{lem: skeleton} and \ref{lem: controlled eq diff}, we further obtain the uniform convergence of $X^{\epsilon}_{\cdot;-n,\xi}$ to $\mathcal{Y}^{\epsilon}$ for all $\epsilon \geq 0$ as $n \to +\infty$; this argument is analogous to the construction of stationary solutions. Combining these two convergence results allows us to verify condition (S1), thereby establishing the LDP for the family of stationary solutions $\{ \mathcal{Y}^{\epsilon}\}_{\epsilon>0}$. The LDP for invariant measures then follows from the contraction principle, yielding the same rate function as the quasi-potential approach. Our method circumvents both the verification of the Dembo--Zeitouni uniform LDP over bounded sets and the construction of the quasi-potential, and thus extends to a broad framework. A limitation of our approach is that it requires the unperturbed equation ($\epsilon = 0$) to possess a unique stationary solution. Indeed, while the perturbed system ($\epsilon > 0$) admits a unique stationary solution, the unperturbed equation may possess several distinct stationary solutions; a classical example is given by the reaction--diffusion equations studied by Freidlin and Wentzell~\cite{Freidlin_1988_Random}. Our proof relies on the uniqueness of the stationary solution at $\epsilon = 0$. In all cases we consider, this requirement is satisfied.
\end{itemize}

\subsection{Organization}

The remainder of this paper is organized as follows. Section \ref{sec: preliminaries} establishes existence results and uniform exponential estimates (Propositions~\ref{prop: basic energy estimate beta} and \ref{prop: basic energy estimate}). Section \ref{sec: stationary solution} demonstrates the existence and pathwise uniqueness of stationary solutions for the SPDEs~\eqref{eq: SPDE}. Section \ref{sec: skeleton eq} investigates the properties of the skeleton equation. Section~\ref{sec: uniform LDP} provides a detailed proof of the uniform large deviation principle (LDP) for solutions of~\eqref{eq: SPDE} over finite time intervals. Section \ref{sec: LDP for stationary solution} derives the LDP for stationary solutions and invariant measures. Section~\ref{sec: application} presents several examples that fit within our framework and compares them with previous work. Appendix \ref{sec: appendix} contains the version of the Yamada--Watanabe theorem adopted for stationary solutions.

\section{Preliminaries} \label{sec: preliminaries}



\subsection{Technical lemmas}

Recall that at the beginning of this paper, we assumed that there exists a constant $\lambda_1 >0 $ such that $\lambda_1 \|v\|_{H}^2 \leq \| v \|_{V}^{2} , \, \forall \, v \in V$. With the aid of the above inequality, the condition {\rm [A5]} can be derived from {\rm [A1-A3]}, thus it can be released for the final LDP for stationary solution and invariant measure.
Concretely, we have the following result.
\begin{lemma}\label{lem-1}
Assume that {\rm [A1-A3]} hold. If $C_{\rho_2} < \lambda_1  \gamma_0$, then for any $0 \leq \epsilon \leq \epsilon_0 < \frac{\lambda_1 \gamma_0-C_{\rho_2}}{2 C_B + \lambda_1 L_B}$,
    \begin{equation} \label{rrr-17}
     \max \{ \| v\|_{H}^{\beta} ,1 \}  \big(   2 _{V^{\ast}}\<A^{\epsilon}(v)  , v\>_V + \epsilon  \, \| B(v) \|_2^2 +  \gamma_0 \|v\|_{V}^{2} \big)  \leq  C_{A,\rho,\epsilon} \leq C_{A, \rho},
    \end{equation}
    where
    \begin{align*}
      & C_{A, \rho,\epsilon}:= \max \Big\{ \frac{\|A^{\epsilon}(0)\|_{V^{\ast}}^2}{4 \delta(\epsilon)}, \frac{2}{2+\beta} \Big(\frac{\beta}{\lambda_1 (2+\beta)} \Big)^{\frac{\beta}{2}} \, \delta(\epsilon) \Big( \frac{ \|A^{\epsilon}(0)\|_{V^{\ast}}^2}{4 \delta(\epsilon)^2} \Big)^{1+\frac{\beta}{2}}   \Big\} + \epsilon C_{B}, \\
      & C_{A, \rho}:= \max \Big\{ \frac{\sup_{0 \leq \epsilon \leq \epsilon_0}\|A^{\epsilon}(0)\|_{V^{\ast}}^2}{4 \delta(\epsilon_0)}, \frac{2}{2+\beta} \Big(\frac{\beta}{\lambda_1 (2+\beta)} \Big)^{\frac{\beta}{2}} \, \delta(\epsilon_0) \Big( \frac{ \sup_{0 \leq \epsilon \leq \epsilon_0}\|A^{\epsilon}(0)\|_{V^{\ast}}^2}{4 \delta(\epsilon_0)^2} \Big)^{1+\frac{\beta}{2}}   \Big\} + \epsilon_0 C_{B} ; \\
& \delta(\epsilon) :=\frac{\lambda_1 \gamma_0- C_{\rho_2} -2 \epsilon \, C_B - \epsilon \, \lambda_1 L_B}{2 \lambda_1}  \geq \delta(\epsilon_0)>0.
\end{align*}
It implies that the condition {\rm [A5]} holds.
\end{lemma}

\begin{proof}
  Taking $v_1=0$, condition [A2] and Cauchy inequality give
  \begin{equation*}
    2 _{V^{\ast}}\<A^{\epsilon}(v)  , v\>_V + \epsilon \, \| B(v) \|_2^2 + 2 \gamma_0 \|v\|_{V}^{2} \leq  \|A^{\epsilon}(0)\|_{V^{\ast}} \|v\|_{V} +  \epsilon \| B(v) \|_2^2 + C_{\rho_2} \|v\|_H^2 .
  \end{equation*}
  Using Young's inequality and condition [A3], we have
  \begin{align*}
   &  \| v\|_{H}^{\beta} \big( 2 _{V^{\ast}}\<A^{\epsilon}(v)  , v\>_V + \epsilon  \, \| B(v) \|_2^2 +  2 \gamma_0 \|v\|_{V}^{2} \big) \\
   & \quad \leq \delta \| v\|_{H}^{\beta}  \|v\|_{V}^2 + \frac{ \|A^{\epsilon}(0)\|_{V^{\ast}}^2}{{4 \delta}} \|v\|_{H}^{\beta} + \epsilon L_{B} \| v\|_{H}^{\beta}  \|v\|_{V}^2 + \big( C_{\rho_2}+  \epsilon \, C_B \big) \| v\|_{H}^{2+\beta} + \epsilon C_{B} \| v\|_{H}^{\beta}  \\
    &  \quad \leq (\delta + \epsilon L_{B}) \| v\|_{H}^{\beta}  \|v\|_{V}^2  +  \big( C_{\rho_2}+ \lambda_1 \delta +2 \epsilon \, C_B \big) \|v\|_H^{2+\beta} + C_{\beta} \, \delta \Big( \frac{ \|A^{\epsilon}(0)\|_{V^{\ast}}^2}{4 \delta^2} \Big)^{1+\frac{\beta}{2}} + \epsilon C_{B},
  \end{align*}
where $\delta>0$ will be decided later, $C_{\beta}=\frac{2}{2+\beta} \Big(\frac{\beta}{\lambda_1 (2+\beta)} \Big)^{\frac{\beta}{2}}$ when $\beta>0$ and $C_{\beta}=1$ when $\beta=0$. The constant $C_{\beta}$ comes from the Young inequality. Due to $C_{\rho_2} < \lambda_1 \gamma_0$ and $0 \leq \epsilon \leq \epsilon_0<\frac{\lambda_1 \gamma_0-C_{\rho_2}}{ 2 C_B + \lambda_1 L_B}$, taking $\delta=\frac{\lambda_1 \gamma_0- C_{\rho_2} -2 \epsilon \, C_B - \epsilon \, \lambda_1 L_B}{2 \lambda_1}  \geq \delta_0= \frac{\lambda_1 \gamma_0- C_{\rho_2} -2 \epsilon_0 \, C_B - \epsilon_0 \, \lambda_1 L_B}{2 \lambda_1}>0$, and using the embedding inequality \eqref{eq: embedding}, we obtain that
\begin{equation*}
   \| v\|_{H}^{\beta} \big( 2 _{V^{\ast}}\<A^{\epsilon}(v)  , v\>_V + \epsilon  \, \| B(v) \|_2^2 +  \gamma_0 \|v\|_{V}^{2} \big) \leq  C_{\beta} \, \delta \Big( \frac{ \|A^{\epsilon}(0)\|_{V^{\ast}}^2}{4 \delta^2} \Big)^{1+\frac{\beta}{2}} + \epsilon C_{B}.
\end{equation*}
When $\beta=0$, the similar estimate also holds. Due to $\delta \geq \delta_0$, we get the desired result \eqref{rrr-17}.
\end{proof}

\begin{remark}
$C_{A,\rho,\epsilon}$ is one of the key parameters governing the behavior of the solution.
For some special cases, we have $\lim_{\epsilon \rightarrow 0} C_{A,\rho,\epsilon}=0$. For example, in the setting of 2D Navier-Stokes equation, we have $C_{\rho_2}=0$ and $\|A^{\epsilon}(0)\|^{2}_{V^{\ast}}=0$, which implies $\lim_{\epsilon \rightarrow 0} C_{A,\rho,\epsilon}=0$.
\end{remark}

\subsection{Definition of solutions}
For the finite time interval, the definition of solutions to Eq. \eqref{eq: SPDE} reads as follows.
\begin{definition} \label{def: finite time solution}
    A continuous $H$-valued $\mathcal{F}_{t}$-adapted process $\{X_{t}^{\epsilon} \}_{t \in I}$ is called a solution of Eq. \eqref{eq: SPDE} in the finite time interval $I=[t_0,t_1]$ ($t_1>t_0$), if for its $\d t \otimes \d \P$-equivalence class $\{\hat{X}_{t}^{\epsilon} \}_{t \in I}$, we have $\hat{X}_{t}^{\epsilon} \in L^2(I \times \Omega, \d t \otimes \d \P; V)$ and $\P$-a.s.
    \begin{equation} \label{eq: def-solution}
        X^{\epsilon}_t= X^{\epsilon}_{t_0} + \int_{t_0}^{t} A^{\epsilon}(\bar{X}^{\epsilon}_s) \, \d s + \sqrt{\epsilon} \int_{t_0}^{t} B(\bar{X}^{\epsilon}_s) \, \d W_s, \quad t \in [t_0,t_1],
    \end{equation}
    where $\bar{X}_{t}^{\epsilon}$ is any $V$-valued progressively measurable $\d t \otimes \d \P$-version of $\hat{X}_{t}^{\epsilon}$.
\end{definition}

Under condition {\rm [A1-A3]} and coercivity condition {\rm [A5]},
\cite{Liu_Large_2020} and \cite{Pan_Large_2026} have achieved the well-posedness and LDP of solutions to Eq. \eqref{eq: SPDE} in a finite time interval, separately.

\begin{remark}
Due to $X_{t}^{\epsilon} \in L^2(I \times \Omega, \d t \otimes \d \P; V)$, it follows from the growth condition {\rm [A3]} that the Bochner integral in Eq. \eqref{eq: def-solution} is well-defined and coincides with the Pettis integral. Thus, Eq. \eqref{eq: def-solution} is equivalent to the following weak formula $\P$-a.s. holds
\begin{equation} \label{eq: weak formula}
    \<X^{\epsilon}_t,v \>_{H} = \< X^{\epsilon}_{t_0},v\>_{H}+ \int_{t_0}^{t} \scalV{A^{\epsilon}(X^{\epsilon}_s),v} \, \d s + \sqrt{\epsilon} \int_{t_0}^{t} \scalH{v,B(X^{\epsilon}_s) \, \d W_s}, \quad t \in [t_0,t_1]
\end{equation}
for all $v \in V$. Moreover, the crucial It\^o formula is given by (see \cite[Theorem 4.2.5]{Liu_Stochastic_2015})
\begin{equation} \label{eq: Ito formula}
    \| X^{\epsilon}_t \|_{H}^2 = \| X^{\epsilon}_0 \|_{H}^2 + \int_{0}^{t} \big( 2 \scalV{A^{\epsilon}(X^{\epsilon}_s),X^{\epsilon}_s} + \epsilon \| B(X^{\epsilon}_s) \|_{2}^2 \big) \, \d s + 2 \sqrt{\epsilon} \int_{0}^{t} \scalH{X^{\epsilon}_s, B (X^{\epsilon}_s) \, \d W_s}.
\end{equation}
\end{remark}

\begin{remark}
The stationary solutions of Eq. \eqref{eq: SPDE} will always be the $V$-valued $\mathcal{F}_{t}$-adapted processes. Thus, in the remainder of this article, we will not strictly distinguish between the different versions of the solution and will uniformly refer to them as $X_{t}^{\epsilon}$.
\end{remark}

With regard to the infinite time interval, we employ the following notations.
For any $t \geq t_0$, let $X^{\epsilon}_{t;t_0,\xi}$ be the solution of Eq. \eqref{eq: SPDE} in the sense of Definition \ref{def: finite time solution} starting from time $t_0$ with initial data $\xi$. For any $t \leq t_0$, let $X^{\epsilon}_{t;t_0,\xi}=:\xi$.
The state space of $X^{\epsilon}_{t;t_0,\xi}$ are defined as follows.
\begin{definition}
    For any $t_0 >-\infty$, we define $\sC_{t_0}$ as the space $C(\R;H) \cap  \{ \cap_{N \geq 1} L^2([t_0,N];V) \}$ equipped with the following distance
    \begin{equation*}
         d_{t_0} (X,Y) := \Big(  \sum_{N=1}^{+\infty} \frac{ 1 \wedge \sup_{s\in [-N,N] } \| X_s -Y_s \|_{H}^2 \wedge \int_{-N \vee t_0}^N  \| X_s -Y_s \|_{V}^2 \, \d s}{2^{N}} \Big)^{1/2},
    \end{equation*}
    for all $X,Y \in \sC_{t_0} $.
    The space $\sC_{+\infty}$ is defined as $C(\R;H)$ equipped with the metric
    \begin{equation*}
         d_{+\infty} (X,Y) := \Big(  \sum_{N=1}^{+\infty} \frac{ 1 \wedge \sup_{s\in [-N,N] } \| X_s -Y_s \|_{H}^2 }{2^{N}} \Big)^{1/2},
    \end{equation*}
    for all $X,Y \in \sC_{+\infty} $. For any $t_0 \in [ -\infty, +\infty]$ and $p \geq 1$, we define $L^{p}(\Omega,\sC_{t_0})$ as the space of all random functions lie on $\sC_{t_0}$ with the metric $ |\E |d_{t_0} (\cdot,\cdot)|^{p}|^{1/p} $.
\end{definition}

\begin{remark}
Clearly, $L^{p}(\Omega,\sC_{t_0})$ is a Polish space.
  For any $t_0 \in [ -\infty, +\infty)$, the metric $d_{t_0}(X,Y)$ is equivalent to
  \begin{equation*}
         d_{t_0,\gamma} (X,Y) := \Big(  \sum_{N=1}^{+\infty} \frac{ 1 \wedge \sup_{s\in [-N,N] } e^{-\gamma(N-s)} \| X_s -Y_s \|_{H}^2 \wedge \int_{ -N \wedge t_0  }^N  \| X_s -Y_s \|_{V}^2 \, \d s}{2^{N}} \Big)^{1/2}
  \end{equation*}
  for any $\gamma \geq 0$. The metric $d_{+\infty}(X,Y)$ is equivalent to
  \begin{equation*}
         d_{+\infty,\gamma} (X,Y) := \Big(  \sum_{N=1}^{+\infty} \frac{ 1 \wedge \sup_{s\in [-N,N] } e^{-\gamma(N-s)} \| X_s -Y_s \|_{H}^2 \d s}{2^{N}} \Big)^{1/2}
  \end{equation*}
  for any $\gamma \geq 0$. In the sequel, we will use both types of equivalent metrics.
\end{remark}

To show that the solution is a Cauchy sequence in $L^{p}(\Omega,\sC_{-\infty})$, we will frequently employ the following technical proposition, which can be proved by the diagonal argument.
\begin{proposition} \label{Prop: Cauchy sequence}
  Let $X_n \in \sC_{-n}$ for any $n \in \N$.
  For fixed $p \geq 1$, if for every $N \in \N$, the sequence $\{X_n \}_{n \geq N}$ is a Cauchy sequence in $L^{p}(\Omega,\sC_{-N})$, then there exists a limit element $X_{\infty} \in L^{p}(\Omega,\sC_{-\infty})$ such that $\{X_n \}_{n \in \N}$ converges to $X_{\infty}$ with respect to the metric $ |\E |d_{-\infty} (\cdot,\cdot)|^{p}|^{1/p} $.

\end{proposition}

Under assumptions {\rm [A1-A3]} and {\rm [A5]}, Liu et al. \cite{Liu_Stochastic_2015} and Pan et al. \cite{Pan_Large_2026} give the existence and pathwise uniqueness of the solution $X^{\epsilon}_{t}$ to Eq. \eqref{eq: SPDE} in any finite time interval. By the Yamada--Watanabe theorem, we derive the following result.
\begin{proposition}\label{Prop: measurable map truncated}
 There is a Borel measurable map $\mathcal{G}^{\epsilon}_{t_0}: \W \times H \rightarrow \sC_{t_0}$ such that $\mathcal{G}^{\epsilon}_{t_0}(W,\xi)(t,\omega)$, adapted to $\mathcal{F}_{t}$, is the unique solution of Eq. \eqref{eq: SPDE} with initial data $\xi \in H$ at $t_0$.
\end{proposition}

\subsection{Exponential estimates} \label{subsec: exponential estimate}

In this subsection, we will show two exponential estimates of $X^{\epsilon}_{t;t_0,\xi}$ uniformly in time that play a fundamental role in the proof of the existence and uniqueness of the stationary solution to Eq. \eqref{eq: SPDE}. The following result concerns the exponential estimates related to $\|\cdot\|^{2+\beta}_H$. It should be pointed out that this result also holds for $\beta=0$.
\begin{proposition} \label{prop: basic energy estimate beta}
Assume that $A^{\epsilon}$ and $B$ satisfy the conditions {\rm [A1- A5]}. For any $t_0 \in \R$, $T>t_0$ and $\xi \in H$, there is the unique solution $X^{\epsilon}_{t;t_0,\xi}$ of Eq. \eqref{eq: SPDE}, and it holds:
    \begin{description}
      \item[(i)] For any $0<\gamma \leq \frac{(1+\beta)\lambda_1 \gamma_0}{2} $, we have
    \begin{align}
        & \E \sup_{t \in [t_0,T]} \Big\{  e^{\gamma (t-t_0)} \| X^{\epsilon}_{t;t_0,\xi} \|_{H}^{2+\beta} + \gamma_0  \int_{t_0}^{t}  e^{\gamma (s-t_0)} \|X^{\epsilon}_{s;t_0,\xi}\|_{H}^{\beta} \|X^{\epsilon}_{s;t_0,\xi}\|_{V}^{2} \, \d s \Big\}  \nonumber \\
        & \quad \leq 2 e^{-\gamma(T-t_0)} \| \xi \|_{H}^{2+\beta} + \frac{(2+\beta) (C_{A,\rho,\epsilon} + (18+10\beta )  C_{B} \epsilon )}{\gamma}, \quad \quad \forall \, T \geq t_0.  \label{est: uniform bounded energy beta}
    \end{align}
      \item[(ii)]For any $0<\delta \leq \frac{ (1+\beta)\lambda_1 \gamma_{0}}{2(2+\beta)^2 C_{B} \epsilon}$, $X^{\epsilon}_{t;t_0,\xi}$ satisfies
    \begin{align}
        \E \Big[ \exp \Big\{  \delta \sup_{t \geq t_0} & \Big( \| X^{\epsilon}_{t;t_0,\xi}\|_{H}^{2+\beta} + \frac{\gamma_0}{2} \int_{t_0}^t \| X^{\epsilon}_{s;t_0,\xi} \|_{H}^{\beta} \| X^{\epsilon}_{s;t_0,\xi} \|_{V}^{2} \, \d s  \nonumber \\
         & \quad -  (1+\frac{\beta}{2}) \big( C_{A,\rho,\epsilon}+ \beta C_B \epsilon \big) (t-t_0) \Big) \Big\} \Big] \leq 2 e^{\delta \| \xi \|_{H}^{2+\beta}  }. \label{est: exp energy beta}
    \end{align}
      \item[(iii)] For any $t_1 \geq t_0$, it satisfies the following uniform estimates
    \begin{align}
        \E  \exp \Big\{ \delta \big\{ \| X^{\epsilon}_{t_1;t_0,\xi}\|_{H}^{2+\beta} \big\} \Big\}  & \leq 2 \,  e^{\delta \big(\| \xi \|_{H}^{2+\beta} + (2+\beta) \frac{ C_{A,\rho,\epsilon}+ \beta C_B \epsilon   }{\lambda_1 \gamma_0} \big)} ; \label{est: exp H beta} \\
        \E \Big[ \exp \Big\{  \frac{\delta}{2} \sup_{t \geq t_1}  \big( \frac{\gamma_0}{2} \int_{t_1}^t \| X^{\epsilon}_{s;t_0,\xi} \|_{H}^{\beta} \| X^{\epsilon}_{s;t_0,\xi} \|_{V}^{2} \, \d s & - (1+\frac{\beta}{2}) \big( C_{A,\rho,\epsilon}+ \beta C_B \epsilon \big) (t-t_1) \big) \Big\} \Big]  \nonumber \\
        &  \leq 2 \, e^{\delta  \big( \frac{\| \xi \|_{H}^{2+\beta}}{2} + \frac{2+\beta}{2} \frac{ C_{A,\rho,\epsilon} + \beta C_{B} \epsilon  }{\lambda_1 \gamma_0} \big)}. \label{est: exp V beta}
    \end{align}
    \end{description}
    The above results imply that $X^{\epsilon}_{t;t_0,\xi} \in L^{2}(\Omega; \sC_{t_0})$.
\end{proposition}
\begin{proof}
Firstly, we proceed with the proof of (i).
    By It\^o formula, coercivity condition {\rm [A5]} and condition {\rm [A4]}, we obtain
    \begin{align}
        & e^{\gamma (t-t_0)} \| X^{\epsilon}_{t;t_0,\xi} \|_{H}^{2+\beta} - \| \xi \|_{H}^{2+\beta} - (2+ \beta) \sqrt{\epsilon} \int_{t_0}^{t} e^{\gamma (s-t_0)}  \|X^{\epsilon}_{s;t_0,\xi}\|_{H}^{\beta} \scalH{X^{\epsilon}_{s;t_0,\xi}, B (X^{\epsilon}_{s;t_0,\xi}) \, \d W_s} \nonumber \\
        &  = \gamma  \int_{t_0}^{t} e^{\gamma (s-t_0)} \|X^{\epsilon}_{s;t_0,\xi}\|_{H}^{2+\beta} \, \d s +  \beta (1+ \frac{\beta}{2}) \epsilon  \int_{t_0}^{t} e^{ \gamma (s-t_0)}  \|X^{\epsilon}_{s;t_0,\xi}\|_{H}^{\beta-2} \sum_{k} \big| \scalH{X^{\epsilon}_{s;t_0,\xi},B(X^{\epsilon}_{s;t_0,\xi})  \mathcal{U}_{k}}  \big|^2 \, \d s  \nonumber \\
        &  \quad + (1+ \frac{\beta}{2})  \int_{t_0}^{t} e^{\gamma (s-t_0)}  \|X^{\epsilon}_{s;t_0,\xi}\|_{H}^{\beta} \big( 2 \, \scalV{A^{\epsilon}(X^{\epsilon}_{s;t_0,\xi}),X^{\epsilon}_{s;t_0,\xi}} + \epsilon \| B(X^{\epsilon}_{s;t_0,\xi}) \|_{2}^2 \big) \, \d s \nonumber \\
        & \leq  \int_{t_0}^{t} e^{\gamma (s-t_0)} \|X^{\epsilon}_{s;t_0,\xi}\|_{H}^{\beta} \big( \gamma \|X^{\epsilon}_{s;t_0,\xi}\|_{H}^{2} - (1+ \frac{\beta}{2}) \gamma_0 \|X^{\epsilon}_{s;t_0,\xi}\|_{V}^{2} \big) \, \d s \nonumber \\
        & \quad  + (1+\frac{\beta}{2}) \big( C_{A,\rho,\epsilon}+ \beta C_B \epsilon \big) \int_{t_0}^{t} e^{\gamma (s-t_0)} \, \d s \label{eq: energy estimate beta} \\
        & \leq -\frac{\gamma_0}{2} \int_{t_0}^{t} e^{\gamma (s-t_0)} \|X^{\epsilon}_{s;t_0,\xi}\|_{H}^{\beta}  \|X^{\epsilon}_{s;t_0,\xi}\|_{V}^{2} \, \d s + (1+\frac{\beta}{2}) (C_{A,\rho,\epsilon}+ \beta C_B \epsilon ) \min \Big\{ \frac{e^{\gamma(t-t_0)}-1}{\gamma}, (t-t_0) e^{\gamma(t-t_0)} \Big\} \nonumber ,
    \end{align}
    where $0\leq \gamma \leq \frac{(1+\beta)\lambda_1 \gamma_0}{2} $ and in the last step used the embedding inequality \eqref{eq: embedding}.

    Using Burkh\"older--Davis--Gundy inequality and condition {\rm [A4]}, we have
    \begin{align*}
        & (2+ \beta)  \sqrt{\epsilon} \, \E \sup_{t \in [t_0,T]} \Big| \int_{t_0}^{t} e^{\gamma (s-t_0)} \|X^{\epsilon}_{s;t_0,\xi}\|_{H}^{\beta} \scalH{X^{\epsilon}_{s;t_0,\xi}, B(X^{\epsilon}_{s;t_0,\xi}) \, \d W_{s}} \Big| \\
        & \leq 3 (2+ \beta) \sqrt{\epsilon} \,  \E \Big| \int_{t_0}^{T} e^{2 \gamma (s-t_0)} \|X^{\epsilon}_{s;t_0,\xi}\|_{H}^{2 \beta} \sum_{k}  \big| \scalH{X^{\epsilon}_{s;t_0,\xi},B(X^{\epsilon}_{s;t_0,\xi}) \, \mathcal{U}_{k}}  \big|^2 \, \d s \Big|^{1/2} \\
        &  \leq \frac{1}{2} \E \sup_{s\in [t_0,t]} e^{\gamma (s-t_0)}\| X^{\epsilon}_{s;t_0,\xi} \|_{H}^{2+\beta}  + \frac{9 (2+ \beta)^2  C_B \, \epsilon}{2} \,  \int_{t_0}^{t} e^{\gamma (s-t_0)}  \d s.
    \end{align*}
Combining the above estimates, we obtain
     \begin{align*}
        & \E \sup_{t \in [t_0,T]} \Big\{  e^{\gamma (t-t_0)} \| X^{\epsilon}_{t;t_0,\xi} \|_{H}^{2+\beta} + \gamma_0  \int_{t_0}^{t}  e^{\gamma (s-t_0)} \|X^{\epsilon}_{s;t_0,\xi}\|_{H}^{\beta} \|X^{\epsilon}_{s;t_0,\xi}\|_{V}^{2} \, \d s \Big\} \\
       & \leq 2 \| \xi \|_{H}^2 + \frac{(2+\beta) (C_{A,\rho,\epsilon} + (18+10\beta )  C_{B} \epsilon )}{\gamma} \big(e^{\gamma(T-t_0)} -1 \big).
    \end{align*}
    Dividing both side by $e^{\gamma (T-t_0)}$, we obtain energy estimate \eqref{est: uniform bounded energy beta}.

Now, we are on a position to prove (ii). Denote by
    \begin{equation*}
    E_{X^{\epsilon}_{t;t_0,\xi};\beta} := \| X^{\epsilon}_{t;t_0,\xi}\|_{H}^{2+\beta} + \frac{\gamma_0}{2} \int_{t_0}^t \| X^{\epsilon}_{s;t_0,\xi} \|_{H}^{\beta} \| X^{\epsilon}_{s;t_0,\xi} \|_{V}^{2} \, \d s, \quad \forall \, t \geq t_0.
    \end{equation*}
    Define
    \begin{equation*}
        M (t) := (2+\beta) \sqrt{\epsilon} \int_{t_0}^{t} \|X^{\epsilon}_{s;t_0,\xi}\|_{H}^{\beta} \scalH{X^{\epsilon}_{s;t_0,\xi}, B(X^{\epsilon}_{s;t_0,\xi}) \, \d W_s}.
    \end{equation*}
  Let $\<M\>_{t}$ be the quadratic variation of local martingale $M(t)$, for any $\delta>0$, define
    \begin{equation*}
    \mathcal{M}_{\delta}(t):= M(t) - \frac{\delta}{2} \<M\>_{t},
    \end{equation*}
    where $\delta$ will be determined latter. Using conditions {\rm [A4]}, we have
    \begin{align*}
        \d \< M \>_t & \leq (2+\beta)^2 \epsilon \, \|X^{\epsilon}_{s;t_0,\xi}\|_{H}^{2 \beta} \| \scalH{X^{\epsilon}_{t;t_0,\xi},B(X^{\epsilon}_{t;t_0,\xi})} \|_{L_2(U;\R)}^2 \, \d t \\
        & \leq (2+\beta)^2  C_{B} \,\epsilon \, \| X^{\epsilon}_{t;t_0,\xi} \|_{H}^{2+\beta} \, \d t \leq \frac{(2+\beta)^2  C_{B} \epsilon }{\lambda_1} \, \| X^{\epsilon}_{t;t_0,\xi} \|_{H}^{\beta} \| X^{\epsilon}_{t;t_0,\xi} \|_{V}^2  \, \d t.
    \end{align*}
    Notice that $\exp \{\delta \mathcal{M}_{\delta}(t)  \}$ is a positive supermartingale whose value is $1$ at time $t_0$. We deduce from maximal supermartingale inequality that
    \begin{equation} \label{est: martingale beta}
        \P \Big( \sup_{t \geq t_0 } \mathcal{M}_{2 \delta}(t) \geq \rho \Big) = \P \Big( \sup_{t \geq t_0 } e^{ 2 \delta \mathcal{M}_{2 \delta}(t)} \geq e^{2 \delta \rho} \Big) \leq e^{-2 \delta \rho},
    \end{equation}
    for any $\rho \geq 0$. Thus, we have
    \begin{equation*}
        \E \Big[   e^{ \delta \sup_{t \geq t_0 } \mathcal{M}_{2 \delta}(t)} \Big] = 1 + \delta \int_{0}^{\infty} e^{\delta \rho} \, \P \Big( \sup_{t \geq t_0 } \mathcal{M}_{2 \delta}(t) \geq \rho \Big) \, \d \rho \leq 2.
    \end{equation*}
    Taking $0<\delta \leq \frac{ (1+\beta)\lambda_1 \gamma_{0}}{2(2+\beta)^2 C_{B} \epsilon}$ and $\gamma=0$ in estimate \eqref{eq: energy estimate beta}, we obtain
    \begin{equation*}
        \| X^{\epsilon}_{t;t_0,\xi}\|_{H}^{2+\beta} + \frac{\gamma_0}{2} \int_{t_0}^t \| X^{\epsilon}_{s;t_0,\xi} \|_{H}^{\beta} \| X^{\epsilon}_{s;t_0,\xi} \|_{V}^{2} \, \d s  \leq \| \xi \|_{H}^2 + \mathcal{M}_{2\delta}(t) + (1+\frac{\beta}{2}) \big( C_{A,\rho,\epsilon}+ \beta C_B \epsilon \big) \,  (t-t_0),
    \end{equation*}
    Due to $0<\delta \leq \frac{ (1+\beta)\lambda_1 \gamma_{0}}{2(2+\beta)^2 C_{B} \epsilon}$, the following inequality gives estimate \eqref{est: exp energy beta},
    \begin{align*}
       &  \E \Big[ \exp  \Big\{  \delta \sup_{t \geq t_0} \big(  \| X^{\epsilon}_{t;t_0,\xi}\|_{H}^{2+\beta} + \frac{\gamma_0}{2} \int_{t_0}^t \| X^{\epsilon}_{s;t_0,\xi} \|_{H}^{\beta} \| X^{\epsilon}_{s;t_0,\xi} \|_{V}^{2} \, \d s   -  (1+\frac{\beta}{2}) \big( C_{A,\rho,\epsilon}+ \beta C_B \epsilon \big)(t-t_0) \big) \Big\} \Big] \\
       & \leq \E \big[ e^{ \delta \sup_{t \geq t_0 } \mathcal{M}_{2 \delta}(t)} \big] \,  e^{\delta \| \xi \|_{H}^{2+\beta}  } \leq 2 e^{\delta \| \xi \|_{H}^{2+\beta}  }.
    \end{align*}

    Finally, we focus on the proof of (iii). For any $t_1 \geq t_0$, we will show the uniform estimate of $\int_{t_1}^{t} \| X^{\epsilon}_{s;t_0,\xi} \|_{H}^{\beta}  \| X^{\epsilon}_{s;t_0,\xi} \|_{V}^{2}  \, \d s$. Due to $X^{\epsilon}_{t;t_0,\xi}= X^{\epsilon}_{t;t_1,X^{\epsilon}_{t_1;t_0,\xi}}$, analogous to estimate \eqref{est: exp energy beta}, we have
    \begin{align}
        \E \Big[ \exp \Big\{  \delta \sup_{t \geq t_1} &  \big( \frac{\gamma_0}{2} \int_{t_1}^t \| X^{\epsilon}_{s;t_0,\xi} \|_{H}^{\beta} \| X^{\epsilon}_{s;t_0,\xi} \|_{V}^{2} \, \d s  \nonumber \\
       &  \quad - \| X^{\epsilon}_{t_1;t_0,\xi} \|_{H}^{2+\beta} - (1+\frac{\beta}{2}) \big( C_{A,\rho,\epsilon}+ \beta C_B \epsilon \big) \, (t-t_1) \big) \Big\} \Big] \leq  2, \label{est: exp energy t_1 beta}
    \end{align}
    for any $0<\delta \leq \frac{ (1+\beta)\lambda_1 \gamma_{0}}{2(2+\beta)^2 C_{B} \epsilon}$ and $t_1 \geq t_0$. To obtain estimate \eqref{est: exp V beta}, we only need to show the exponential estimate of $\| X^{\epsilon}_{t_1;t_0,\xi} \|_{H}^{2+\beta}$ like \eqref{est: exp H beta}. For any $0<\gamma \leq \frac{(1+\beta) \lambda_1 \gamma_0}{2}$, let
    \begin{equation*}
        M_{\gamma}(t):= (2+\beta) \sqrt{\epsilon} \int_{t_0}^{t} e^{\gamma (s-t_0)} \|X^{\epsilon}_{s;t_0,\xi}\|_{H}^{\beta} \scalH{X^{\epsilon}_{s;t_0,\xi}, B(X^{\epsilon}_{s;t_0,\xi}) \, \d W_s}
    \end{equation*}
    and $\mathcal{M}_{\delta^{\prime},\gamma}(t):= M_{\gamma}(t) -\frac{\delta^{\prime}}{2} \<M_{\gamma} \>_{t}$. Using conditions {\rm [A4]}, we have
    \begin{align*}
        \d \< M_{\gamma}\>_t & \leq (2+\beta)^2 \epsilon \, e^{2 \gamma (t-t_0)} \|X^{\epsilon}_{t;t_0,\xi}\|_{H}^{2\beta}  \| \scalH{X^{\epsilon}_{t;t_0,\xi},B(X^{\epsilon}_{t;t_0,\xi})} \|_{L_2(U;\R)}^2 \, \d t  \\
        &  \leq (2+\beta)^2 C_{B} \, \epsilon  \, e^{\gamma (t_1-t_0)} \,   e^{\gamma (t-t_0)}    \| X^{\epsilon}_{t;t_0,\xi} \|_{H}^{2+\beta}  \d t,
    \end{align*}
    for any $t \in [t_0,t_1]$. For any $0<\delta^{\prime} <\frac{ (2+\beta)\lambda_1 \gamma_0- 2\gamma}{2 (2+\beta)^2 C_{B} \epsilon} e^{-\gamma (t_1-t_0)} $, estimate \eqref{eq: energy estimate beta} yields
    \begin{equation*}
        e^{\gamma (t-t_0)}\| X^{\epsilon}_{t;t_0,\xi}\|_{H}^{2+\beta} \leq \| \xi \|_{H}^{2+\beta} + \mathcal{M}_{2\delta^{\prime},\gamma}(t) +  (1+\frac{\beta}{2}) \frac{C_{A,\rho,\epsilon}+ \beta C_B \epsilon }{\gamma} \big(e^{\gamma(t-t_0)} -1 \big),
    \end{equation*}
    for any $t \in [t_0,t_1]$. By maximal supermartingale inequality, we obtain
    \begin{align*}
        & \P \Big( \sup_{t \in [t_0,t_1] } \big\{ e^{\gamma (t-t_0)}\| X^{\epsilon}_{t;t_0,\xi}\|_{H}^2 - \| \xi \|_{H}^{2+\beta} - (1+\frac{\beta}{2}) \frac{C_{A,\rho,\epsilon}+ \beta C_B \epsilon }{\gamma}  (e^{\gamma (t-t_0)}-1 )  \big\} \geq \rho \Big)  \\
        &  \quad  \leq \P \Big( \sup_{t \geq t_0 } \mathcal{M}_{2\delta^{\prime},\gamma}(t) \geq \rho \Big) \leq e^{-2 \delta^{\prime} \rho},
      \end{align*}
    for any $\rho>0$. Let $\delta= e^{\gamma (t_1-t_0)} \, \delta^{\prime} < \frac{ (2+\beta)\lambda_1 \gamma_0-2\gamma}{2 (2+\beta)^2 C_{B} \epsilon}$, we obtain
    \begin{equation*}
        \E \exp \Big\{ \delta \sup_{t \in [t_0,t_1] } \big\{ e^{-\gamma (t_1-t)}\| X^{\epsilon}_{t;t_0,\xi}\|_{H}^{2+\beta} - e^{-\gamma (t_1-t_0)} \| \xi \|_{H}^{2+\beta} - (1+\frac{\beta}{2}) \frac{C_{A,\rho,\epsilon}+ \beta C_B \epsilon }{\gamma} \frac{e^{\gamma (t-t_0)}-1}{e^{\gamma (t_1-t_0)}} \big\} \Big\} \leq 2.
    \end{equation*}
    In order to unify the range of values for $\delta$, we take $\gamma=\frac{\lambda_1 \gamma_0}{2} $. For any $0<\delta < \frac{ (1+\beta)\lambda_1 \gamma_{0}}{2(2+\beta)^2 C_{B} \epsilon}$, the above inequality yields estimate \eqref{est: exp H beta},
       \begin{equation*}
        \E \exp \Big\{ \delta \big\{ \| X^{\epsilon}_{t_1;t_0,\xi}\|_{H}^{2+\beta} \big\} \Big\} \leq 2 \, e^{\delta \big(\| \xi \|_{H}^{2+\beta} + (2+\beta) \frac{ C_{A,\rho,\epsilon}+ \beta C_B \epsilon   }{\lambda_1 \gamma_0} \big)}.
    \end{equation*}
    Using H\"older inequality, \eqref{est: exp H beta} and \eqref{est: exp energy t_1 beta}, we obtain
    \begin{align*}
        & \E \Big[ \exp \Big\{  \frac{\delta}{2} \sup_{t \geq t_1} \big( \frac{\gamma_0}{2} \int_{t_1}^t \| X^{\epsilon}_{s;t_0,\xi} \|_{H}^{\beta} \| X^{\epsilon}_{s;t_0,\xi} \|_{V}^{2} \, \d s- (1+\frac{\beta}{2}) \big( C_{A,\rho,\epsilon}+ \beta C_B \epsilon \big) (t-t_1) \big) \Big\} \Big]  \\
        &  \leq  \Big( \E \Big[ \exp \Big\{  \delta \sup_{t \geq t_1} \big( \frac{\gamma_0}{2} \int_{t_1}^t \| X^{\epsilon}_{s;t_0,\xi} \|_{H}^{\beta} \| X^{\epsilon}_{s;t_0,\xi} \|_{V}^{2} \, \d s - \| X^{\epsilon}_{t_1;t_0,\xi} \|_{H}^{2+\beta} -(1+\frac{\beta}{2}) \big( C_{A,\rho,\epsilon}+ \beta C_B \epsilon \big) (t-t_1) \big) \Big\} \Big] \Big)^{1/2} \\
        & \quad \quad \times \Big(  \E \exp \Big\{ \delta \big\{ \| X^{\epsilon}_{t_1;t_0,\xi}\|_{H}^{2+\beta} \big\} \Big\} \Big)^{1/2}  \leq 2 \, e^{\delta  \big( \frac{\| \xi \|_{H}^{2+\beta}}{2} + \frac{2+\beta}{2} \frac{ C_{A,\rho,\epsilon} +  \beta C_{B} \epsilon  }{\lambda_1 \gamma_0} \big)}.
    \end{align*}
    Thus, we obtain the result \eqref{est: exp V beta}.
\end{proof}

As mentioned at the beginning of this section, the ``$\beta=0$'' version of Proposition \ref{prop: basic energy estimate beta} also holds. For convenience, we list it below.
\begin{proposition} \label{prop: basic energy estimate}
    Assume that $A^{\epsilon}$ and $B$ satisfy the conditions {\rm [A1-A5]}, for any $t_0 \in \R$, $T>t_0$ and $\xi \in H$, let $X^{\epsilon}_{t;t_0,\xi}$ be the unique solution of Eq. \eqref{eq: SPDE}, the following results hold.
    \begin{description}
      \item[(i)] For any $0 < \gamma \leq \frac{\lambda_1 \gamma_0 }{2}$ and $0 \leq \epsilon \leq \epsilon_0 < \frac{\lambda_1 \gamma_0-C_{\rho_2}}{2 C_B + \lambda_1 L_B}$, there is a uniform energy estimate
    \begin{align}
        & \E \sup_{t\in [t_0,T]} e^{-\gamma (T-t)} \Big\{ \| X^{\epsilon}_{t;t_0,\xi} \|_{H}^2 + \gamma_0  \int_{t_0}^{t}  e^{-\gamma (t-s)} \|X^{\epsilon}_{s;t_0,\xi}\|_{V}^{2} \, \d s \Big\} \nonumber \\
        & \quad \leq 2 e^{-\gamma(T-t_0)} \| \xi \|_{H}^2 + \frac{2 C_{A,\rho,\epsilon} + 36 \epsilon \, C_{B} }{\gamma} , \quad \quad \forall \, T \geq t_0.  \label{est: uniform bounded energy}
    \end{align}
      \item[(ii)] For any $0<\delta <  \frac{\lambda_1 \gamma_0 }{8 \epsilon \, C_{B}  }$, $X^{\epsilon}_{t;t_0,\xi}$ satisfies
    \begin{equation} \label{est: exp energy}
        \E \Big[ \exp \Big\{  \delta \sup_{t \geq t_0} \Big( \| X^{\epsilon}_{t;t_0,\xi}\|_{H}^2 + \frac{\gamma_0}{2} \int_{t_0}^t \| X^{\epsilon}_{s;t_0,\xi} \|_{V}^{2} \, \d s- C_{A,\rho,\epsilon}  (t-t_0) \Big) \Big\} \Big] \leq 2 e^{\delta \| \xi \|_{H}^2  }.
    \end{equation}
      \item[(iii)] For any $t_1 \geq t_0$, it satisfies the follows uniform estimates
    \begin{gather}
        \E \exp \Big\{ \delta \big\{ \| X^{\epsilon}_{t_1;t_0,\xi}\|_{H}^2 \big\} \Big\}  \leq 2 \, e^{\delta \big(\| \xi \|_{H}^2 + 2 \frac{ C_{A,\rho,\epsilon}  }{\lambda_1 \gamma_0} \big)} ; \label{est: exp H} \\
        \E \Big[ \exp \Big\{  \frac{\delta}{2} \sup_{t \geq t_1} \big( \frac{\gamma_0}{2} \int_{t_1}^t \| X^{\epsilon}_{s;t_0,\xi} \|_{V}^{2} \d s- C_{A,\rho,\epsilon}  (t-t_1) \big) \Big\} \Big]  \leq 2 e^{\delta  \big( \frac{\| \xi \|_{H}^2}{2} +\frac{ C_{A,\rho,\epsilon} }{\lambda_1 \gamma_0} \big)}. \label{est: exp V}
    \end{gather}
    \end{description}
 The above results imply $X^{\epsilon}_{t;t_0,\xi} \in L^{2}(\Omega; \sC_{t_0})$.
\end{proposition}

\section{Well-posedness of stationary solutions} \label{sec: stationary solution}

From Proposition \ref{Prop: measurable map truncated}, there exists a Borel measurable map $\mathcal{G}^{\epsilon}_{t_0}: \W \times H \rightarrow \sC_{t_0}$ such that $\mathcal{G}^{\epsilon}_{t_0}(W,\xi) (t,\omega)$ is the solution of Eq. \eqref{eq: SPDE} started from $t_0$ with initial data $\xi \in H$. Thus, for any fixed $t_0$, we can define the solution map $U_{t_0}^{\epsilon}$ from $\R_{+} \times H \times \Omega$ to $H$ as
\begin{equation}\label{U(t,theta)}
 U_{t_0}^{\epsilon}: (t,\xi,\omega) \mapsto \mathcal{G}^{\epsilon}_{t_0}(W,\xi)(t+t_0, \omega)=X^{\epsilon}_{t+t_0;t_0,\xi}(\omega) \in H.
\end{equation}

Next, we will introduce the concept of the crude cocycle, which is adapted slightly according to our needs, c.f. \cite{Arnold_Perfect_1995,Kager_Generation_1997,Mohammed_Stable_2008}.
The standard $\mathbb{P}$-preserving ergodic Wiener shift $\theta:\mathbb{R}\times \Omega\rightarrow \Omega$ is defined by $\theta(t,\omega(s)):=\omega(t+s)-\omega(s)$ for any $s,t\in\mathbb{R}$. Moreover, denote $\theta(t,\omega(0))$ as $\theta(t,\omega)$. Using the stationary increments of the Wiener process, it holds that for any $t_0 \in \R$, $\P$-a.s.
\begin{equation} \label{ident: trans}
 U_{t_0}^{\epsilon}\left(t,\xi, \omega\right)=X^{\epsilon}_{t+t_0;t_0,\xi}(\omega)=X^{\epsilon}_{t;0,\xi}(\theta(t_0,\omega))=U_0^{\epsilon}\left(t,\xi, \theta\left(t_0, \omega\right)\right), \quad \forall \, \xi \in H,  \, \forall \, {t \geq 0},
\end{equation}
see also in \cite[P111 (4.65)]{Liu_Stochastic_2015}. According to \cite[Definition 1.1.2]{Mohammed_Stable_2008}, a random dynamical system $(U^{\epsilon}_{t_0}, \theta)$ is called a crude cocycle if the following conditions are satisfied:
\begin{itemize}
 \item for any $s,t \geq 0$, $\mathbb{P}$-a.s., $U_{t_0}^{\epsilon} \left(t+s, \cdot, \omega\right)=U_{t_0}^{\epsilon} \left(t, U_{t_0}^{\epsilon} \left(s, \cdot, \omega\right), \theta\left(s, \omega\right)\right)$; 
 \item $U_{t_0}^{\epsilon} (0, \xi, \omega)=\xi$ for all $\xi \in H$, $\omega \in \Omega$.
\end{itemize}
Using the Markov property c.f. \cite[Section 4.3]{Liu_Stochastic_2015}, we have
\begin{align*}
  U_{t_0}^{\epsilon} \left(t+s, \cdot, \omega\right) & =X^{\epsilon}_{t+s+t_0;t_0,\xi}(\omega) \\
  & = X^{\epsilon}_{t+t_0;t_0,X^{\epsilon}_{s+t_0;t_0,\xi}(\omega)}(\theta\left(s, \omega\right))=U_{t_0}^{\epsilon} \left(t, U_{t_0}^{\epsilon} \left(s, \cdot, \omega\right), \theta\left(s, \omega\right)\right),
\end{align*}
for any $s,t \geq 0$ and $t_0 \in \R$. As a result, $(U_{t_0}^{\epsilon},\theta)$ is the crude cocycle defined as \cite[Defintion 1.1.2]{Mohammed_Stable_2008} for any $t_0 \in \R$. We call $(U^{\epsilon},\theta)$ is a crude cocycle, if the above equality \eqref{ident: trans} holds and for every $t_0 \in \R$ the random system $(U^{\epsilon}_{t_0},\theta)$ is a crude cocycle. This definition removes the dependence on the initial time, thereby enabling an extension to the entire real line.

Stationary solution is a fundamental concept in studying the long time behavior of stochastic dynamical systems (see \cite{Bakhtin_Existence_2003, Bakhtin_Stationary_2005, Ito_Stationary_1964, Mohammed_Stable_2008}). For crude cocycle $(U^{\epsilon},\theta)$, it is defined as follows.
\begin{definition} \label{def: stationary solution}
 An $\mathcal{F}_t$ progressively measurable process $\mathcal{Y}^{\epsilon}: \R \times \Omega \rightarrow H$ is said to be a stationary solution for the crude cocycle $(U^{\epsilon}, \theta)$ if for any fixed $s \in \R$, $\P$-a.s.
	\begin{equation} \label{identity: stationary-solutions}
		U^{\epsilon}_0(t, \mathcal{Y}(s,\omega), \theta(s, \omega))=\mathcal{Y}^{\epsilon}(t+s,\omega)=\mathcal{Y}^{\epsilon}(t,\theta(s, \omega)), \quad \forall t \, \geq 0.
\end{equation}
\end{definition}
\begin{remark} \label{remark: stationary solution}
  For any $s \in \R$, the equality \eqref{ident: trans} and the first equality in \eqref{identity: stationary-solutions} show
  \begin{equation*}
    U_{s}^{\epsilon}(t, \mathcal{Y}^{\epsilon}(s,\omega), \omega)=U_0^{\epsilon}(t, \mathcal{Y}^{\epsilon}(s,\omega), \theta(s, \omega))=\mathcal{Y}^{\epsilon}(t+s,\omega), \quad \forall \, t \geq 0,
  \end{equation*}
which means that $\mathcal{Y}^{\epsilon}(t,\omega)$ is the solution of Eq. \eqref{eq: SPDE} with initial data $\mathcal{Y}^{\epsilon}(s, \omega)$ and any starting time $s \in \R$. The second equality of \eqref{identity: stationary-solutions} shows that the value of stationary solution at different times differ only by a Wiener shift in the Wiener space.
\end{remark}

The following describes the well-posedness of the stationary solution of Eq. \eqref{eq: SPDE} for the case when $\gamma_0$ is sufficiently big. Set
\begin{align*}
  \tilde{\gamma}_0:=&\frac{C_{\rho_2}}{\lambda_1} \vee \big( \frac{(4+\beta) C_{\rho_1} C_{A,\rho}}{\lambda_1} \big)^{1/2},  \\
   \tilde{\epsilon}:=&\frac{ (1+\beta)\lambda_1 \gamma_0^2}{8 (2+\beta)^2  C_{\rho_1} C_{B}} \wedge \frac{\gamma_0}{(18 \gamma_0 +  \beta  (2+\beta)  C_{\rho_1} )C_B} \Big( 2 \lambda_1 \gamma_0-\frac{   (4+\beta) C_{\rho_1} C_{A,\rho}   }{\gamma_0} -C_{\rho_2}  \Big),
\end{align*}
where $C_{A, \rho}$ is given by condition {\rm [A5]}.

\begin{theorem} \label{thm: stationary solution}
Assume that the conditions {\rm [A1-A5]} hold. For any $\gamma_0 > \tilde{\gamma}_0$ and $0 \leq \epsilon< \tilde{\epsilon}$, there is a pathwise unique stationary solution $\cY^{\epsilon} \in  L^{\infty}(\R; L^{2}(\Omega; H))$ for cocycle $(U^{\epsilon},\theta)$. The stationary solution $\cY^{\epsilon}(t)$ satisfies
\begin{small}
  \begin{gather}
     \E \exp \Big\{ \delta \big\{ \| \cY^{\epsilon}(t) \|_{H}^2 \big\} \Big\}  \leq 2 \, e^{2 \delta \frac{C_{A,\rho,\epsilon}}{\lambda_1 \gamma_0}}; \label{est: exp H stationary} \\
     \E \Big[ \exp \Big\{  \frac{\delta}{2} \sup_{t \geq t_1} \big( \frac{\gamma_0}{2} \int_{t_1}^t \| \cY^{\epsilon}(s) \|_{V}^{2} \, \d s-C_{A,\rho,\epsilon}   \big)(t-t_1) \big) \Big\} \Big]  \leq    2 e^{\delta  \frac{ C_{A,\rho,\epsilon} }{\lambda_1 \gamma_0}}; \label{est: exp V stationary} \\
      \E \Big[ \exp \Big\{  \frac{\delta}{2} \sup_{t \geq t_1} \big( \frac{\gamma_0}{2} \int_{t_1}^t \| \cY^{\epsilon}(s) \|_{H}^{\beta} \| \cY^{\epsilon}(s) \|_{V}^{2} \d s-\big( C_{A, \rho}+ \beta C_B \epsilon \big)(t-t_1) \big) \Big\} \Big]  \leq   2  e^{\delta  \frac{2+\beta}{2} \big( \frac{ C_{A,\rho,\epsilon} + \beta C_{B} \epsilon  }{\lambda_1 \gamma_0} \big)}, \label{est: exp V stationary beta}
 \end{gather}
\end{small}
for any $t_0 \in \R$, where $\frac{ 4 C_{\rho_1}}{\gamma_0} \leq \delta<  \frac{ (1+\beta)\lambda_1 \gamma_{0}}{2(2+\beta)^2 C_{B} \epsilon} \leq \frac{ \lambda_1 \gamma_{0}}{8 C_{B} \epsilon}$. For any $-\infty<t_0<t<+\infty$, the stationary solution $\cY^{\epsilon} \in L^{2}(\Omega;\sC_{-\infty})$ $\P$-a.s. satisfies
\begin{equation} \label{eq: stationary solution}
    \scalV{\cY^{\epsilon}(t),v} = \scalV{\cY^{\epsilon}(t_0),v} + \int_{t_0}^{t} \scalV{A^{\epsilon}(\cY^{\epsilon}(s)),v} \d s + \sqrt{\epsilon} \int_{t_0}^{t} \scalH{B (\cY^{\epsilon}(s)) \, \d W_s,v},
\end{equation}
for any $v \in V$. It is called the ``very weak solution'' in \cite[Definition 3.4]{Brzezniak_Large_2017}.

Conversely, if a $\cF_{t}$-adapted process $\cY^{\epsilon}$ satisfies Eq. \eqref{eq: stationary solution} and $\sup_{t \in \R} \E \|\cY^{\epsilon}(t) \|_{H}^2< +\infty$, then it is indeed the stationary solution for cocycle $(U^{\epsilon},\theta)$. There exists a Borel measurable map $\mathcal{G}^{\epsilon}_{-\infty}: \W \rightarrow \sC_{-\infty}$ such that $\mathcal{G}^{\epsilon}_{-\infty}(W)(\cdot)$ is an $\mathcal{F}_{t}$-adapted solution to Eq. \eqref{eq: stationary solution} and $\sup_{t \in \R} \E \|\mathcal{G}^{\epsilon}_{-\infty}(W)(t) \|_{H}^2< +\infty$, called the strong stationary solution for cocycle $(U^{\epsilon},\theta)$.
\end{theorem}

\begin{proof}
    We split the proof into three steps.

    {\bf Step 1. Construct stationary solution.} To construct stationary solution, we first need to prove $\{X^{\epsilon}_{\cdot;-n,\xi}\}_{n \geq N}$ is a Cauchy sequence in $L^{2/p_0}(\Omega;\sC_{-N})$ for every $N \in \N$ for some $1<p_0 \leq 2$. Once we establish this, Proposition \ref{Prop: Cauchy sequence} shows that $\{X^{\epsilon}_{\cdot;-n,\xi}\}_{n}$ converges to a limit element $\mathcal{Y}^{\epsilon}$ with respect to $ |\E |d_{-\infty} (\cdot,\cdot)|^{2/p_0}|^{p_0/2} $. We will show $\mathcal{Y}^{\epsilon}$ is the stationary solution.

    Let $m>n>0$, for any $t \geq -n$, denote $\eta(t)=X^{\epsilon}_{t;-n,\xi}-X^{\epsilon}_{t;-m,\xi}$ and
\begin{align*}
      \beta_{\gamma}(t)  := & \gamma (t+n) - \frac{2 C_{\rho_1}}{\gamma_0} \Big( \frac{\gamma_0}{2} \int_{-n}^{t} \big(  1 +  \|X^{\epsilon}_{s;-n,\xi}\|_{H}^{\beta} \big) \|X^{\epsilon}_{s;-n,\xi}\|_{V}^{2} \big)  \, \d s \\
      &  - \frac{1}{2} \big((4+\beta) C_{A,\rho,\epsilon} + \beta (2+\beta) C_B \epsilon) \big)  (t+n) \Big).
    \end{align*}
    Using It\^o formula, we have
    \begin{align*}
        & e^{\beta_{\gamma}(t) } \| \eta(t) \|_{H}^2 - \| \eta(-n) \|_{H}^2 -2 \sqrt{\epsilon} \int_{-n}^{t} e^{\beta_{\gamma}(s)} \scalH{\eta(s), B(X^{\epsilon}_{s;-n,\xi})-B(X^{\epsilon}_{s;-m,\xi}) \, \d W_s} \\
        &  = \Big(\gamma+\frac{ C_{\rho_1} ( (4+\beta)  C_{A,\rho,\epsilon} + \beta (2+\beta) C_B \epsilon)  }{\gamma_0} +C_{\rho_2} \Big) \int_{-n}^{t}  e^{\beta_{\gamma}(s)} \|\eta(s)\|_{H}^{2} \, \d s \\
        & \quad -   \int_{-n}^{t}  e^{\beta_{\gamma}(s)} \Big( C_{\rho_1} \big(  1+ \|X^{\epsilon}_{s;-n,\xi}\|_{H}^{\beta} \big) \|X^{\epsilon}_{s;-n,\xi}\|_{V}^{2}  +C_{\rho_2} \Big) \|\eta(s)\|_H^2  \, \d s \\
        & \quad  + \int_{-n}^{t} e^{\beta_{\gamma}(s)} \big( 2 \, \scalV{A^{\epsilon}(X^{\epsilon}_{s;-n,\xi})-A^{\epsilon}(X^{\epsilon}_{s;-m,\xi}),\eta(s)}  + \epsilon \| B(X^{\epsilon}_{s;-n,\xi})-B(X^{\epsilon}_{s;-m,\xi})  \|_{2}^2  \big) \, \d s \\
        &  \leq  \int_{-n}^{t}  e^{\beta_{\gamma}(s)} \Big( \big(\gamma+\frac{ C_{\rho_1} ((4+\beta)   C_{A,\rho,\epsilon} + \beta (2+\beta) C_B \epsilon)  }{\gamma_0} +C_{\rho_2} \big) \|\eta(s)\|_{H}^{2} - 2 \gamma_0 \|\eta(s)\|_{V}^{2} \Big) \, \d s,
    \end{align*}
    where in the last step we used the local monotonicity condition {\rm [A2]}. By Burkh\"older--Davis--Gundy inequality, Cauchy--Schwarz inequality and condition {\rm [A4]}, we obtain
    \begin{align*}
        & 2 \sqrt{\epsilon} \, \E \sup_{s\in [-n,t]} \Big| \int_{-n}^{s}  e^{\beta_{\gamma}(r)} \scalH{\eta(r), B(X^{\epsilon}_{r;-n,\xi})-B(X^{\epsilon}_{r;-m,\xi}) \, \d W_{r}} \Big| \\
        & \leq 6 \sqrt{ \epsilon} \, \E \Big| \int_{-n}^{t}  e^{2 \beta_{\gamma}(s)} \sum_{k} \Big| \bigscalH{\eta(s), \big( B(X^{\epsilon}_{s;-n,\xi})-B(X^{\epsilon}_{s;-m,\xi}) \big) \, \mathcal{U}_{k}} \Big|^2 \, \d s  \Big|^{1/2} \\
        &  \leq \frac{1}{2} \E \sup_{s\in [-n,t]}  e^{\beta_{\gamma}(s)} \| \eta(s) \|_{H}^2  + 18 C_B \epsilon \, \E \int_{-n}^{t} e^{\beta_{\gamma}(s)} \| \eta(s) \|_{H}^2 \, \d s.
    \end{align*}
    Due to $\gamma_0 > \frac{C_{\rho_2}}{\lambda_1} \vee \big( \frac{(4+\beta) C_{\rho_1} C_{A,\rho,\epsilon}}{\lambda_1} \big)^{1/2}$ and $\epsilon <\frac{\gamma_0}{(18 \gamma_0 +  \beta  (2+\beta)  C_{\rho_1} )C_B} ( 2 \lambda_1 \gamma_0-\frac{  (4+\beta) C_{\rho_1}  C_{A,\rho,\epsilon}   }{\gamma_0} -C_{\rho_2} \big)$, we can choose $\gamma,\Delta_{\gamma}>0$ such that $\gamma + \lambda_1 \Delta_{\gamma} < 2 \lambda_1 \gamma_0 - \frac{C_{\rho_1} ( (4+\beta) C_{A,\rho,\epsilon} + \beta (2+\beta) C_B \epsilon))}{\gamma_0} -C_{\rho_2} - 18 C_B \epsilon$. Using assumption {\rm [A3]} and embedded inequality \eqref{eq: embedding}, we have
    \begin{equation}\label{eq: stationary diff weighting}
    \max\Big\{   \E \big[ \sup_{s\in [-n,t]}  e^{\beta_{\gamma}(s)} \| \eta(s) \|_{H}^2  \Big] , 2 \Delta_{\gamma}  \E \big[   \int_{-n}^{t}  e^{\beta_{\gamma}(s)}  \|\eta(s)\|_{V}^{2} \, \d s  \Big] \Big\} \leq 2 \E \| \eta(-n) \|_{H}^2.
    \end{equation}
    Due to $0 \leq \epsilon < \frac{ (1+\beta)\lambda_1 \gamma_0^2}{8 (2+\beta)^2  C_{\rho_1} C_{B}} $, we can take $\frac{ 4 C_{\rho_1}}{\gamma_0} \leq \delta<  \frac{ (1+\beta)\lambda_1 \gamma_{0}}{2(2+\beta)^2 C_{B} \epsilon}$ and $2 \geq p_0=\frac{4 C_{\rho_1}}{\gamma_0 \delta}+1> \frac{8 (2+\beta)^2 C_{\rho_1} C_{B} \epsilon }{(1+\beta) \lambda_1 \gamma_0^2}+1$, then $ \frac{1}{p_0-1} \frac{2 C_{\rho_1}}{\gamma_0} =\frac{\delta}{2}$. Cauchy-Schwarz inequality yields
    \begin{align*}
        & \E \big[\sup_{s\in [-n,t]} e^{\gamma (s+n)} \| \eta(s) \|_{H}^2 \big]^{\frac{1}{p_0}} \leq \big( \E \sup_{s\in [-n,t]} e^{\beta_{\gamma}(s)} \| \eta(s) \|_{H}^2 \big)^{\frac{1}{p_0}}  \\
        & \times \Big( \E \Big[ \exp \Big\{  \frac{1}{p_0-1} \frac{2  C_{\rho_1}}{\gamma_0}  \sup_{s \geq -n}  \big( \frac{\gamma_0}{2} \int_{-n}^{s} \|X^{\epsilon}_{r;-n,\xi}\|_{H}^{\beta} \|X^{\epsilon}_{r;-n,\xi}\|_{V}^{2}  \, \d r -(1+\frac{\beta}{2}) \big( C_{A, \rho}+ \beta C_B \epsilon \big) (s+n) \big) \Big\} \Big] \Big) ^{\frac{p_0-1}{p_0}}  \\
        & \times \Big( \E \Big[ \exp \Big\{  \frac{1}{p_0-1} \frac{2  C_{\rho_1}}{\gamma_0}  \sup_{s \geq -n}  \big( \frac{\gamma_0}{2} \int_{-n}^{s}  \|X^{\epsilon}_{r;-n,\xi}\|_{V}^{2}  \, \d r -  C_{A, \rho}  (s+n) \big) \Big\} \Big] \Big) ^{\frac{p_0-1}{p_0}}  \\
        & \quad \leq 2 (\E \| \eta(-n) \|_{H}^2 )^{\frac{1}{p_0}}  e^{\frac{ 2  C_{\rho_1}}{\gamma_0 p_0} \big( \| \xi \|_{H}^{2} + \| \xi \|_{H}^{2+\beta} + \big(2+\frac{\beta}{2} \big) \frac{ 2 C_{A,\rho,\epsilon} + \beta C_{B} \epsilon  }{\lambda_1 \gamma_0} \big)},
    \end{align*}
    where we have used estimates \eqref{est: exp V}, \eqref{est: exp V beta} and \eqref{eq: stationary diff weighting} in the last step. Then, we obtain
    \begin{align}
       &  \E \big[ \sup_{s\in [-n,t]} e^{-\gamma (t-s)} \| \eta(s) \|_{H}^2 \big]^{\frac{1}{p_0}} \nonumber \\
        & \quad \leq 2 \Big[  e^{ \frac{2 C_{\rho_1}}{\gamma_0} \big(\| \xi \|_{H}^{2}+\| \xi \|_{H}^{2+\beta} + \frac{4+\beta}{2}  \frac{2 C_{A,\rho,\epsilon}+\beta C_{B} \epsilon}{\lambda_1 \gamma_0} \big) - \gamma (t+n) } \E \| \eta(-n) \|_{H}^2  \Big]^{\frac{1}{p_0}}. \label{est: uniform diff H}
    \end{align}
    By estimate \eqref{est: exp H}, we know $\E \| \eta(-n) \|_{H}^2$ has a uniform upper bound. Thus, for any $N>0$,
    \begin{align*}
        & \limsup_{n \rightarrow +\infty }  \E \big[\sup_{s\in [-N,N]} e^{-\gamma (N-s)} \| X^{\epsilon}_{s;-n,\xi}-X^{\epsilon}_{s;-m,\xi} \|_{H}^2 \big]^{\frac{1}{p_0}} \\
        & \quad \leq \lim_{n \rightarrow +\infty}  \E \big[\sup_{s\in [-n,N]} e^{-\gamma (N-s)} \| X^{\epsilon}_{s;-n,\xi}-X^{\epsilon}_{s;-m,\xi} \|_{H}^2 \big]^{\frac{1}{p_0}} =0.
    \end{align*}
    Inequality $\big( \sum_{n=1}^{\infty} |a_{n}|^{p_0} \big)^{1/p_0} \leq  \sum_{n=1}^{\infty} |a_{n}|$ yields
    \begin{equation*}
       \E \, \big| d_{+\infty,\gamma}(X^{\epsilon}_{t;-n,\xi},X^{\epsilon}_{t;-m,\xi}) \big|^{2/p_0} \leq  \E \sum_{N=1}^{+\infty} \Big(\frac{ 1 \wedge \sup_{s\in [-N,N] } e^{-\gamma (N-s)} \| X^{\epsilon}_{s;-n,\xi}-X^{\epsilon}_{s;-m,\xi} \|_{H}^2 }{2^{N}} \Big)^{1/p_0}.
    \end{equation*}
    The above estimates yield that $\{X^{\epsilon}_{\cdot;-n,\xi}\}_{n}$ is a Cauchy sequence in $L^{2/p_0}(\Omega;\sC_{+\infty})$. Denote the limit element of Cauchy sequence as $\cY^{\epsilon}$. Due to $\{X^{\epsilon}_{\cdot;-n,\xi}\}_{n}$ has a uniform bound in $L^{2}(\Omega;\sC_{+\infty})$, we obtain $\cY^{\epsilon} \in L^{2}(\Omega;\sC_{+\infty})$. By Fatou's Lemma and estimate \eqref{est: exp H}, we have
    \begin{align*}
        \E \exp \Big\{ \delta \big\{ \| \cY^{\epsilon}(t) \|_{H}^{2} \big\} \Big\} \leq \E \liminf_{n \rightarrow +\infty} \exp \Big\{ \delta \big\{ \| X^{\epsilon}_{t;-n,\xi}\|_{H}^{2} \big\} \Big\}  \leq 2 \,  e^{\delta \big(\| \xi \|_{H}^{2} + 2 \frac{ C_{A, \rho}+ \beta C_B \epsilon   }{\lambda_1 \gamma_0} \big)} .
    \end{align*}
    Taking $\xi =0$, we obtain estimate \eqref{est: exp H stationary} and $\cY^{\epsilon} \in L^{\infty}(\R;L^2(\Omega;H))$.

    Analogous to the estimate of $\|\eta\|_{H}$, by estimate \eqref{est: exp V}, \eqref{est: exp V beta} and \eqref{eq: stationary diff weighting}, we have
    \begin{align*}
        & \E \Big[\int_{-n}^{t}  e^{{\gamma}(s+n)}  \|\eta(s)\|_{V}^{2} \, \d s \Big]^{\frac{1}{p_0}}  \leq \big( \E  \int_{-n}^{t}  e^{\beta_{\gamma}(s)}  \|\eta(s)\|_{V}^{2} \, \d s \big)^{\frac{1}{p_0}}  \\
        & \times \Big( \E \Big[ \exp \Big\{  \frac{1}{p_0-1} \frac{2 C_{\rho_1}}{\gamma_0}  \sup_{s \geq -n}  \big( \frac{\gamma_0}{2} \int_{-n}^{s} \|X^{\epsilon}_{r;-n,\xi}\|_{H}^{\beta} \|X^{\epsilon}_{r;-n,\xi}\|_{V}^{2}  \, \d r -(1+\frac{\beta}{2}) \big( C_{A, \rho}+ \beta C_B \epsilon \big) (s+n) \big) \Big\} \Big] \Big) ^{\frac{p_0-1}{p_0}}  \\
        & \times \Big( \E \Big[ \exp \Big\{  \frac{1}{p_0-1} \frac{2 C_{\rho_1}}{\gamma_0}  \sup_{s \geq -n}  \big( \frac{\gamma_0}{2} \int_{-n}^{s} \|X^{\epsilon}_{r;-n,\xi}\|_{V}^{2}  \, \d r -C_{A, \rho} (s+n) \big) \Big\} \Big] \Big) ^{\frac{p_0-1}{p_0}}  \\
        & \quad \leq 2 \Big( \frac{1}{2 \Delta_{\gamma}} \big(\E \| \eta(-n) \|_{H}^2 \big) \Big)^{\frac{1}{p_0}} e^{\frac{  C_{\rho_1}}{\gamma_0 p_0} \big( \| \xi \|_{H}^{2+\beta} + (2+\frac{\beta}{2}) \frac{ C_{A,\rho,\epsilon} + \beta C_{B} \epsilon  }{\lambda_1 \gamma_0} \big)}.
    \end{align*}
    For any $t_2>t_1 \geq -n$, dividing both sides by $e^{\gamma(t_1+n)/p_0}$ gives
    \begin{equation} \label{est: uniform diff V}
        \E \Big[\int_{t_1}^{t_2} \! \|\eta(s)\|_{V}^{2} \d s \Big]^{\frac{1}{p_0}} \leq 2 \Big( \frac{e^{-{\gamma}(t_1+n)}}{2 \Delta_{\gamma}} \E \| \eta(-n) \|_{H}^2  \Big)^{\frac{1}{p_0}} \! e^{\frac{ 2  C_{\rho_1}}{\gamma_0 p_0} \big( \|\xi\|_H^2 + \| \xi \|_{H}^{2+\beta} + \frac{4+\beta}{2} \frac{ 2 C_{A,\rho,\epsilon} + \beta C_{B} \epsilon  }{\lambda_1 \gamma_0} \big)}.
    \end{equation}
    So $\{X^{\epsilon}_{\cdot;-n,\xi}\}_{n}$ is a Cauchy sequence in $L^{2/p_0}(\Omega; L^2([t_1,t_2];V))$ and the limit element also is $\cY^{\epsilon}$. Thus, $\{X^{\epsilon}_{\cdot;-n,\xi}\}_{n \geq N}$ is a Cauchy sequence in $L^{2/p_0}(\Omega;\sC_{-N})$ for every $N \in \N$. And $\{X^{\epsilon}_{\cdot;-n,\xi}\}_{n \geq N}$ converges to $\cY^{\epsilon} \in \sC_{-\infty}$ with respect to distance $ |\E |d_{-\infty} (\cdot,\cdot)|^{2/p_0}|^{p_0/2} $ as $n \rightarrow +\infty$. By Fatou's Lemma and estimate \eqref{est: exp V}, we have
    \begin{align*}
        & \E \Big[ \exp \Big\{  \frac{\delta}{2} \sup_{t \geq t_1} \big( \frac{\gamma_0}{2} \int_{t_1}^t \| \cY^{\epsilon}(s) \|_{V}^{2} \, \d s-C_{A,\rho,\epsilon} (t-t_1) \big) \Big\} \Big]  \nonumber \\
        & \leq  \E \Big[ \lim_{n \rightarrow +\infty} \exp \Big\{  \frac{\delta}{2} \sup_{t \geq t_1} \big( \frac{\gamma_0}{2} \int_{t_1}^t \| X^{\epsilon}_{s;-n,\xi} \|_{V}^{2} \, \d s-C_{A,\rho,\epsilon}  (t-t_1) \big) \Big\} \Big] \leq    2 e^{\delta  \big( \frac{\| \xi \|_{H}^2}{2} +\frac{ C_{A,\rho,\epsilon} }{\lambda_1 \gamma_0} \big)}.
    \end{align*}
    Taking $\xi =0$, we obtain estimate \eqref{est: exp V stationary} and $\cY^{\epsilon} \in L^{2}(\Omega;\sC_{-\infty})$. By the same way,  Fatou's Lemma and estimate \eqref{est: exp V beta} yield estimate \eqref{est: exp V stationary beta}.

    Next, we show $\cY^{\epsilon}(t)$ is the stationary solution of Eq. \eqref{eq: SPDE}. For $t_1 \geq -n$, we have $\P$-a.s.
    \begin{equation*}
    \scalV{ X^{\epsilon}_{t+t_1;-n,\xi}, v} = \scalV{X^{\epsilon}_{t_1;-n,\xi}, v} + \int_{t_1}^{t+t_1} \scalV{A^{\epsilon}(X^{\epsilon}_{s;-n,\xi}), v} \, \d s + \sqrt{\epsilon} \int_{t_1}^{t+t_1} \scalH{B(X^{\epsilon}_{s;-n,\xi}) \, \d W_s, v}
    \end{equation*}
    for any $ t \geq 0$ and for any $v \in V$. By the convergence of $\{X^{\epsilon}_{\cdot;-n,\xi}\}_{n \geq N}$ in $L^2([t_1,t_2];V)$ and assumption {\rm [A3]}, we obtain that $\{ A^{\epsilon}(X^{\epsilon}_{\cdot;-n,\xi}) \}_{n}$ and $\{ B(X^{\epsilon}_{\cdot;-n,\xi}) \}_{n}$ are Cauchy sequences in $L^{2/p_0}(\Omega; L^2([t_1,t_2];V^{\ast}))$ and $L^{2/p_0}(\Omega; L^2([t_1,t_2];L_{2}(U; H)))$, respectively, with limits $A^{\epsilon}(\cY^{\epsilon})$ and $B(\cY^{\epsilon})$. Thus, taking $n \rightarrow +\infty$, we obtain that $\cY^{\epsilon}(t)$ $\P$-a.s. satisfies
    \begin{equation*}
    \scalV{\cY^{\epsilon}(t+t_1),v} = \scalV{\cY^{\epsilon}(t_1),v}+ \int_{t_1}^{t+t_1} \scalV{A^{\epsilon}(\cY^{\epsilon}(s)),v} \, \d s + \sqrt{\epsilon} \int_{t_1}^{t+t_1} \scalH{B(\cY^{\epsilon}(s)) \, \d W_s,v}
    \end{equation*}
    for any $ t \geq 0$ and for any $v \in V$. Then $\cY^{\epsilon}(t+t_1)$ is a solution of Eq. \eqref{eq: SPDE} starting from $t_1$ with initial $\cY^{\epsilon}(t_1)$ called ``very weak solution''. So $\P$-a.s. it holds
    \begin{equation*}
        U_0^{\epsilon}(t, \mathcal{Y}^{\epsilon}(t_1,\omega), \theta(t_1, \omega))=\mathcal{Y}^{\epsilon}(t+t_1,\omega), \quad \forall \, t \geq 0.
    \end{equation*}
    On the other hand, by the definition of the Wiener shift $\theta$, we have $X^{\epsilon}_{t;-n,\xi}( \theta(t_1,\omega))=X^{\epsilon}_{t+t_1;-n+t_1,\xi}( \omega)$. Taking $n \rightarrow +\infty$, then $\mathcal{Y}^{\epsilon}(t+t_1,\omega) =\mathcal{Y}^{\epsilon}(t,\theta(t_1,\omega))$ $\P$-a.s. holds. Combining the above results, we proved that $\cY^{\epsilon}$ is the stationary solution for crude cocycle $(U^{\epsilon}, \theta)$.

    {\bf Step 2. Pathwise uniqueness.}  If $\Tilde{\cY^{\epsilon}} \in L^{\infty}(\R; L^{2}(\Omega;H))$ is another stationary solution for the crude cocycle $(U^{\epsilon},\theta)$, then for any $t_0 \in \R$, it $\P$-a.s. satisfies
    \begin{equation*}
        \tilde{\cY^{\epsilon}}(t) = \tilde{\cY^{\epsilon}}(t_0)+ \int_{t_0}^{t} A^{\epsilon}(\tilde{\cY^{\epsilon}}(s)) \, \d s + \sqrt{\epsilon} \int_{t_0}^{t} B(\tilde{\cY^{\epsilon}}(s)) \, \d \tilde{W}_s, \quad \forall \, t \geq 0.
    \end{equation*}
    Using \eqref{est: exp V stationary} and analogous to estimate \eqref{est: uniform diff H}, we have
    \begin{align}\notag
        & \E \big[  \sup_{s\in [-n,t]} e^{-\gamma (t-s)} \| \cY^{\epsilon}(s)- \tilde{\cY^{\epsilon}}(s) \|_{H}^2 \big]^{\frac{1}{p_0}} \\
        \label{rrr-5}
        & \leq  2 \Big(  e^{ \frac{2  C_{\rho_1}}{\gamma_0} \big(\| \xi \|_{H}^2 + \| \xi \|_{H}^{2+\beta} +\frac{4+\beta}{2}  \frac{2 C_{A,\rho,\epsilon}+\beta C_{B} \epsilon}{\lambda_1 \gamma_0} \big) - \gamma (t+n) } \E \|  \cY^{\epsilon}(-n)- \tilde{\cY^{\epsilon}}(-n) \|_{H}^2  \Big)^{\frac{1}{p_0}}.
    \end{align}
    Due to $\cY^{\epsilon}, \Tilde{\cY^{\epsilon}} \in L^{\infty}(\R; L^{2}(\Omega;H))$, taking $n \rightarrow +\infty$, we obtain that $\cY^{\epsilon}=\Tilde{\cY^{\epsilon}}$ in the Polish space $L^{2/p_0}(\Omega,\sC_{+\infty})$. Using \eqref{est: exp V stationary} and analogous to estimate \eqref{est: uniform diff V}, we obtain that
    \begin{align}\notag
       & \E \Big[\int_{t_1}^{t_2} \! \!  \|\cY^{\epsilon}(s)- \tilde{\cY^{\epsilon}}(s)\|_{V}^{2} \, \d s \Big]^{\frac{1}{p_0}} \\
       \label{rrr-6}
       & \leq  2 \Big( \frac{e^{-{\gamma}(t_1+n)}}{2 \Delta_{\gamma}} \big(\E \| \cY^{\epsilon}(-n)- \tilde{\cY^{\epsilon}}(-n)\|_{H}^2 \big) \Big)^{\frac{1}{p_0}} e^{\frac{ 2  C_{\rho_1}}{\gamma_0 p_0} \big( \| \xi \|_{H}^2 + \| \xi \|_{H}^{2+\beta} + \frac{4+\beta}{2} \frac{ 2 C_{A,\rho,\epsilon} + \beta C_{B} \epsilon  }{\lambda_1 \gamma_0} \big)},
    \end{align}
    and $\cY^{\epsilon}=\Tilde{\cY^{\epsilon}}$ in $L^{2/p_0}(\Omega,L^2([t_1,t_2];V))$ for any $-\infty<t_1 \leq t_2<+\infty$. Thus, we get the pathwise uniqueness of the stationary solution in $L^{2/p_0}(\Omega,\sC_{-\infty})$. Note that to prove the pathwise uniqueness, we only used $\sup_{t \in \R} \E\|\Tilde{\cY^{\epsilon}}(t)\|_{H}^2<+\infty$ and $\tilde{\cY^{\epsilon}}$ satisfying Eq. \eqref{eq: stationary solution}.

    {\bf Step 3: Strong stationary solution.} Eq. \eqref{eq: stationary solution} gives a constraint $\Gamma^{\epsilon}$ that determines $S_{\Gamma^{\epsilon}} \subset \mathcal{P}( \W \times \sC_{-\infty})$. We say that $\varrho \in S_{\Gamma^{\epsilon}}$ with constraint $\Gamma^{\epsilon}$, if there exists $(w, X) \in \W \times \sC_{-\infty}$ such that $\varrho= \P \circ (w,X)^{-1}$ with marginal distribution $\P \circ (W)^{-1} $ on $\W $ and $(X,w)$ satisfies
    \begin{align*}
        \lim_{n \rightarrow + \infty} \E \bigg[ & _{V^\ast}\bigg{\<} X(t) -X(t_0) - \int_{t_0}^{t} A \big(X(s) \big) \, \d s \bigg. \bigg. \\
    \bigg. \bigg. & -\sqrt{\epsilon} \sum_{k=0}^{\lfloor n (t-t_0) \rfloor} B \Big( X \Big(t_0+ \frac{k}{n} \Big) \Big) \cdot \Big( w \Big( t \wedge \big(t_0+\frac{k+1}{n} \big) \Big)-w \Big(t_0+ \frac{k}{n} \Big) \Big), v \bigg\>_{V} \bigg] =0,
    \end{align*}
    for any $v \in V$ and $\sup_{t \in \R} \E\|X(t)\|_{H}^2<+\infty $. It is clear that $S_{\Gamma^{\epsilon}}$ is convex. By the pathwise uniqueness and the general Yamada--Watanabe Theorem (Lemma \ref{lem: Yamada-Watanabe}), there is a Borel measurable $\mathcal{G}^{\epsilon}_{-\infty}: \W \rightarrow \sC_{-\infty}$ such that $\mathcal{G}^{\epsilon}_{-\infty}(W)(\cdot)$ is an $\mathcal{F}_{t}$-adapted process satisfying Eq. \eqref{eq: stationary solution} and $\sup_{t \in \R} \E \|\mathcal{G}^{\epsilon}_{-\infty}(W)(t) \|_{H}^2< +\infty$. Thus, $\mathcal{G}^{\epsilon}_{-\infty}(W)$ is the stationary solution for $(U^{\epsilon}, \theta)$.
\end{proof}

\section{The skeleton equation} \label{sec: skeleton eq}
In this section, we will make a finer analysis of the skeleton equation related to Eq. \eqref{eq: SPDE}.

Let $\mathcal{H}$ be the Cameron-Martin space, defined by
\begin{align*}
\mathcal{H}:=&\Big\{\phi: \phi\ {\rm{is\ a\ }}\  U\text{-}{\rm{valued}}\ \{\mathcal{F}_t\}\text{-}{\rm{predictable\ process\ such\ that}} \ \int^{+\infty}_{-\infty} |\phi(s)|^2_U \, \d s<+\infty\ \mathbb{P}\text{-}a.s.\Big\}.
\end{align*}
We also need the sets
\begin{align*}
S_M:=&\{ v\in L^2(\mathbb{R};U): \int^{+\infty}_{-\infty} |v(s)|^2_U \, \d s\leq M\};\\
\mathcal{A}_M:=&\{\phi\in \mathcal{H}: \phi(\omega)\in S_M,\ \mathbb{P}\text{-}a.s.\}.
\end{align*}
Here and in the sequel of this paper, we will always refer to the weak topology on the set $S_M$.

For any $v\in \mathcal{H}$, we set the map $\int^{\cdot}_{0}v(s) \,  \d s: \mathbb{R}\rightarrow U$ as follows:
\begin{equation*}
  t\in \mathbb{R}\mapsto \begin{cases}
    \int^t_0v(s) \, \d s, & {\rm{when}}\ t\geq 0; \\
    -\int^0_tv(s) \, \d s, & {\rm{when}}\ t<0.
  \end{cases}
\end{equation*}

Consider the skeleton equation related to \eqref{eq: SPDE},
\begin{align}\label{eq: skeleton}
  X^{0,v}_{t;t_0,\xi}=\begin{cases}
    \xi+\int^t_{t_0}A^{0}(X^{0,v}_{s;t_0,\xi}) \,  \d s+\int^t_{t_0}B(X^{0,v}_{s;t_0,\xi}) v(s) \,  \d s, & {\rm{if}}\ t\geq t_0; \\
    \xi, & {\rm{if}}\ t<t_0,
  \end{cases}
\end{align}
where $t_0\in \mathbb{R}, \xi\in H$ and $v\in L^2(\mathbb{R};U)$. Pan et al. \cite{Pan_Large_2026} shows that the skeleton equation \eqref{eq: skeleton} admits a unique solution $X^{0,v}_{t;t_0,\xi}$ in $C([t_0,T];H)\cap L^2([t_0,T];V)$. It is worth noting that the conditions {\rm [A1-A5]} in this paper are stronger than those given by Pan et al. in \cite[Section 4]{Pan_Large_2026}, therefore consequently ensuring the well-posedness of the skeleton equation \eqref{eq: skeleton}.

In the following, we will show that the solutions $\{X^{0,v}_{\cdot;-n,\xi}\}_{n\in \mathbb{N}}$ converge to $\mathcal{Y}^{0,v}\in \sC_{-\infty}$ as $n \rightarrow +\infty$ with respect to metric $d_{-\infty}(\cdot,\cdot)$. Moreover, the limit element $\mathcal{Y}^{0,v}$ satisfies
\begin{equation}\label{eq: skeleton stationary}
  \mathcal{Y}^{0,v}(t_2)=\mathcal{Y}^{0,v}(t_1) +\int^{t_2}_{t_1}A^{0}(\mathcal{Y}^{0,v}(s)) \,  \d s +\int^{t_2}_{t_1}B(\mathcal{Y}^{0,v}(s)) v(s) \,  \d s, \quad \forall \, t_2 > t_1.
\end{equation}
These results are formulated in Lemma \ref{lem: skeleton}.

To describe LDP, we need to define rate functions.
For each $v\in L^2(\mathbb{R};U)$, we denote $\mathcal{G}^{0}_{t_0,\xi}\Big(\int^{\cdot}_0v(s) \,  \d s\Big) :=X^{0,v}_{\cdot;t_0,\xi}$ and $\mathcal{G}^{0}_{-\infty}\Big(\int^{\cdot}_0v(s) \,  \d s\Big) :=\mathcal{Y}^{0,v}(\cdot)$. For each $\Phi\in \sC_{-\infty}$ and $\xi \in H$, define the rate functions as
\begin{align}
  \mathcal{I}_{t_0,\xi}(\Phi):=\inf_{\Big \{v\in L^2(\mathbb{R};U); \, \Phi=\mathcal{G}^{0}_{t_0,\xi}\Big(\int^{\cdot}_0 v(s) \,  \d s\Big)\Big\}} \big\{\frac{1}{2}\int^{+\infty}_{-\infty} |v(s)|^2_U \, \d s \big\}, \label{def: rate func 1} \\
 \mathcal{I}_{-\infty}(\Phi):=\inf_{\Big\{v\in L^2(\mathbb{R};U); \, \Phi=\mathcal{G}^{0}_{-\infty}\Big(\int^{\cdot}_0v(s) \,  \d s\Big) \Big\}} \big\{\frac{1}{2}\int^{+\infty}_{-\infty} |v(s)|^2_U \, \d  s \big\}.  \label{def: rate func 2}
\end{align}
where the infimum over the empty set is taken to be $+\infty$. In Lemma \ref{lem: good rate func}, we will show that $\mathcal{I}_{t_0,\xi}$ and $\mathcal{I}_{-\infty}$
are good rate functions.

\begin{lemma}\label{lem: skeleton}
 For any $t_0\in \mathbb{R}$, $\xi\in H$ and $v\in S_M$, let $X^{0,v}_{t;t_0,\xi}$ be the unique solution to Eq. \eqref{eq: skeleton}, if $\gamma_0 \geq  \tilde{\gamma}_0$, then the following results hold.
 \begin{description}
   \item[(i)] For $\gamma\in [0,\lambda_1 \gamma_0/4]$ and $t\geq t_0$, it holds that
   \begin{align}
  & \sup_{s\in [t_0,t]} \big\{ e^{-\gamma(t-s)} \|X^{0,v}_{s;t_0,\xi}\|^{2+\beta}_H+ \frac{\gamma_0}{2} \int^s_{t_0} e^{-\gamma(t-\tau)} \|X^{0,v}_{\tau;t_0,\xi}\|^{\beta}_H \|X^{0,v}_{\tau;t_0,\xi}\|^2_V \, \d \tau \big\} \nonumber   \\
  & \quad \leq e^{-\gamma(t-t_0)} \|\xi\|^{2+\beta}_H+ \big(1+\frac{\beta}{2} \big) \big(\frac{4 C_B}{\lambda_1 \gamma_0}M + \min \big\{ \frac{1}{\gamma} ,t-t_0 \big\} C_{A,\rho,0} \big); \label{eq: energy estimate skeleton beta}  \\
  & \sup_{s\in [t_0,t]} \big\{ e^{-\gamma(t-s)} \|X^{0,v}_{s;t_0,\xi}\|^{2}_H+ \frac{\gamma_0}{2} \int^s_{t_0} e^{-\gamma(t-\tau)} \|X^{0,v}_{\tau;t_0,\xi}\|^2_V \, \d \tau \big\} \nonumber \\
  & \quad  \leq e^{-\gamma(t-t_0)} \|\xi\|^2_H+ \big(\frac{4 C_B}{\lambda_1 \gamma_0}M + \min \big\{ \frac{1}{\gamma} ,t-t_0 \big\} C_{A,\rho,0} \big). \label{eq: energy estimate skeleton}
\end{align}
Taking $s=t$, it implies that $\|X^{0,v}_{t;t_0,\xi}\|_H^2$ is uniformly bounded over $[t_0,+\infty)$. The upper bound is less than $C(\xi,M):= \|\xi\|^2_H+ \frac{4 }{\lambda_1 \gamma_0} \big(C_B M +C_{A,\rho,0} \big)$.
   \item[(ii)] For any $\xi_1,\xi_2\in H$, $\gamma \in (0,(\lambda_1 \gamma_0- C_{\rho_2}) \vee \frac{\lambda_1 \gamma_0}{4}]$ and $t_2 \geq t_1 \geq t_0$, we have
\begin{align}
  & \sup_{s\in [t_1,t_2]} \big\{ \|X^{0,v}_{s;t_0,\xi_1}-X^{0,v}_{s;t_0,\xi_2}\|^2_H +\frac{\gamma_0-\tilde{\gamma}_0}{2} \int^s_{t_0} \|X^{0,v}_{\tau;t_0,\xi_1}-X^{0,v}_{\tau;t_0,\xi_2}\|^2_V \, \d \tau \big\} \nonumber \\
  & \leq  \|\xi_1-\xi_2\|^2_H \exp\Big\{ \frac{2 C_B}{\lambda_1 (\gamma_0-\tilde{\gamma}_0)}M+\frac{ 2 C_{\rho_1}}{\gamma_0}\big( \|\xi_2\|^{2}_H+ \|\xi_2\|^{2+\beta}_H + \frac{2(4+\beta) C_B}{\lambda_1 \gamma_0}M ) \Big\}  e^{-\gamma(t_1-t_0)}. \label{eq: energy estimate skeleton diff}
\end{align}
It implies the continuity of solution with respect to initial value.
   \item[(iii)] For every $N \in \N$, $\{X^{0,v}_{\cdot;-n,\xi}\}_{n \geq N}$ is a Cauchy sequence in $\sC_{-N}$ that converges to the same limit element $\mathcal{Y}^{0,v} \in \sC_{-\infty}$. The limit element $\mathcal{Y}^{0,v}$ is the unique solution of the controlled Eq. \eqref{eq: skeleton stationary}. And $\sup_{t \in \R} \| \mathcal{Y}^{0,v}(t) \|_{H}^2$ is uniformly bounded by $C(0,M)$ over all $v \in S_{M}$.
\end{description}

\end{lemma}
\begin{proof}
Let $\beta \geq 0$ be the same as $\beta$ appearing in condition {\rm [A2]}. For any $t\geq t_0$, and $s\in [t_0,t]$,
\begin{align*}
  e^{\gamma(s-t_0)}\|X^{0,v}_{s;t_0,\xi}\|^{2+\beta}_H & =\|\xi\|^{2+\beta}_H +\gamma\int^s_{t_0}e^{\gamma(\tau-t_0)}\|X^{0,v}_{\tau;t_0,\xi}\|^{2+\beta}_H \, \d \tau \\
 & \quad  + (2+\beta) \int^s_{t_0} e^{\gamma(\tau-t_0)} \|X^{0,v}_{\tau;t_0,\xi}\|^{\beta}_H  \, \scalV{ A^{0}(X^{0,v}_{\tau;t_0,\xi}),X^{0,v}_{\tau;t_0,\xi}} \, \d \tau\\
 & \quad + (2+\beta) \int^s_{t_0}e^{\gamma(\tau-t_0)}  \|X^{0,v}_{\tau;t_0,\xi}\|^{\beta}_H \, \scalH{B(X^{0,v}_{\tau;t_0,\xi})v,X^{0,v}_{\tau;t_0,\xi}} \, \d \tau.
\end{align*}
Using condition {\rm [A4]} and mean value theorem for integrals, we have
\begin{align*}
 & (2+\beta) \Big|\int^s_{t_0} e^{\gamma(\tau-t_0)} \|X^{0,v}_{\tau;t_0,\xi}\|^{\beta}_H \scalH{ B(X^{0,v}_{\tau;t_0,\xi})v,X^{0,v}_{\tau;t_0,\xi}} \, \d \tau \Big| \\
  &= (2+\beta) \Big|\int^s_{t_0}e^{\gamma(\tau-t_0)} \|X^{0,v}_{\tau;t_0,\xi}\|^{\beta}_H \sum_{k} \scalH{ B(X^{0,v}_{\tau;t_0,\xi}) \, \mathcal{U}_{k},X^{0,v}_{\tau;t_0,\xi}} \scalU{v,\mathcal{U}_{k}} \, \d \tau \Big| \\
 & \leq (\beta +2) \Big(\int^s_{t_0}e^{\gamma(\tau-t_0)} \|X^{0,v}_{\tau;t_0,\xi}\|^{2\beta}_H \sum_{k} \scalH{ B(X^{0,v}_{\tau;t_0,\xi})\mathcal{U}_{k},X^{0,v}_{\tau;t_0,\xi}}^2  \, \d \tau\Big)^{1/2}\Big(\int^s_{t_0}e^{\gamma(\tau-t_0)}\|v\|^2_U \, \d \tau\Big)^{1/2} \\
 & \leq (1+\frac{\beta}{2}) \big( \frac{\lambda_1 \gamma_0}{4}\int^s_{t_0}e^{\gamma(\tau-t_0)}\|X^{0,v}_{\tau;t_0,\xi}\|^{2+\beta}_H \, \d \tau+ \frac{4 C_B}{\lambda_1 \gamma_0} e^{\gamma(s-t_0)} M \big).
\end{align*}
Using the coercivity condition {\rm [A5]} and combining the above estimates, we obtain
\begin{align*}
  & e^{\gamma(s-t_0)}\|X^{0,v}_{s;t_0,\xi}\|^{2+\beta}_H  \\
  & \leq \|\xi\|^{2+\beta}_H +\gamma\int^s_{t_0}e^{\gamma(\tau-t_0)} \|X^{0,v}_{\tau;t_0,\xi}\|^{2+\beta}_H \, \d \tau \\
  & \quad -\frac{(2+\beta)\gamma_0}{2}\int^s_{t_0}e^{\gamma(\tau-t_0)}  \|X^{0,v}_{\tau;t_0,\xi}\|^{\beta}_H \|X^{0,v}_{\tau;t_0,\xi}\|^2_V \, \d \tau + (1+\frac{\beta}{2}) e^{\gamma(s-t_0)} \min \big\{ \frac{1}{\gamma} ,s-t_0 \big\} C_{A,\rho,0}\\
  &\quad + (1+\frac{\beta}{2}) \frac{\lambda_1 \gamma_0}{4}\int^s_{t_0}e^{\gamma(\tau-t_0)}\|X^{0,v}_{\tau;t_0,\xi}\|^{2+\beta}_H \, \d \tau+ (1+\frac{\beta}{2}) \frac{4 C_B}{\lambda_1 \gamma_0} e^{\gamma(s-t_0)} M .
\end{align*}
By choosing $0 \leq \gamma \leq \frac{2 + 3 \beta}{8} \lambda_1 \gamma_0$, it gives
\begin{align*}
  & \sup_{s\in [t_0,t]} \big\{ e^{\gamma(s-t_0)} \|X^{0,v}_{s;t_0,\xi}\|^{2+\beta}_H+ \frac{\gamma_0}{2} \int^s_{t_0} e^{\gamma(\tau-t_0)} \|X^{0,v}_{\tau;t_0,\xi}\|^{\beta}_H \|X^{0,v}_{\tau;t_0,\xi}\|^2_V \, \d \tau \big\} \\
  & \quad \leq \|\xi\|^{2+\beta}_H+ \big(1+\frac{\beta}{2} \big) \big(\frac{4 C_B}{\lambda_1 \gamma_0}M +\min \big\{ \frac{1}{\gamma} ,t-t_0 \big\} C_{A,\rho,0} \big) e^{\gamma(t-t_0)}.
\end{align*}
Dividing $e^{\gamma(t-t_0)}$ in the both sides, we get \eqref{eq: energy estimate skeleton beta}. When $\beta=0$ and $0 \leq \gamma \leq \frac{\lambda_1 \gamma_0}{4} $, the similar estimate \eqref{eq: energy estimate skeleton} also holds by the same way. Taking $s=t$ and $\gamma= \frac{\lambda_1 \gamma_0}{4} $, we obtain that $\|X^{0,v}_{t;t_0,\xi}\|^{2}_H$ is uniformly bounded over $[t_0,+\infty)$ and is less than $C (\xi,M):= \|\xi\|^2_H+ \frac{4 }{\lambda_1 \gamma_0} \big(C_B M +C_{A,\rho,0} \big)$.

For fixed $t_0\in \mathbb{R}$, $t\geq t_0$ and $s\in [t_0,t]$, for $\xi_1,\xi_2\in H$, let $\eta(t)=X^{0,v}_{t;t_0,\xi_1}-X^{0,v}_{t;t_0,\xi_2}$. Set
\begin{equation*}
  B_{\gamma}(s)= \gamma (s-t_0)-\frac{2 C_{\rho_1}}{\gamma_0} \Big( \frac{\gamma_0}{2} \int^s_{t_0} (1+\|X^{0,v}_{\tau;t_0,\xi_2}\|^{\beta}_H) \|X^{0,v}_{\tau;t_0,\xi_2}\|^2_V \, \d \tau - \frac{(4+\beta)C_{A,\rho,\epsilon}}{2} (s-t_0) \Big).
\end{equation*}
By the chain rule, we get
\begin{align*}
 & e^{B_{\gamma}(s)}\|\eta(s)\|^2_H -\|\xi_1-\xi_2\|^2_H \\
 & = \big( \gamma+\frac{(4+\beta)C_{\rho_1}C_{A,\rho,\epsilon}}{\lambda_1 \gamma_0}  \big) \int^s_{t_0} e^{B_{\gamma}(\tau)}\|\eta(\tau)\|^2_H \, \d \tau -C_{\rho_1}\int^s_{t_0}e^{B_{\gamma}(\tau)} (1+\|X^{0,v}_{\tau;t_0,\xi_2}\|^{\beta}_H) \|X^{0,v}_{\tau;t_0,\xi_2}\|^2_V \|\eta(\tau)\|^2_H \, \d \tau \\
& \quad +\int^s_{t_0}e^{B_{\gamma}(\tau)} \big(2 \scalV{ A^{0}(X^{0,v}_{\tau;t_0,\xi_1})-A^{0}(X^{0,v}_{\tau;t_0,\xi_2}),\eta(\tau)}  +2 \scalH{ (B(X^{0,v}_{\tau;t_0,\xi_1})-B(X^{0,v}_{\tau;t_0,\xi_2}))v,\eta(\tau) } \big) \, \d \tau.
\end{align*}
The monotonicity condition {\rm [A2]} gives
\begin{align*}
 &\int^s_{t_0}e^{B_{\gamma}(\tau)}2 \scalV{ A^{0}(X^{0,v}_{\tau;t_0,\xi_1})-A^{0}(X^{0,v}_{\tau;t_0,\xi_2}),\eta(\tau)} \, \d \tau\\
& \leq \int^s_{t_0}e^{B_{\gamma}(\tau)}\Big[-2\gamma_0\|\eta(\tau)\|^2_V+  \Big(C_{\rho_1} \big( 1+ \|X^{0,v}_{\tau;t_0,\xi_2}\|^{\beta}_H \big) \|X^{0,v}_{\tau;t_0,\xi_2}\|^2_V +C_{\rho_2} \Big) \|\eta(\tau)\|^2_H\Big] \, \d \tau.
\end{align*}
The condition {\rm [A4]} and Cauchy--Schwarz inequality give
\begin{align*}
  &2\int^s_{t_0}e^{B_{\gamma}(\tau)}\scalH{ (B(X^{0,v}_{\tau;t_0,\xi_1})-B(X^{0,v}_{\tau;t_0,\xi_2}))v,\eta(\tau) } \, \d \tau\\
&  =2 \int^s_{t_0}e^{B_{\gamma}(\tau)} \sum_{k} \scalH{ (B(X^{0,v}_{\tau;t_0,\xi_1})-B(X^{0,v}_{\tau;t_0,\xi_2})) \, \mathcal{U}_{k},\eta(\tau)} \scalU{v,\mathcal{U}_{k}} \, \d \tau\\
& \leq 2 \int^s_{t_0}e^{B_{\gamma}(\tau)} \big(\sum_{k} \big| \scalH{ (B(X^{0,v}_{\tau;t_0,\xi_1})-B(X^{0,v}_{\tau;t_0,\xi_2})) \, \mathcal{U}_{k},\eta(\tau)} \big|^2 \big)^{\frac{1}{2}}\|v(\tau)\|_U  \, \d \tau \\
& \leq \frac{\lambda_1 (\gamma_0-\tilde{\gamma}_0)}{2}\int^s_{t_0}e^{B_{\gamma}(\tau)}\| \eta(\tau)\|^2_H \, \d \tau+ \frac{2 C_B}{\lambda_1 (\gamma_0-\tilde{\gamma}_0)} \int^s_{t_0}e^{B_{\gamma}(\tau)}\| \eta(\tau)\|^2_H\|v(\tau)\|^2_U \, \d \tau,
\end{align*}
where $\tilde{\gamma}_0$ defined as \eqref{def: gamma}. Combining the above estimates, we obtain
\begin{align*}
 & e^{B_{\gamma}(s)}\|\eta(s)\|^2_H + \frac{\gamma_0-\tilde{\gamma}_0}{2} \int^s_{t_0}e^{B_{\gamma}(\tau)}\|\eta(\tau)\|^2_V \, \d \tau  -\|\xi_1-\xi_2\|^2_H \\
 & \leq \big( \frac{\gamma + C_{\rho_2}}{\lambda_1}+\frac{(4+\beta)C_{\rho_1}C_{A,\rho,\epsilon}}{\lambda_1 \gamma_0} -\tilde{\gamma}_0- \gamma_0  \big)\int^s_{t_0} e^{B_{\gamma}(\tau)}\|\eta(\tau)\|^2_V \, \d \tau \\
& \quad   + \frac{2 C_B}{\lambda_1 (\gamma_0-\tilde{\gamma}_0)} \int^s_{t_0}e^{B_{\gamma}(\tau)}\| \eta(\tau)\|^2_H\|v(\tau)\|^2_U \, \d \tau.
\end{align*}
Recall the definition of $\tilde{\gamma}_0$ defined as \eqref{def: gamma}, we have $\gamma_0\geq \tilde{\gamma}_0 > \frac{C_{\rho_2}}{\lambda_1}$ and $\gamma_0> \frac{(4+\beta)C_{\rho_1}C_{A,\rho,\epsilon}}{\lambda_1 \gamma_0}$. When $\gamma\in [0,\lambda_1 \gamma_0- C_{\rho_2}]$, using the Gr\"onwall inequality and the fact $v \in S_{M}$, we get
\begin{equation*}
  \sup_{s\in [t_0,t]}e^{B_{\gamma}(s)}\|\eta(s)\|^2_H+\frac{\gamma_0-\tilde{\gamma}_0}{2}\int^t_{t_0}e^{B_{\gamma}(\tau)}\|\eta(\tau)\|^2_V \, \d \tau
\leq  e^{\frac{2 C_B}{\lambda_1 (\gamma_0-\tilde{\gamma}_0)}M} \|\xi_1-\xi_2\|^2_H .
\end{equation*}
Then for any $t_2 \geq t_1 \geq t_0$, using \eqref{eq: energy estimate skeleton} and \eqref{eq: energy estimate skeleton beta}, we arrive at
\begin{align*}
  & \sup_{s\in [t_1,t_2]} \big\{ e^{\gamma(s-t_0)}\|\eta(s)\|^2_H + \frac{\gamma_0-\tilde{\gamma}_0}{2} \int^s_{t_0}e^{\gamma(\tau-t_0)}\|\eta(\tau)\|^2_V \, \d \tau \big\} \\
  & \quad \leq \sup_{s\in [t_0,t]} \big\{ e^{B_{\gamma}(s)}\|\eta(s)\|^2_H +\frac{\gamma_0-\tilde{\gamma}_0}{2}\int^s_{t_0}e^{B_{\gamma}(\tau)}\|\eta(\tau)\|^2_V \, \d \tau \big\} \\
  & \quad \quad \times \exp \Big\{ \frac{2 C_{\rho_1}}{\gamma_0} \Big( \frac{\gamma_0}{2} \int^s_{t_0} (1+\|X^{0,v}_{\tau;t_0,\xi_2}\|^{\beta}_H) \|X^{0,v}_{\tau;t_0,\xi_2}\|^2_V \, \d \tau - \frac{(4+\beta)C_{A,\rho,\epsilon}}{2} (s-t_0) \Big) \Big\}\\
  & \quad \leq  \|\xi_1-\xi_2\|^2_H \exp\Big\{ \frac{2 C_B}{\lambda_1 (\gamma_0-\tilde{\gamma}_0)}M+\frac{ 2 C_{\rho_1}}{\gamma_0}\big( \|\xi_2\|^{2}_H+ \|\xi_2\|^{2+\beta}_H + \frac{2(4+\beta) C_B}{\lambda_1 \gamma_0}M ) \Big\}.
\end{align*}
Dividing $e^{\gamma(t_1-t_0)}$ in the above inequality, we have
\begin{align*}
  & \sup_{s\in [t_1,t_2]} \big\{ \|\eta(s)\|^2_H +\frac{\gamma_0-\tilde{\gamma}_0}{2} \int^s_{t_0} \|\eta(\tau)\|^2_V \, \d \tau \big\} \\
  & \leq e^{-\gamma(t_1-t_0)}  \sup_{s\in [t_1,t_2]} \big\{ e^{\gamma(s-t_0)}\|\eta(s)\|^2_H + \frac{\gamma_0-\tilde{\gamma}_0}{2} \|\eta(\tau)\|^2_V \, \d \tau \big\} \\
  & \leq  \|\xi_1-\xi_2\|^2_H \exp\Big\{ \frac{2 C_B}{\lambda_1 (\gamma_0-\tilde{\gamma}_0)}M+\frac{ 2 C_{\rho_1}}{\gamma_0}\big( \|\xi_2\|^{2}_H+ \|\xi_2\|^{2+\beta}_H + \frac{2(4+\beta) C_B}{\lambda_1 \gamma_0}M ) \Big\} \, e^{-\gamma(t_1-t_0)}.
\end{align*}
Thus, we obtain \eqref{eq: energy estimate skeleton diff}. Following the proof of Theorem \ref{thm: stationary solution} and using estimate \eqref{eq: energy estimate skeleton diff}, we conclude that $\{X^{0,v}_{\cdot;-n,\xi}\}_{n \geq N}$ is a Cauchy sequence in $\sC_{-N}$ for every $N \in \N$ and converges to the same limit element $\mathcal{Y}^{0,v} \in \sC_{-\infty}$. By the similar arguments in Theorem \ref{thm: stationary solution}, the convergence yields that $\mathcal{Y}^{0,v}$ satisfies the skeleton Eq. \eqref{eq: skeleton stationary} and $\sup_{t \in \R} \| \mathcal{Y}^{0,v}(t) \|_{H}^2$ is uniformly bounded by $C(0,M)$. The estimate \eqref{eq: energy estimate skeleton diff} also implies the uniqueness.
\end{proof}

Next, we will show $\mathcal{I}_{t_0,\xi}$ and $\mathcal{I}_{-\infty}$ are good rate functions.

\begin{lemma}\label{lem: good rate func}
  For each $M>0$, $\xi\in H$ and for any fixed $t_0\in \mathbb{R}$, the maps $G_{t_0,\xi}: v\mapsto X^{0,v}_{\cdot;t_0,\xi}$ and $G_{-\infty}: v \mapsto \cY^{0,v}$ are continuous from $S_M$ to $\sC_{t_0}$. Thus, the following sets
\begin{align*}
  \Gamma_{M;t_0,\xi}:=\Big\{\mathcal{G}^0_{t_0,\xi}\Big(\int^{\cdot}_0 v(s) \,  \d s\Big)\in \sC_{t_0}: v\in S_M\Big\};\
\Gamma_{M;-\infty}:=\Big\{\mathcal{G}^0_{-\infty}\Big(\int^{\cdot}_0 v(s) \,  \d s \Big)\in \sC_{-\infty}: v\in S_M\Big\}
\end{align*}
are compact subsets of $\sC_{t_0}$. It follows that $\mathcal{I}_{t_0,\xi}$ and $\mathcal{I}_{-\infty}$ defined by \eqref{def: rate func 1}-\eqref{def: rate func 2} are good rate functions. Moreover, for each compact set $K\subset H$, the set $\Gamma_{M;t_0,K}:=\cup_{\xi\in K}\Gamma_{M;t_0,\xi}$ is compact.
\end{lemma}
\begin{proof}
  Let $\{v_n\}_{n \geq 1} \subset S_M$ be the sequence, and $\{v_n\}_{n \geq 1}$ converges to $v$ in $S_M$ as $n\rightarrow \infty$. For fixed $t_0\in \mathbb{R}$, for any $\gamma\in [0,\lambda_1 \gamma_0)$, $t\geq t_0$ and $s\in [t_0,t]$, let
\begin{equation*}
  B_{\gamma}(s)=\gamma(s-t_0)-C_{\rho_1}\int^s_{t_0}  (1+ \|X^{0,v}_{\tau;t_0,\xi}\|^{\beta}_H) \|X^{0,v}_{\tau;t_0,\xi}\|^2_V \, \d \tau \leq \gamma(s-t_0) .
\end{equation*}
Denote $\eta^n(s)=X^{0,v_n}_{s;t_0,\xi}-X^{0,v}_{s;t_0,\xi}$. Then, we get
\begin{align*}
  e^{B_{\gamma}(s)}\|\eta^n(s)\|^2_H & =\gamma\int^s_{t_0}e^{B_{\gamma}(\tau)}\|\eta^n(\tau)\|^2_H \, \d \tau-C_{\rho_1}\int^s_{t_0}e^{B_{\gamma}(\tau)} (1+ \|X^{0,v}_{\tau;t_0,\xi}\|^{\beta}_H) \|X^{0,v}_{\tau;t_0,\xi}\|^{2}_V \, \d \tau\\
& \quad +2 \int^s_{t_0} e^{B_{\gamma}(\tau)} \scalV{ A^{0}(X^{0,v_n}_{\tau;t_0,\xi})-A^{0}(X^{0,v}_{\tau;t_0,\xi}),\eta^n(\tau)} \, \d \tau \\
& \quad +2\int^s_{t_0} e^{B_{\gamma}(\tau)} \scalH{ B(X^{0,v_n}_{\tau;t_0,\xi})v_n-B(X^{0,v}_{\tau;t_0,\xi})v,\eta^n(\tau) } \, \d \tau.
\end{align*}
Owing to condition {\rm [A2]}, it follows that
\begin{align*}
 &\int^s_{t_0} e^{B_{\gamma}(\tau)} 2 \scalV{ A^{0}(X^{0,v_n}_{\tau;t_0,\xi})-A^{0}(X^{0,v}_{\tau;t_0,\xi}),\eta^n(\tau)} \, \d \tau\\
& \leq  \int^s_{t_0}e^{B_{\gamma}(\tau)}\Big[-2\gamma_0\|\eta^n(\tau)\|^2_V+ \big(C_{\rho_1} (1+ \|X^{0,v}_{\tau;t_0,\xi}\|^{\beta}_H) \|X^{0,v}_{\tau;t_0,\xi}\|^2_V + C_{\rho_2} \big) \|\eta^n(\tau)\|^2_H   \Big] \, \d \tau.
\end{align*}
Using condition {\rm [A4]}, we have
\begin{align*}
&2\int^s_{t_0}e^{B_{\gamma}(\tau)} \scalH{ B(X^{0,v_n}_{\tau;t_0,\xi})v_n-B(X^{0,v}_{\tau;t_0,\xi})v,\eta^n(\tau) } \, \d \tau\\
& \leq 2\int^s_{t_0}e^{B_{\gamma}(\tau)}\scalH{(B(X^{0,v_n}_{\tau;t_0,\xi})-B(X^{0,v}_{\tau;t_0,\xi}))v_n,\eta^n(\tau) } \, \d \tau
+2\int^s_{t_0}e^{B_{\gamma}(\tau)}\scalH{B(X^{0,v}_{\tau;t_0,\xi})(v_n-v),\eta^n(\tau) } \, \d \tau\\
& \leq 2 \sqrt{C_B}\int^s_{t_0}e^{B_{\gamma}(\tau)}\|\eta^n(\tau)\|^2_H\|v_n\|_U \, \d \tau+|R_n(t)|,
\end{align*}
where
\begin{equation*}
  R_n(t):=2\int^t_{t_0}e^{B_{\gamma}(s)}\scalH{B(X^{0,v}_{s;t_0,\xi})(v_n(s)-v(s)),\eta^n(s) }\, d s.
\end{equation*}
Using the fact $\lambda_1\gamma_0 > C_{\rho_2}$ and choosing $\gamma \in (0, \lambda_1 \gamma_0 - C_{\rho_2}]$, we get
\begin{equation*}
   e^{B_{\gamma}(s)}\|\eta^n(s)\|^2_H + \gamma_0 \int^s_{t_0}e^{B_{\gamma}(\tau)}\|\eta^n(\tau)\|^2_V \, \d \tau  \leq 2 \sqrt{C_B}\int^s_{t_0}e^{B_{\gamma}(\tau)}\|\eta^n(\tau)\|^2_H\|v_n\|_U \, \d \tau+|R_n(t)|.
\end{equation*}
By Gr\"onwall inequality and using estimates \eqref{eq: energy estimate skeleton} and \eqref{eq: energy estimate skeleton beta}, we obtain
\begin{align}
   & \sup_{s\in [t_0,t]} \big\{ e^{\gamma(s-t_0)}\|\eta^n(s)\|^2_H+\gamma_0\int^t_{t_0} e^{\gamma(s-t_0)}\|\eta^n(s)\|^2_V \, \d s \big\} \nonumber  \\
   & \leq  \sup_{s\in [t_0,t]} \big\{ e^{B_{\gamma}(s)}\|\eta^n(s)\|^2_H+\gamma_0\int^t_{t_0} e^{B_{\gamma}(s)}\|\eta^n(s)\|^2_V \, \d s \big\} \times e^{C_{\rho_1}\int^t_{t_0}  (1+ \|X^{0,v}_{s;t_0,\xi}\|^{\beta}_H) \|X^{0,v}_{s;t_0,\xi}\|^2_V \, \d s} \nonumber \\
  & \leq  \sup_{s\in [t_0,t]}|R_n(s)| \exp \{ 2 \sqrt{C_B M (t-t_0)}  \}  \nonumber \\
  & \quad \times \exp\Big\{\frac{2 C_{\rho_1}}{\gamma_0}\big( \|\xi\|^{2}_H+ \|\xi\|^{2+\beta}_H + (2+ \frac{\beta}{2}) \big(\frac{4 C_B}{\lambda_1 \gamma_0}M +  C_{A,\rho,0}  ( t-t_0) \big) \big) \Big\}. \label{est: skeleton diff v}
\end{align}
Analogous to the estimate for remainder term $Z_{n}(t)$ in \cite[Section 4.4]{Pan_Large_2026}, using the estimate \eqref{eq: energy estimate skeleton} and the relative compactness argument, we get
\begin{equation*}
  \lim_{n\rightarrow \infty}\sup_{s\in [t_0,t]}|R_n(s)|=0, \quad \forall \, t_0<t <+\infty.
\end{equation*}
Thus, $\{ \eta^n \}_{n}$ is a Cauchy sequence in $(\sC_{t_0},d_{t_0})$ and the map $G_{t_0,\xi}: v\mapsto X^{0,v}_{\cdot;t_0,\xi}$ is continuous from $S_M$ to $\sC_{t_0}$. By the compactness of $S_M$ under weak topology and the weak-strong continuity of $G_{t_0,\xi}$, we deduce that $\Gamma_{M;t_0,\xi}$ is compact. For each compact set $K\subset H$, using \eqref{eq: energy estimate skeleton diff} and the compactness of $\Gamma_{M;t_0,\xi}$, it follows that $\Gamma_{M;t_0,K}:=\cup_{\xi\in K}\Gamma_{M;t_0,\xi}$ is also compact. Thus, level set
\begin{align*}
  \{\Phi\in \sC_{t_0}: \mathcal{I}_{t_0,\xi}(\Phi)\leq M\}=\cap_{n\geq 1} \Gamma_{2M+\frac{1}{n};t_0,\xi}
\end{align*}
is compact, the rate function $\mathcal{I}_{t_0,\xi}$ is a good rate function.

Next, we  will show that $\Gamma_{M;-\infty}:=\{\mathcal{G}^0_{-\infty}(\int^{\cdot}_0v(s) \,  \d s): v\in S_M\}$ is compact and $\mathcal{I}_{-\infty}$ is a good rate function. Recall the definition of $(\sC_{-\infty},d_{-\infty})$, to get the continuity of $G_{-\infty}: S_M \rightarrow \sC_{-\infty}$, we only need to prove that if $v_n$ converges to $v$ in $S_M$, then for any $t_2>t_1$,
\begin{equation*}
  \lim_{n \rightarrow +\infty} \sup_{s\in [t_1,t_2]} \Big\{\|\mathcal{Y}^{0,v_n}(s)-\mathcal{Y}^{0,v}(s)\|^2_H + \frac{\gamma_0}{2} \int^{t_2}_{t_1}\|\mathcal{Y}^{0,v_n}(s)-\mathcal{Y}^{0,v}(s)\|^2_V\, d s \Big\}=0.
\end{equation*}
Lemma \ref{lem: skeleton} shows the fact $\mathcal{Y}^{0,v}(s)=X^{0,v}_{s;-N,\cY^{0,v}(-N)}$. Thus, it holds that
\begin{align*}
& \sup_{s\in [t_1,t_2]} \Big\{ \|\mathcal{Y}^{0,v_n}(s)-\mathcal{Y}^{0,v}(s)\|^2_H+ \frac{\gamma_0}{2}\int^{t_2}_{t_1}\|\mathcal{Y}^{0,v_n}(s)-\mathcal{Y}^{0,v}(s)\|^2_V\, d s \Big\} \\
& = \sup_{s\in [t_1,t_2]} \Big\{ \|X^{0,v_n}_{s;-N,\cY^{0,v_n}(-N)}-X^{0,v}_{s;-N,\cY^{0,v}(-N)}\|^2_H+ \frac{\gamma_0}{2} \int^{t_2}_{t_1}\|X^{0,v_n}_{s;-N,\cY^{0,v_n}(-N)}-X^{0,v}_{s;-N,\cY^{0,v}(-N)}\|^2_V \, \d s \Big\} \\
& \leq 2 \sup_{s\in [t_1,t_2]} \Big\{ \|X^{0,v_n}_{s;-N,\cY^{0,v_n}(-N)}-X^{0,v_n}_{s;-N,\cY^{0,v}(-N)}\|^2_H+ \frac{\gamma_0}{2} \int^{t_2}_{t_1}\|X^{0,v_n}_{s;-N,\cY^{0,v_n}(-N)}-X^{0,v_n}_{s;-N,\cY^{0,v}(-N)}\|^2_V \, \d s \Big\}  \\
& \quad + 2 \sup_{s\in [t_1,t_2]} \Big\{ \|X^{0,v_n}_{s;-N,\cY^{0,v}(-N)}-X^{0,v}_{s;-N,\cY^{0,v}(-N)}\|^2_H+ \frac{\gamma_0}{2} \int^{t_2}_{t_1}\|X^{0,v_n}_{s;-N,\cY^{0,v}(-N)}-X^{0,v}_{s;-N,\cY^{0,v}(-N)}\|^2_V \, \d s \Big\}.
\end{align*}
Taking $\gamma \in (0,(\lambda_1 \gamma_0- C_{\rho_2}) \vee \frac{\lambda_1 \gamma_0}{4}]$ and using estimate \eqref{eq: energy estimate skeleton diff}, we have
\begin{align*}
  & \sup_n \sup_{s\in [t_1,t_2]} \Big\{ \|X^{0,v_n}_{s;-N,\cY^{0,v_n}(-N)}-X^{0,v_n}_{s;-N,\cY^{0,v}(-N)}\|^2_H+ \frac{\gamma_0-\tilde{\gamma}_0}{2} \int^{t_2}_{t_1}\|X^{0,v_n}_{s;-N,\cY^{0,v_n}(-N)}-X^{0,v_n}_{s;-N,\cY^{0,v}(-N)}\|^2_V \, \d s \Big\} \\
  & \leq \exp\Big\{ \frac{2 C_B}{\lambda_1 (\gamma_0-\tilde{\gamma}_0)}M+\frac{ 2 C_{\rho_1}}{\gamma_0}\big( \|\cY^{0,v}(-N)\|^{2}_H+ \|\cY^{0,v}(-N)\|^{2+\beta}_H + \frac{2(4+\beta) C_B}{\lambda_1 \gamma_0}M ) \Big\}  \\
  & \quad \times \sup_{n} \big\{ \|\cY^{0,v}(-N)-\cY^{0,v_n}(-N)\|^2_H \, e^{-\gamma(t_1+N)} \big\}.
\end{align*}
Recall that $C(\xi,M):= \|\xi\|^2_H+ \frac{4 }{\lambda_1 \gamma_0} \big(C_B M +C_{A,\rho,0} \big)$ and $ \sup_{t \in \R} \sup_{n} \| \cY^{0,v}(t) \|_{H}^2 \vee \| \cY^{0,v}(t) \|_{H}^2$ is bounded by $C(0,M)$, we obtain that the right of the above inequality tends to zero as $N \rightarrow +\infty$. For any fixed $N \geq 1$, taking into account the weak-strong continuity of $G_{-N,\cY^{0,v}(-N)}: S_M \rightarrow \sC_{-N}$, we get that when $n$ tends to zero, it holds
\begin{equation*}
\sup_{s\in [t_1,t_2]} \Big\{ X^{0,v_n}_{s;-N,\cY^{0,v}(-N)}-X^{0,v}_{s;-N,\cY^{0,v}(-N)}\|^2_H+ \frac{\gamma_0 }{2} \int^{t_2}_{t_1} \|X^{0,v_n}_{s;-N,\cY^{0,v}(-N)}-X^{0,v}_{s;-N,\cY^{0,v}(-N)}\|^2_V \, \d s \Big\} \rightarrow 0,
\end{equation*}
for any $t_2 \geq t_1 \geq -N$. Combining the above estimates, we obtain $\lim\limits_{n\rightarrow+\infty}\|\mathcal{Y}^{0,v_n}-\mathcal{Y}^{0,v}\|^2_{\sC_{-\infty}}=0$ and $G_{-\infty}: S_M \rightarrow \sC_{-\infty}$ is a continuous map. Then the set $\Gamma_{M;-\infty}$ is compact, and $\mathcal{I}_{-\infty}$ is also a good rate function.
\end{proof}

By a similar method as \cite[Lemma 4.3]{Liu_Large_2025}, we derive the following property of $\mathcal{I}_{t_0,\xi}$, which is used to prove uniform LDP for the solution map.
\begin{lemma}\label{lem: lsc uniform LDP}
  For each $\Phi\in \sC_{t_0}$, $\xi\mapsto \mathcal{I}_{t_0,\xi}$ is a lower semicontinuous (l.s.c.) map from $H$ to $[0,+\infty]$. Furthermore, for each $M<+\infty$,
\begin{align*}
  \Lambda_{M,t_0,K}:=\{\Phi\in \sC_{t_0}: \mathcal{I}_{t_0,\xi}(\Phi)\leq M, \, \xi\in K\}
\end{align*}
is a compact subset of $\sC_{t_0}$.
\end{lemma}
\begin{proof}
  Let $\{\xi_n\}_{n}$ be a convergence sequence in $H$ with the limit $\xi$. For each $\Phi\in\sC_{t_0}$, if
$\liminf_{n\rightarrow +\infty}\mathcal{I}_{t_0,\xi_n}(\Phi)=M<+\infty$, then by the definition of $\mathcal{I}_{t_0,\xi_n}(\Phi)$, for any $\delta>0$, there exists $\{v_{n_k}\}_{k\in \mathbb{N}}\subset S_{2M+\delta}$ such that
\begin{align*}
  \Phi=\mathcal{G}^0_{t_0,\xi_{n_k}}\Big(\int^{\cdot}_0v_{n_k}(s) \,  \d s\Big),\ \forall \, k \in \mathbb{N},
\end{align*}
where $\{n_k\}_k$ is a subsequence of $\mathbb{N}$. By the compactness of $S_{2M+\delta}$, there is a subsequence (still denoted by $\{n_k\}_k$) such that $v_{n_k}$ converges to $v\in S_{2M+\delta}$ as $k\rightarrow +\infty$. With the aid of Lemma \ref{lem: good rate func} and estimate \eqref{eq: energy estimate skeleton diff}, for any $t_0\in \mathbb{R}$ and $t_2>t_1\geq t_0$, we have
\begin{align*}
  &\sup_{t\in [t_1,t_2]} \Big\{ \|\Phi(t)-\mathcal{G}^0_{t_0,\xi}\Big(\int^{\cdot}_0v(s) \,  \d s\Big)(t)\|^2_{H}
  + \gamma_0 \int^{t}_{t_1}\|\Phi(\tau)-\mathcal{G}^0_{t_0,\xi}\Big(\int^{\cdot}_0v(s) \,  \d s\Big)(\tau)\|^2_{V} \, \d \tau \Big\}\\
& \leq \lim_{k\rightarrow \infty}\sup_{s\in [t_1,t_2]}\|\mathcal{G}^0_{t_0,\xi_{n_k}}\Big(\int^{\cdot}_0v_{n_k}(s) \,  \d s\Big)-\mathcal{G}^0_{t_0,\xi}\Big(\int^{\cdot}_0v(s) \,  \d s\Big)(t)\|^2_{H}\\
& \quad + \lim_{k\rightarrow \infty}  \gamma_0  \int^{t_2}_{t_1}\|\mathcal{G}^0_{t_0,\xi_{n_k}}\Big(\int^{\cdot}_0v_{n_k}(s) \,  \d s\Big)-\mathcal{G}^0_{t_0,\xi}\Big(\int^{\cdot}_0v(s) \,  \d s\Big)(t)\|^2_{V} \, \d t\\
& \leq \lim_{k\rightarrow \infty}\sup_{s\in [t_1,t_2]}\|\mathcal{G}^0_{t_0,\xi_{n_k}}\Big(\int^{\cdot}_0v_{n_k}(s) \,  \d s\Big)-\mathcal{G}^0_{t_0,\xi}\Big(\int^{\cdot}_0v_{n_k}(s) \,  \d s\Big)(t)\|^2_{H}\\
& \quad +\lim_{k\rightarrow \infty}\sup_{s\in [t_1,t_2]}\|\mathcal{G}^0_{t_0,\xi}\Big(\int^{\cdot}_0v_{n_k}(s) \,  \d s\Big)-\mathcal{G}^0_{t_0,\xi}\Big(\int^{\cdot}_0v(s) \,  \d s\Big)(t)\|^2_{H}\\
& \quad + \lim_{k\rightarrow \infty} \gamma_0 \int^{t_2}_{t_1}\|\mathcal{G}^0_{t_0,\xi_{n_k}}\Big(\int^{\cdot}_0v_{n_k}(s) \,  \d s\Big)-\mathcal{G}^0_{t_0,\xi}\Big(\int^{\cdot}_0v_{n_k}(s) \,  \d s\Big)(t)\|^2_{V} \, \d t\\
& \quad +\lim_{k\rightarrow \infty} \gamma_0 \int^{t_2}_{t_1}\|\mathcal{G}^0_{t_0,\xi}\Big(\int^{\cdot}_0v_{n_k}(s) \,  \d s\Big)-\mathcal{G}^0_{t_0,\xi}\Big(\int^{\cdot}_0v(s) \,  \d s\Big)(t)\|^2_{V} \, \d t\\
& =0.
\end{align*}
So $\mathcal{G}^0_{t_0,\xi}\Big(\int^{\cdot}_0 v(s) \,  \d s\Big)=\Phi$ in $\sC_{t_0}$ and $\mathcal{I}_{t_0,\xi}(\Phi)\leq M+\frac{\delta}{2}$. By the arbitrariness of $\delta$, we have
\begin{align*}
\mathcal{I}_{t_0,\xi}(\Phi)\leq M=\liminf_{n\rightarrow +\infty}\mathcal{I}_{t_0,\xi_n}(\Phi).
\end{align*}
We get $\xi\mapsto \mathcal{I}_{t_0,\xi}(\Phi)$ is an l.s.c. map. Following the proof of Lemma 5 in \cite{Budhiraja_Large_2008}, we obtain
\begin{align*}
\Lambda_{M,t_0,K}=\bigcap_{n\geq 1}\Gamma_{2M+\frac{1}{n};t_0,K}.
\end{align*}
Thus, $\Lambda_{M,t_0,K}$ is a compact follows from the compactness of $\Gamma_{2M+\frac{1}{n};t_0,K}$.
\end{proof}

\section{ULDP for the solution on compact sets} \label{sec: uniform LDP}
Recall that  for $t\in [t_0,T]$, $X^{\epsilon}_{t;t_0,\xi}=\mathcal{G}^{\epsilon}_{t_0}(W,\xi) (t,\omega)=:\mathcal{G}^{\epsilon}_{t_0,\xi}(W) (t,\omega)$ is the solution of Eq. \eqref{eq: SPDE} started from $t_0$ with initial data $\xi \in H$. In this section, we will prove the uniform LDP for $\{X^{\epsilon}_{\cdot;t_0,\xi}\}_{\epsilon>0} \subset \mathcal{E}:=C([t_0,T];H)\cap L^2([t_0,T];V)$ with initial datas $\xi$ in compact sets of $H$. The equipped distance is $d_{\mathcal{E}}(X,Y):= \sup_{s \in [t_0,T]} \| X(s)-Y(s)\|_{H}^2 \wedge \int_{t_0}^T \| X(s)-Y(s) \|_{V}^2 \, \d s$.
\begin{definition}[Freidlin-Wentzell uniform LDP] \label{def: ULDP}
    Let $\mathcal{K}$ be the family of all compactness subset of $H$. We say sequence $\{X^{\epsilon}_{\cdot;t_0,\xi}\}_{\epsilon>0}\subset \mathcal{E}$ indexed by $\xi \in H$ satisfies the uniform large deviation principle with rate function $\cI_{\xi}$, uniformly in $\mathcal{K}$ if the following two conditions hold:
    \begin{itemize}
        \item [1.] Upper bound. For any $K \in \mathcal{K}$, $\delta>0$ and $s_0>0$,
            \begin{equation*}
                \limsup_{\epsilon \rightarrow 0} \sup_{\xi \in K} \sup_{s \leq s_0} \Big\{ \epsilon \log \P \big( \inf_{\Phi \in \{\Phi \in \mathcal{E}; \, \cI_{\xi}(\Phi) \leq s\} } d_{\mathcal{E}}(X^{\epsilon}_{\cdot;t_0,\xi},\phi) \geq \delta \big) + s \Big\} \leq 0.
            \end{equation*}
        \item [2.] Lower bound. For any $K \in \mathcal{K}$, $\delta>0$ and $s_0>0$,
            \begin{equation*}
                \liminf_{\epsilon \rightarrow 0} \inf_{\xi \in K} \inf_{\{\Phi \in \mathcal{E}; \, \cI_{\xi}(\Phi) \leq s_0\}} \Big\{ \epsilon\,\log \P \big( d_{\mathcal{E}}(X^{\epsilon}_{\cdot;t_0,\xi},\Phi) <\delta \big) + \cI_{\xi}(\Phi) \Big\} \geq 0.
            \end{equation*}
    \end{itemize}
\end{definition}

Referring to \cite[Theorem 4.4]{Budhiraja_variational_2000} and \cite[Theorem 5]{Budhiraja_Large_2008}, the following two conditions (S1)-(S2) are sufficient to ensure the uniform LDP for $\{X^{\epsilon}_{\cdot;t_0,\xi}\}_{\epsilon>0}$, which is the well-known weak convergence method. It reads as follows.
\begin{proposition} \label{prp-1}
For the solution $X^{\epsilon}_{t;t_0,\xi}$ of Eq. \eqref{eq: SPDE}, if the following two items hold:
\begin{description}
  \item[(S1)] Let $M>0$, the sequence $\{v^{\epsilon}: \epsilon>0\}\subset \mathcal{A}_M$ and $\{\xi^{\epsilon}\}_{\epsilon>0} \subset H$. If $v^{\epsilon}$ converges to $v$ in distribution as $S_M-$valued random elements and $\xi^{\epsilon}$ converges to $\xi$ in $H$ as $\epsilon\rightarrow 0$, then
\begin{equation*}
  \mathcal{G}^{\epsilon}_{t_0,\xi^{\epsilon}}\Big(W(\cdot)+\frac{1}{\sqrt{\epsilon}}\int^{\cdot}_0v^{\epsilon}(s) \, \d s\Big) \rightarrow \mathcal{G}^{0}_{t_0,\xi}\Big(\int^{\cdot}_0v(s) \,  \d s\Big) \quad \mbox{on} \quad \sC_{t_0},
 \end{equation*}
in distribution as $\epsilon\rightarrow 0$.
  \item[(S2)] $\xi\mapsto \mathcal{I}_{t_0,\xi}$ is an l.s.c. map from $H$ to $[0,+\infty]$; and for each $M>0$ and compact set $K\subset H$, the set $\Gamma_{M; t_0,K}=\{\mathcal{G}^{0}_{t_0,\xi}(\int^{\cdot}_0v(s) \,  \d s)\in\sC_{t_0}: v\in S_M, \xi\in K\}$ is compact,
\end{description}
then $X^{\epsilon}_{t;t_0,\xi}$ satisfies ULDP over compact subsets  on $\sC_{t_0}$ with the good rate function $\mathcal{I}_{t_0,\xi}$:
\begin{equation*}
  \mathcal{I}_{t_0,\xi}(\Phi):=\inf_{\Big \{v\in L^2(\mathbb{R};U); \, \Phi=\mathcal{G}^{0}_{t_0,\xi}\Big(\int^{\cdot}_0 v(s) \,  \d s\Big)\Big\}} \big\{\frac{1}{2}\int^{+\infty}_{-\infty} |v(s)|^2_U \, \d s \big\}, \quad \forall \, \Phi \in \sC_{-\infty}.
\end{equation*}
\end{proposition}

Clearly, the condition (S2) has been proved by Lemma \ref{lem: lsc uniform LDP}. In the following, we devote to verifying the condition (S1). For each $\xi\in H$ and $v\in \mathcal{A}_M$, let $X^{\epsilon,v}_{t;t_0,\xi}$ be the solution of
\begin{equation}\label{eq: controlled equation truncated}
  X^{\epsilon,v}_{t;t_0,\xi}=\xi+\int^t_{t_0}A^{\epsilon}(X^{\epsilon,v}_{s;t_0,\xi}) \,  \d s+
\int^t_{t_0}B(X^{\epsilon,v}_{s;t_0,\xi})v(s) \,  \d s+\sqrt{\epsilon}\int^t_{t_0}B(X^{\epsilon,v}_{s;t_0,\xi}) \, \d W_s, \quad \forall \, t\geq t_0,
\end{equation}
and equal to $\xi$ otherwise. By Proposition \ref{Prop: measurable map truncated} and Girsanov transform, for each $\epsilon>0$, it holds
\begin{equation*}
  \mathbb{P}\mbox{-}a.s.,\ X^{\epsilon,v}_{t;t_0,\xi}= \mathcal{G}^{\epsilon}_{t_0;\xi}\Big(W(\cdot)+\frac{1}{\sqrt{\epsilon}}\int^{\cdot}_0v(s) \,  \d s\Big)(t), \quad \forall \,  t\in \mathbb{R},
\end{equation*}
which also implies the existence and uniqueness of solution $X^{\epsilon,v}_{t;t_0,\xi}$ to the controlled equation \eqref{eq: controlled equation truncated}. The following result gives an energy estimate of $X^{\epsilon,v}_{t;t_0,\xi}$ uniformly with $\epsilon $ and $v \in \mathcal{A}_{M}$. It is necessary for the ULDP in finite time interval and the LDP for stationary solution.
\begin{lemma} \label{lem: well-posedness controlled eq}
  Let $ M>0$, $\epsilon_0>0$ and $v \in \mathcal{A}_M$. Assume that the conditions {\rm [A1-A5]} hold for any $0 \leq \epsilon < \epsilon_0$, then there exists a unique solution $X^{\epsilon,v}_{t;t_0,\xi}$ to Eq. \eqref{eq: controlled equation truncated}. It satisfies
\begin{align}
 \sup_{v\in \mathcal{A}_M}  &   \E \Big[ \exp \Big\{  \delta \sup_{t \geq t_0} \Big( \| X^{\epsilon,v}_{t;t_0,\xi}\|_{H}^{2+\beta} + \frac{\gamma_0}{2} \int_{t_0}^t \| X^{\epsilon,v}_{s;t_0,\xi} \|_{H}^{\beta} \| X^{\epsilon,v}_{s;t_0,\xi} \|_{V}^{2} \, \d s \nonumber \\
      & \quad  \quad - (1+\frac{\beta}{2}) \big( \big( C_{A,\rho,\epsilon}+ \beta C_B \epsilon \big) \,  (t-t_0)  \big) \Big) \Big\} \Big]  \leq  2 e^{\delta ( \| \xi \|_{H}^{2+\beta} + (2+\beta) \frac{2 C_B}{\lambda_1 \gamma_0} M )  }; \label{est: energy controlled eq beta} \\
 \sup_{v\in \mathcal{A}_M}  &   \E \Big[ \exp \Big\{  \delta \sup_{t \geq t_0} \big( \| X^{\epsilon,v}_{t;t_0,\xi}\|_{H}^{2} +  \frac{\gamma_0}{2} \int_{t_0}^t  \| X^{\epsilon,v}_{s;t_0,\xi} \|_{V}^{2} \, \d s  -  C_{A,\rho,\epsilon} \,  (t-t_0)   \big) \Big\} \Big]  \leq  2 e^{\delta ( \| \xi \|_{H}^{2} + \frac{4 C_B}{\lambda_1 \gamma_0} M )  }. \label{est: energy controlled eq}
\end{align}
for any $0<\delta \leq \frac{ (1+\beta)\lambda_1 \gamma_{0}}{4 (2+\beta)^2 C_{B} \epsilon}$ and for any $0<\epsilon \leq \epsilon_0$. Furthermore, we have
\begin{align}
 \sup_{v\in \mathcal{A}_M}  &   \E \exp \Big\{ \delta \| X^{\epsilon,v}_{t_1;t_0,\xi}\|_{H}^{2+\beta}  \Big\} \leq 2 e^{\delta ( \| \xi \|_{H}^{2+\beta} + (2+\beta) \frac{C_{A,\rho,\epsilon}+ 2 C_B M+ \beta C_B \epsilon }{\lambda_1 \gamma_0} )}, \label{est: energy controlled eq beta H} \\
\sup_{v\in \mathcal{A}_M}  &     \E \Big[ \exp \Big\{  \frac{\delta}{2} \sup_{t \geq t_1} \big( \frac{\gamma_0}{2} \int_{t_1}^t \| X^{\epsilon,v}_{s;t_0,\xi} \|_{H}^{\beta} \| X^{\epsilon,v}_{s;t_0,\xi} \|_{V}^{2} \, \d s- (1+\frac{\beta}{2}) \big( C_{A,\rho,\epsilon}+ \beta C_B \epsilon \big) (t-t_1) \big) \Big\} \Big] \nonumber \\
 &  \leq 2 \, e^{\delta  \big( \frac{\| \xi \|_{H}^{2+\beta}}{2} + (1+\frac{\beta}{2}) \frac{ C_{A,\rho,\epsilon} + 2C_B M+  \beta C_{B} \epsilon  }{\lambda_1 \gamma_0} \big)}. \label{est: energy controlled eq beta V}
\end{align}
The ``$\beta=0$'' version of the above estimates also hold. Taking the supremum over $\epsilon \in (0,\epsilon_0]$ of the above inequalities yields the uniform estimate respect to $\epsilon$ on a finite time interval.
\end{lemma}

\begin{proof}
The existence and uniqueness of solution $X^{\epsilon,v}_{t;t_0,\xi}$ to the controlled equation \eqref{eq: controlled equation truncated} is given by Proposition \ref{Prop: measurable map truncated} and the Girsanov theorem. Next, we will show the uniform energy estimate.

By It\^o formula and coercivity conditions {\rm [A5]} and condition {\rm [A4]}, we obtain
\begin{align}
        & \|X^{\epsilon,v}_{t;t_0,\xi}\|_{H}^{2+\beta} - \| \xi \|_{H}^{2+\beta} - (2+\beta) \sqrt{\epsilon} \, \int_{t_0}^{t} \|X^{\epsilon,v}_{s;t_0,\xi}\|_{H}^{\beta} \scalH{X^{\epsilon,v}_{s;t_0,\xi}, B(X^{\epsilon,v}_{s;t_0,\xi}) \, \d W_s}  \nonumber \\
        & \quad  - (2+\beta) \int_{t_0}^{t} \|X^{\epsilon,v}_{s;t_0,\xi} \|_{H}^{\beta} \scalH{X^{\epsilon,v}_{s;t_0,\xi}, B(X^{\epsilon,v}_{s;t_0,\xi}) v(s)} \, \d s  \nonumber \\
        & \quad  = (1+\frac{\beta}{2}) \int_{t_0}^{t} \|X^{\epsilon,v}_{s;t_0,\xi}\|_{H}^{\beta} \big( 2 \, \scalV{A^{\epsilon}(X^{\epsilon,v}_{s;t_0,\xi}),X^{\epsilon,v}_{s;t_0,\xi}} +  \epsilon \, \| B (X^{\epsilon,v}_{s;t_0,\xi}) \|_{2}^2 \big) \, \d s \nonumber \\
        & \quad \quad + \beta (1+ \frac{\beta}{2}) \, \epsilon  \int_{t_0}^{t} \|X^{\epsilon,v}_{s;t_0,\xi}\|_{H}^{\beta-2} \sum_{k} \big| \scalH{X^{\epsilon,v}_{s;t_0,\xi},B(X^{\epsilon,v}_{s;t_0,\xi})  \mathcal{U}_{k}}  \big|^2 \, \d s \nonumber \\
        & \quad  \leq -(1+\frac{\beta}{2}) \gamma_0 \int_{t_0}^{t}  \|X^{\epsilon,v}_{s;t_0,\xi} \|_{H}^{\beta} \|X^{\epsilon}_{s;t_0,\xi}\|_{V}^{2} \, \d s + (1+\frac{\beta}{2}) \int_{t_0}^{t} \big( C_{A, \rho, \epsilon} + \beta C_{B} \epsilon \big) \, \d s. \label{est: controlled eq basic}
\end{align}
Using the condition {\rm [A4]}, we obtain
\begin{align}
 & (2+\beta) \Big| \int_{t_0}^{t} \|X^{\epsilon,v}_{s;t_0,\xi} \|_{H}^{\beta} \scalH{X^{\epsilon,v}_{s;t_0,\xi}, B(X^{\epsilon,v}_{s;t_0,\xi}) v(s)} \, \d s \Big| \nonumber \\
  &= (2+\beta) \Big|\int^t_{t_0} \|X^{\epsilon,v}_{s;t_0,\xi}\|^{\beta}_H \sum_{k} \scalH{ B(X^{\epsilon,v}_{s;t_0,\xi}) \, \mathcal{U}_{k},X^{\epsilon,v}_{s;t_0,\xi}} \scalU{v(s),\mathcal{U}_{k}} \, \d s \Big| \nonumber  \\
 & \leq (\beta +2) \Big(\int^t_{t_0} \|X^{\epsilon,v}_{s;t_0,\xi}\|^{2\beta}_H \sum_{k} \scalH{ B(X^{\epsilon,v}_{s;t_0,\xi})\mathcal{U}_{k},X^{\epsilon,v}_{s;t_0,\xi}}^2  \, \d s \Big)^{1/2}\Big(\int^t_{t_0} \|v(s)\|^2_U \, \d s\Big)^{1/2} \nonumber \\
 & \leq (1+\frac{\beta}{2}) \big( \frac{\lambda_1 \gamma_0}{4}\int^t_{t_0} \|X^{\epsilon,v}_{s;t_0,\xi}\|^{2+\beta}_H \, \d s+ \frac{4 C_B}{\lambda_1 \gamma_0}  \int^t_{t_0} \|v(s)\|^2_U \, \d s \big). \label{est: controlled eq basic 2}
\end{align}
Analogous to the proof of Proposition \ref{prop: basic energy estimate beta}, define
 \begin{equation*}
        M (t) := (2+\beta) \sqrt{\epsilon} \int_{t_0}^{t} |X^{\epsilon,v}_{s;t_0,\xi}\|_{H}^{\beta} \scalH{X^{\epsilon,v}_{s;t_0,\xi}, B(X^{\epsilon,v}_{s;t_0,\xi}) \, \d W_s}.
    \end{equation*}
  Let $\<M\>_{t}$ be the quadratic variation of local martingale $M(t)$, for any $\delta>0$, define
    \begin{equation*}
    \mathcal{M}_{\delta}(t):= M(t) - \frac{\delta}{2} \<M\>_{t},
    \end{equation*}
    where $0<\delta \leq \frac{ (2+3\beta)\lambda_1 \gamma_{0}}{8 (2+\beta)^2 C_{B} \epsilon}$. Using conditions {\rm [A4]}, we have
    \begin{align*}
        \d \< M \>_t & \leq (2+\beta)^2 \epsilon \, \|X^{\epsilon,v}_{s;t_0,\xi}\|_{H}^{2 \beta} \| \scalH{X^{\epsilon,v}_{t;t_0,\xi},B(X^{\epsilon,v}_{t;t_0,\xi})} \|_{L_2(U;\R)}^2 \, \d t \\
 & \leq (2+\beta)^2  C_{B} \,\epsilon \, \| X^{\epsilon,v}_{t;t_0,\xi} \|_{H}^{2+\beta} \, \d t \leq \frac{(2+\beta)^2  C_{B} \epsilon }{\lambda_1} \, \| X^{\epsilon,v}_{t;t_0,\xi} \|_{H}^{\beta} \| X^{\epsilon,v}_{t;t_0,\xi} \|_{V}^2  \, \d t.
\end{align*}
Combining the above estimates and taking $0<\delta \leq \frac{ (2+3\beta)\lambda_1 \gamma_{0}}{8 (2+\beta)^2 C_{B} \epsilon}$, we obtain
    \begin{align*}
        E_{X^{\epsilon,v}_{t;t_0,\xi};\beta} &:= \| X^{\epsilon,v}_{t;t_0,\xi}\|_{H}^{2+\beta} + \frac{\gamma_0}{2} \int_{t_0}^t \| X^{\epsilon,v}_{s;t_0,\xi} \|_{H}^{\beta} \| X^{\epsilon,v}_{s;t_0,\xi} \|_{V}^{2} \, \d s \\
        & \, \leq \| \xi \|_{H}^2 + \mathcal{M}_{2\delta}(t) + (1+\frac{\beta}{2}) \Big( \big( C_{A,\rho,\epsilon}+ \beta C_B \epsilon \big) \,  (t-t_0) + \frac{4 C_B}{\lambda_1 \gamma_0} \int^t_{t_0} \|v(s)\|^2_U \, \d s  \Big).
    \end{align*}

Notice that $\exp \{\delta \mathcal{M}_{\delta}(t)  \}$ is a positive supermartingale whose value is $1$ at time $t_0$. By the same argument in the proof of Proposition \ref{prop: basic energy estimate beta}, we have
    \begin{equation*}
        \E \Big[   e^{ \delta \sup_{t \geq t_0 } \mathcal{M}_{2 \delta}(t)} \Big] = 1 + \delta \int_{0}^{\infty} e^{\delta \rho} \, \P \Big( \sup_{t \geq t_0 } \mathcal{M}_{2 \delta}(t) \geq \rho \Big) \, \d \rho \leq 1 + \delta \int_{0}^{\infty} e^{- \delta \rho} \d \rho =  2.
    \end{equation*}
The following inequality gives,
    \begin{align*}
      &   \E \Big[ \exp \Big\{  \delta \sup_{t \geq t_0} \Big( E_{X^{\epsilon,v}_{t;t_0,\xi};\beta} - (1+\frac{\beta}{2}) \big( \big( C_{A,\rho,\epsilon}+ \beta C_B \epsilon \big) \,  (t-t_0) - \frac{4 C_B}{\lambda_1 \gamma_0} \int^t_{t_0} \|v(s)\|^2_U \, \d s \big) \Big) \Big\} \Big] \\
       &  \quad   \leq \E \big[ e^{ \delta \sup_{t \geq t_0 } \mathcal{M}_{2 \delta}(t)} \big] \,  e^{\delta \| \xi \|_{H}^{2+\beta}  } \leq 2 e^{\delta \| \xi \|_{H}^{2+\beta}  }.
    \end{align*}
Using the fact $v \in \mathcal{A}_{M}$, we get estimate \eqref{est: energy controlled eq beta} holds for $0<\delta \leq \frac{ (1+\beta)\lambda_1 \gamma_{0}}{4 (2+\beta)^2 C_{B} \epsilon} \leq \frac{ (2+3\beta)\lambda_1 \gamma_{0}}{8 (2+\beta)^2 C_{B} \epsilon}$. For $0<\delta \leq \frac{ (1+\beta)\lambda_1 \gamma_{0}}{4 (2+\beta)^2 C_{B} \epsilon} \leq \frac{ \lambda_1 \gamma_{0}}{16 C_{B} \epsilon}$, the ``$\beta=0$'' version uniform estimate \eqref{est: energy controlled eq} also holds.

Finally, we focus on the uniform estimate of $\|X^{\epsilon,v}_{t;t_0,\xi} \|_{H}$ and $\int_{t_1}^{t} \| X^{\epsilon,v}_{s;t_0,\xi} \|_{H}^{\beta}  \| X^{\epsilon,v}_{s;t_0,\xi} \|_{V}^{2}  \, \d s$ for all $t_1 \geq t_0$. Similar to the estimate \eqref{est: controlled eq basic} and \eqref{est: controlled eq basic 2}, we have
\begin{align}
        & e^{\gamma (t-t_0)} \|X^{\epsilon,v}_{t;t_0,\xi}\|_{H}^{2+\beta} - \| \xi \|_{H}^{2+\beta} - (2+\beta) \sqrt{\epsilon} \, \int_{t_0}^{t} e^{\gamma (s-t_0)} \|X^{\epsilon,v}_{s;t_0,\xi}\|_{H}^{\beta} \scalH{X^{\epsilon,v}_{s;t_0,\xi}, B(X^{\epsilon,v}_{s;t_0,\xi}) \, \d W_s}  \nonumber \\
        &  \leq \big(  \gamma + (1+\frac{\beta}{2}) \frac{\lambda_1 \gamma_0}{4} \big) \int_{t_0}^{t} e^{\gamma (s-t_0)} \|X^{\epsilon,v}_{s;t_0,\xi}\|_{H}^{2+\beta} \, \d s  -(1+\frac{\beta}{2}) \gamma_0 \int_{t_0}^{t}  e^{\gamma (s-t_0)} \|X^{\epsilon,v}_{s;t_0,\xi} \|_{H}^{\beta} \|X^{\epsilon}_{s;t_0,\xi}\|_{V}^{2} \, \d s \nonumber \\
        & \quad + (1+\frac{\beta}{2}) \int_{t_0}^{t} \big( C_{A, \rho, \epsilon} + \beta C_{B} \epsilon + \frac{4 C_B}{\lambda_1 \gamma_0}  \|v(s)\|^2_U  \big) \, e^{\gamma (s-t_0)} \, \d s, \label{est: controlled eq basic 3}
\end{align}
where $\gamma>0$ will be determined later. Denote
    \begin{equation*}
        M_{\gamma}(t):= (2+\beta) \sqrt{\epsilon} \int_{t_0}^{t} e^{\gamma (s-t_0)} \|X^{\epsilon,v}_{s;t_0,\xi}\|_{H}^{\beta} \scalH{X^{\epsilon,v}_{s;t_0,\xi}, B(X^{\epsilon,v}_{s;t_0,\xi}) \, \d W_s}
    \end{equation*}
    and $\mathcal{M}_{\delta^{\prime},\gamma}(t):= M_{\gamma}(t) -\frac{\delta^{\prime}}{2} \<M_{\gamma} \>_{t}$. Using conditions {\rm [A4]}, we have
    \begin{align*}
        \d \< M_{\gamma}\>_t & \leq (2+\beta)^2 \epsilon \, e^{2 \gamma (t-t_0)} \|X^{\epsilon,v}_{s;t_0,\xi}\|_{H}^{2\beta}  \| \scalH{X^{\epsilon,v}_{t;t_0,\xi},B(X^{\epsilon,v}_{t;t_0,\xi})} \|_{L_2(U;\R)}^2 \, \d t  \\
        &  \leq (2+\beta)^2 C_{B} \, \epsilon  \, e^{\gamma (t_1-t_0)} \,   e^{\gamma (t-t_0)}    \| X^{\epsilon,v}_{t;t_0,\xi} \|_{H}^{2+\beta}  \d t,
    \end{align*}
    for any $t \in [t_0,t_1]$. For any $0<\delta^{\prime} <\frac{ 3(2+\beta)\lambda_1 \gamma_0- 8\gamma}{8 (2+\beta)^2 C_{B} \epsilon} e^{-\gamma (t_1-t_0)} $, estimate \eqref{est: controlled eq basic 3} yields
    \begin{align*}
        e^{\gamma (t-t_0)} \|X^{\epsilon,v}_{t;t_0,\xi}\|_{H}^{2+\beta} & \leq  \| \xi \|_{H}^{2+\beta} + \mathcal{M}_{2\delta^{\prime},\gamma}(t) \\
        &  \quad +  (1+\frac{\beta}{2})  \Big(\frac{C_{A,\rho,\epsilon}+ \beta C_B \epsilon }{\gamma} \big(e^{\gamma(t-t_0)} -1 \big) + \frac{4 C_B}{\lambda_1 \gamma_0} \int_{t_0}^{t} e^{\gamma (s-t_0)} \, \|v(s)\|^2_U  \,  \d s \Big),
    \end{align*}
for any $t \in [t_0,t_1]$. Let $\delta= e^{\gamma (t_1-t_0)} \, \delta^{\prime} < \frac{ 3(2+\beta)\lambda_1 \gamma_0- 8\gamma}{4 (2+\beta)^2 C_{B} \epsilon}$. By the supermartingale argument analogous to the proof of (iii) in Proposition \ref{prop: basic energy estimate beta}, we derive
\begin{align*}
        \E \exp \Big\{ \delta \sup_{t \in [t_0,t_1] } \Big\{ &  e^{-\gamma (t_1-t)}\| X^{\epsilon,v}_{t;t_0,\xi}\|_{H}^{2+\beta} - e^{-\gamma (t_1-t_0)} \| \xi \|_{H}^{2+\beta} - (1+\frac{\beta}{2}) \frac{C_{A,\rho,\epsilon}+ \beta C_B \epsilon }{\gamma} \frac{e^{\gamma (t-t_0)}-1}{e^{\gamma (t_1-t_0)}} \\
        & - (1+\frac{\beta}{2}) \frac{4 C_B}{\lambda_1 \gamma_0} \int_{t_0}^{t} e^{-\gamma (t-s)} \, \|v(s)\|^2_U  \,  \d s \Big\} \Big\} \leq 2.
\end{align*}
Taking $\gamma=\frac{\lambda_1 \gamma_0}{2} $, the above inequality yields
\begin{equation*}
   \E \exp \Big\{ \delta  \Big\{ \| X^{\epsilon,v}_{t_1;t_0,\xi}\|_{H}^{2+\beta} - (1+\frac{\beta}{2}) \frac{4 C_B}{\lambda_1 \gamma_0} \int_{t_0}^{t_1} \, \|v(s)\|^2_U  \,  \d s \Big\} \Big\} \leq 2 e^{\delta ( \| \xi \|_{H}^{2+\beta} + (2+\beta) \frac{C_{A,\rho,\epsilon}+ \beta C_B \epsilon }{\lambda_1 \gamma_0} )},
\end{equation*}
for any $0<\delta < \frac{ (2+3\beta)\lambda_1 \gamma_{0}}{8(2+\beta)^2 C_{B} \epsilon}$ and any $t_1 \geq t_0$. Using the fact $v \in \mathcal{A}_{M}$, we get estimate \eqref{est: energy controlled eq beta H}. Due to $X^{\epsilon,v}_{t;t_0,\xi}= X^{\epsilon,v}_{t;t_1,X^{\epsilon,v}_{t_1;t_0,\xi}}$, analogous to estimate \eqref{est: energy controlled eq beta}, we have
    \begin{align*}
        \E \Big[ \exp \Big\{  \delta \sup_{t \geq t_1} &  \big( \frac{\gamma_0}{2} \int_{t_1}^t \| X^{\epsilon,v}_{s;t_0,\xi} \|_{H}^{\beta} \| X^{\epsilon,v}_{s;t_0,\xi} \|_{V}^{2} \, \d s  - \| X^{\epsilon,v}_{t_1;t_0,\xi} \|_{H}^{2+\beta}  \nonumber \\
       &  \quad - (1+\frac{\beta}{2}) \frac{4 C_B}{\lambda_1 \gamma_0} \int^t_{t_1} \|v(s)\|^2_U \, \d s  - (1+\frac{\beta}{2}) \big( C_{A,\rho,\epsilon}+ \beta C_B \epsilon \big) \, (t-t_1) \big) \Big\} \Big] \leq  2,
    \end{align*}
for any $0<\delta < \frac{ (2+3\beta)\lambda_1 \gamma_{0}}{8(2+\beta)^2 C_{B} \epsilon}$ and any $t_1 \geq t_0$. Combining the above exponential estimate of $\| X^{\epsilon,v}_{t_1;t_0,\xi} \|_{H}^{\beta+2}$ and using H\"older inequality, we obtain
\begin{align*}
        & \E \Big[ \exp \Big\{  \frac{\delta}{2} \sup_{t \geq t_1} \big( \frac{\gamma_0}{2} \int_{t_1}^t \| X^{\epsilon,v}_{s;t_0,\xi} \|_{H}^{\beta} \| X^{\epsilon,v}_{s;t_0,\xi} \|_{V}^{2} \, \d s- (1+\frac{\beta}{2}) \big( C_{A,\rho,\epsilon}+ \beta C_B \epsilon \big) (t-t_1) \big) \Big\} \Big]  \\
        &  \leq  \Big( \E \Big[ \exp \Big\{  \delta \sup_{t \geq t_1} \big( \frac{\gamma_0}{2} \int_{t_1}^t \| X^{\epsilon,v}_{s;t_0,\xi} \|_{H}^{\beta} \| X^{\epsilon,v}_{s;t_0,\xi} \|_{V}^{2} \, \d s - \| X^{\epsilon,v}_{t_1;t_0,\xi} \|_{H}^{2+\beta} -(1+\frac{\beta}{2}) \big( C_{A,\rho,\epsilon}+ \beta C_B \epsilon \big) (t-t_1) \big) \Big\} \Big] \Big)^{1/2} \\
        & \quad \quad \times \Big(  \E \exp \Big\{ \delta \big\{ \| X^{\epsilon,v}_{t_1;t_0,\xi}\|_{H}^{2+\beta} \big\} \Big\} \Big)^{1/2}  \leq 2 \, e^{\delta  \big( \frac{\| \xi \|_{H}^{2+\beta}}{2} + (1+\frac{\beta}{2}) \frac{ C_{A,\rho,\epsilon} + 2C_B M+  \beta C_{B} \epsilon  }{\lambda_1 \gamma_0} \big)}.
\end{align*}
    Thus, we obtain the result \eqref{est: energy controlled eq beta V} and the corresponding  ``$\beta=0$'' version uniform estimates.
\end{proof}

Next, we will prove the estimate of the difference between solutions $X^{\epsilon,v}_{t;t_0,\xi_1}$ and $X^{\epsilon,v}_{t;t_0,\xi_2}$ starting from different initial values, which is uniform with $\epsilon > 0$ and $v \in \mathcal{A}_{M}$.
\begin{lemma} \label{lem: controlled eq diff}
Let $ M>0$, $\epsilon_0>0$ and $v \in \mathcal{A}_M$. Assume that the conditions {\rm [A1-A5]} hold for any $0 \leq \epsilon < \epsilon_0$, then for any two $\mathcal{F}_{t_0}$ measurable initial datas $\xi_1,\xi_2\in H$, the difference between solutions $X^{\epsilon,v}_{t;t_0,\xi_1}$ and $X^{\epsilon,v}_{t;t_0,\xi_2}$ to Eq. \eqref{eq: controlled equation truncated} satisfies the uniform estimate
\begin{align}
   & \max \Big\{ \E \big[\sup_{s\in [t_0,t]} e^{- \gamma (t-s)} \| X^{\epsilon,v}_{s;t_0,\xi_1}-X^{\epsilon,v}_{s;t_0,\xi_2} \|_{H}^2 \big]^{\frac{1}{2}}, \E \big[ (\gamma_0 -\tilde{\gamma}_0) \int_{t_0}^{t} e^{- \gamma (t-s)}   \| X^{\epsilon,v}_{s;t_0,\xi_1}-X^{\epsilon,v}_{s;t_0,\xi_2} \|_{V}^2 \, \d s \big]^{\frac{1}{2}} \Big\}  \nonumber \\
    & \leq 2 (\E \| \xi_1-\xi_2 \|_{H}^2 )^{\frac{1}{2}}  e^{\frac{  C_{\rho_1}}{\gamma_0} \big( \| \xi_2 \|_{H}^{2} + \| \xi_2 \|_{H}^{2+\beta} + \big(2+\frac{\beta}{2} \big) \frac{ 2 C_{A,\rho} + 4C_B M+ \beta C_{B} \tilde{\epsilon}  }{\lambda_1 \gamma_0} \big)+ \frac{C_B M}{\lambda_1(\gamma_0-\tilde{\gamma}_0)}} e^{-\frac{\gamma }{2}(t-t_0)}, \label{est: diff controlled eq}
\end{align}
for any $\epsilon<\frac{\tilde{\epsilon}}{4} \vee \epsilon_0$ and $\gamma \in [0,\frac{\lambda_1 (\gamma_0-\tilde{\gamma}_0) }{2}]$.
\end{lemma}
\begin{proof}
  For fixed $t_0\in \mathbb{R}$, $t\geq t_0$, for any $\xi_1,\xi_2\in H$, let $\eta(t)=X^{\epsilon,v}_{t;t_0,\xi_1}-X^{\epsilon,v}_{t;t_0,\xi_2}$ and define
\begin{align*}
      \beta_{\gamma,v}(t)  := & \gamma (t-t_0)  - \frac{2 C_{\rho_1}}{\gamma_0} \Big( \frac{\gamma_0}{2} \int^t_{t_0} (1+\|X^{\epsilon,v}_{s;t_0,\xi_2}\|^{\beta}_H) \|X^{\epsilon,v}_{s;t_0,\xi_2}\|^2_V \, \d s \\
      &  \quad - \frac{1}{2} \big((4+\beta) C_{A,\rho,\epsilon} + \beta (2+\beta) C_B \epsilon) \big)  (t-t_0) \Big) - \zeta  \int^t_{t_0} \|v(s)\|^2_U \, \d s, \quad \forall \, t \geq t_0,
\end{align*}
where $\zeta>0$ will be determined later. Using It\^o formula, we have
    \begin{align}
        & e^{\beta_{\gamma,v}(t) } \| \eta(t) \|_{H}^2 - \| \eta(t_0) \|_{H}^2 -2 \sqrt{\epsilon} \int_{t_0}^{t} e^{\beta_{\gamma,v}(s)} \scalH{\eta(s), B(X^{\epsilon,v}_{s;t_0,\xi_1})-B(X^{\epsilon,v}_{s;t_0,\xi_2}) \, \d W_s} \nonumber \\
        &  = \Big(\gamma+\frac{ C_{\rho_1} ( (4+\beta)  C_{A,\rho,\epsilon} + \beta (2+\beta) C_B \epsilon)  }{\gamma_0} +C_{\rho_2} \Big) \int_{t_0}^{t}  e^{\beta_{\gamma,v}(s)} \|\eta(s)\|_{H}^{2} \, \d s \nonumber \\
        & \quad -   \int_{t_0}^{t}  e^{\beta_{\gamma,v}(s)} \big( C_{\rho_1} \big(  1+ \|X^{\epsilon,v}_{s;t_0,\xi}\|_{H}^{\beta} \big) \|X^{\epsilon,v}_{s;t_0,\xi}\|_{V}^{2} \big) +C_{\rho_2} + \zeta \| v(s) \|_{U}^2  \big) \|\eta(s)\|_H^2  \, \d s \nonumber \\
        & \quad  + \int_{-n}^{t} e^{\beta_{\gamma,v}(s)} \big( 2 \, \scalV{A^{\epsilon}(X^{\epsilon,v}_{s;t_0,\xi})-A^{\epsilon}(X^{\epsilon,v}_{s;t_0,\xi}),\eta(s)}  + \epsilon \| B(X^{\epsilon,v}_{s;t_0,\xi})-B(X^{\epsilon,v}_{s;t_0,\xi})  \|_{2}^2  \big) \, \d s \nonumber \\
        & \quad +\int^t_{t_0} e^{\beta_{\gamma,v}(s)} 2 \scalH{ (B(X^{\epsilon,v}_{s;t_0,\xi_1})-B(X^{\epsilon,v}_{s;t_0,\xi_2}))v,\eta(s) } \, \d s. \nonumber
\end{align}
The condition {\rm [A4]} and Cauchy--Schwarz inequality give
\begin{align*}
  &2\int^t_{t_0}e^{\beta_{\gamma,v}(s)}\scalH{ (B(X^{\epsilon,v}_{s;t_0,\xi_1})-B(X^{\epsilon,v}_{s;t_0,\xi_2}))v,\eta(s) } \, \d s\\
& \leq 2 \int^t_{t_0}e^{\beta_{\gamma,v}(s)} \big(\sum_{k} \big| \scalH{ (B(X^{\epsilon,v}_{s;t_0,\xi_1})-B(X^{\epsilon,v}_{s;t_0,\xi_2})) \, \mathcal{U}_{k},\eta(s)} \big|^2 \big)^{\frac{1}{2}}\|v(s)\|_U  \, \d s \\
& \leq \frac{C_B}{\zeta}\int^t_{t_0}e^{\beta_{\gamma,v}(s)}\| \eta(s)\|^2_H \, \d s+\zeta \int^t_{t_0}e^{\beta_{\gamma,v}(s)}\| \eta(s)\|^2_H\|v(s)\|^2_U \, \d s.
\end{align*}
By Burkh\"older--Davis--Gundy inequality and condition {\rm [A4]}, we obtain
    \begin{align*}
        & 2 \sqrt{\epsilon} \, \E \sup_{s\in [t_0,t]} \Big| \int_{t_0}^{s}  e^{\beta_{\gamma,v}(\tau)} \scalH{\eta(\tau), B(X^{\epsilon,v}_{r;t_0,\xi})-B(X^{\epsilon,v}_{\tau;t_0,\xi}) \, \d W_{\tau}} \Big| \\
        & \leq 6 \sqrt{ \epsilon} \, \E \Big| \int_{-n}^{t}  e^{2 \beta_{\gamma,v}(s)} \sum_{k} \Big| \bigscalH{\eta(s), \big( B(X^{\epsilon,v}_{s;t_0,\xi})-B(X^{\epsilon,v}_{s;t_0,\xi}) \big) \, \mathcal{U}_{k}} \Big|^2 \, \d s  \Big|^{1/2} \\
        &  \leq \frac{1}{2} \E \sup_{s\in [t_0,t]}  e^{\beta_{\gamma,v}(s)} \| \eta(s) \|_{H}^2  + 18 C_B  \epsilon \, \E \int_{-n}^{t} e^{\beta_{\gamma,v}(s)} \| \eta(s) \|_{H}^2 \, \d s.
\end{align*}
Combining the above estimates, condition {\rm [A2]} and embeddeding inequality \eqref{eq: embedding}, we have
\begin{align*}
 &  \E \Big[ \sup_{s\in [t_0,t]} \big\{ e^{\beta_{\gamma,v}(s)} \| \eta(s) \|_{H}^2 +  \Delta_{\gamma} \int_{t_0}^{s}  e^{\beta_{\gamma,v}(\tau)}  \|\eta(\tau)\|_{V}^{2} \, \d \tau \big\}  \Big]  \\
& \leq  \E \| \eta(-n) \|_{H}^2 + \frac{1}{2} \E \sup_{s\in [t_0,t]}  e^{\beta_{\gamma,v}(s)} \| \eta(s) \|_{H}^2 +  \E \Big[ \sup_{s\in [t_0,t]}  \Big\{ \int_{t_0}^{s}  e^{\beta_{\gamma,v}(\tau)}  \|\eta(\tau)\|_{H}^{2} \, \d \tau  \\
& \quad \quad \times \Big(\gamma+\frac{ C_{\rho_1} ( (4+\beta)  C_{A,\rho,\epsilon} + \beta (2+\beta) C_B \epsilon)  }{\gamma_0} +C_{\rho_2} + 18 C_B \epsilon + \frac{C_B}{\zeta} +  \lambda_1 \big(\Delta_{\gamma}  -2  \gamma_0 \big) \Big) \Big\}  \Big],
\end{align*}
where $0<\Delta_{\gamma}<2 \gamma_0$. Recall that $\tilde{\gamma}_0= \frac{C_{\rho_2}}{\lambda_1} \vee \big( \frac{(4+\beta) C_{\rho_1} C_{A,\rho,\epsilon}}{\lambda_1} \big)^{1/2}$ and $\tilde{\epsilon} \leq \frac{\gamma_0}{(18 \gamma_0 +  \beta  (2+\beta)  C_{\rho_1} )C_B} ( 2 \lambda_1 \gamma_0-\frac{  (4+\beta) C_{\rho_1}  C_{A,\rho,\epsilon}   }{\gamma_0} -C_{\rho_2} \big)$. When $0< \epsilon < \frac{\tilde{\epsilon}}{4} $ and $\gamma_0 > \tilde{\gamma}_0$, taking $\zeta=\frac{2C_B}{\lambda_1 (\gamma_0-\tilde{\gamma}_0)}$ and $\gamma \leq \lambda_1 \Delta_{\gamma}=\frac{\lambda_1 (\gamma_0-\tilde{\gamma}_0)}{2}$, we get the last term of the above inequality is non-positive. Thus, we have
\begin{equation}\label{eq: diff weighting controll}
\max\Big\{ \frac{1}{2} \E \big[ \sup_{s\in [t_0,t]} \big\{ e^{\beta_{\gamma,v}(s)} \| \eta(s) \|_{H}^2 \big], \Delta_{\gamma} \E \Big[  \int_{t_0}^{t}  e^{\beta_{\gamma,v}(s)}  \|\eta(s)\|_{V}^{2} \, \d s \big\}  \Big] \Big\} \leq  \E \| \eta(t_0) \|_{H}^2.
\end{equation}

Recall that $\tilde{\epsilon} \leq \frac{ (1+\beta)\lambda_1 \gamma_0^2}{8 (2+\beta)^2  C_{\rho_1} C_{B}}$. Due to $0< \epsilon < \frac{\tilde{\epsilon}}{4} $, we can take $\frac{ 4 C_{\rho_1}}{\gamma_0} \leq \delta \leq \frac{ (1+\beta)\lambda_1 \gamma_{0}}{(2+\beta)^2 C_{B} \tilde{\epsilon}}  < \frac{ (1+\beta)\lambda_1 \gamma_{0}}{4(2+\beta)^2 C_{B} \epsilon}$ and $2 \geq p_0=\frac{4 C_{\rho_1}}{\gamma_0 \delta}+1> \frac{4 (2+\beta)^2 C_{\rho_1} C_{B} \tilde{\epsilon} }{(1+\beta) \lambda_1 \gamma_0^2}+1$, then $ \frac{1}{p_0-1} \frac{2 C_{\rho_1}}{\gamma_0} =\frac{\delta}{2}$. Applying estimate \eqref{est: energy controlled eq beta V} and \eqref{eq: diff weighting controll} and using Cauchy-Schwarz inequality, we obtain
    \begin{align*}
        & \E \big[\sup_{s\in [t_0,t]} e^{\gamma (s-t_0)} \| \eta(s) \|_{H}^2 \big]^{\frac{1}{p_0}} \leq \big( \E \sup_{s\in [t_0,t]} e^{\beta_{\gamma,v}(s)} \| \eta(s) \|_{H}^2 \big)^{\frac{1}{p_0}} \times \E \exp \big\{\frac{\zeta}{p_0}  \int^t_{t_0} \|v(s)\|^2_U \, \d s \big\} \\
        & \times \Big( \E \Big[ \exp \Big\{  \frac{1}{p_0-1} \frac{2  C_{\rho_1}}{\gamma_0}  \sup_{s \geq t_0}  \big( \frac{\gamma_0}{2} \int_{t_0}^{s} \|X^{\epsilon}_{r;t_0,\xi_2}\|_{H}^{\beta} \|X^{\epsilon}_{r;t_0,\xi_2}\|_{V}^{2}  \, \d r -(1+\frac{\beta}{2}) \big( C_{A, \rho,\epsilon}+ \beta C_B \epsilon \big) (s-t_0) \big) \Big\} \Big] \Big) ^{\frac{p_0-1}{p_0}}  \\
        & \times \Big( \E \Big[ \exp \Big\{  \frac{1}{p_0-1} \frac{2  C_{\rho_1}}{\gamma_0}  \sup_{s \geq t_0}  \big( \frac{\gamma_0}{2} \int_{t_0}^{s}  \|X^{\epsilon}_{r;t_0,\xi_2}\|_{V}^{2}  \, \d r -  C_{A, \rho, \epsilon}  (s-t_0) \big) \Big\} \Big] \Big) ^{\frac{p_0-1}{p_0}}  \\
        & \quad \leq 2 (\E \| \eta(t_0) \|_{H}^2 )^{\frac{1}{p_0}}  e^{\frac{ 2  C_{\rho_1}}{\gamma_0 p_0} \big( \| \xi_2 \|_{H}^{2} + \| \xi_2 \|_{H}^{2+\beta} + \big(2+\frac{\beta}{2} \big) \frac{ 2 C_{A,\rho,\epsilon} + 4C_B M+ \beta C_{B} \epsilon  }{\lambda_1 \gamma_0} \big)+ \frac{2C_B M}{\lambda_1(\gamma_0-\tilde{\gamma}_0)p_0}} .
\end{align*}
Taking $p_0=2$ and dividing $e^{\gamma (t-t_0)/2}$ in both side of the above inequality, we obtain
\begin{align*}
   & \E \big[\sup_{s\in [t_0,t]} e^{- \gamma (t-s)} \| X^{\epsilon,v}_{s;t_0,\xi_1}-X^{\epsilon,v}_{s;t_0,\xi_2} \|_{H}^2 \big]^{\frac{1}{2}} \\
    & \leq 2 (\E \| \xi_1-\xi_2 \|_{H}^2 )^{\frac{1}{2}}  e^{\frac{  C_{\rho_1}}{\gamma_0} \big( \| \xi_2 \|_{H}^{2} + \| \xi_2 \|_{H}^{2+\beta} + \big(2+\frac{\beta}{2} \big) \frac{ 2 C_{A,\rho} + 4C_B M+ \beta C_{B} \tilde{\epsilon}  }{\lambda_1 \gamma_0} \big)+ \frac{C_B M}{\lambda_1(\gamma_0-\tilde{\gamma}_0)}} e^{-\frac{\gamma}{2} (t-t_0)} .
\end{align*}
Similar to the above estimate, we also have
\begin{align*}
   & \E \big[ (\gamma_0 -\tilde{\gamma}_0) \int_{t_0}^{t} e^{- \gamma (t-s)}   \| X^{\epsilon,v}_{s;t_0,\xi_1}-X^{\epsilon,v}_{s;t_0,\xi_2} \|_{V}^2 \, \d s \big]^{\frac{1}{2}} \\
    & \leq 2 (\E \| \xi_1-\xi_2 \|_{H}^2 )^{\frac{1}{2}}  e^{\frac{  C_{\rho_1}}{\gamma_0} \big( \| \xi_2 \|_{H}^{2} + \| \xi_2 \|_{H}^{2+\beta} + \big(2+\frac{\beta}{2} \big) \frac{ 2 C_{A,\rho} + 4C_B M+ \beta C_{B} \tilde{\epsilon}  }{\lambda_1 \gamma_0} \big)+ \frac{C_B M}{\lambda_1(\gamma_0-\tilde{\gamma}_0)}} e^{-\frac{\gamma }{2}(t-t_0)} .
\end{align*}
Thus, we obtain estimate \eqref{est: diff controlled eq}.
\end{proof}

With the help of Lemma \ref{lem: good rate func} and Lemma \ref{lem: controlled eq diff}, we can verify the above condition (S1), thereby obtaining the uniform LDP on compact sets.
\begin{lemma}\label{lem: S1 condition}
  Assume that conditions {\rm [A1-A5]} are in force, then the condition (S1) holds.
\end{lemma}
\begin{proof}
  Let $\{v^{\epsilon}\}_{\epsilon} \subset \mathcal{A}_{M}$ weakly converge to $v$ in distribution and $\{ \xi^{\epsilon} \}_{\epsilon} \subset H$ converges to $\xi$ in $H$ strongly as $\epsilon\rightarrow 0$. Denote $Z_1^{\epsilon}:=X^{\epsilon,v^{\epsilon}}_{t;t_0,\xi^{\epsilon}}-X^{\epsilon,v^{\epsilon}}_{t;t_0,\xi}$, $Z_2^{\epsilon}:=X^{\epsilon,v^{\epsilon}}_{t;t_0,\xi}-X^{0,v^{\epsilon}}_{t;t_0,\xi}$ and $Z_3^{\epsilon}:=X^{0,v^{\epsilon}}_{t;t_0,\xi}-X^{0,v}_{t;t_0,\xi}$. Due to
\begin{equation*}
\mathcal{G}^{\epsilon}_{t_0,\xi^{\epsilon}}\Big(W(\cdot)+\frac{1}{\sqrt{\epsilon}}\int^{\cdot}_0v^{\epsilon}(s) \, \d s\Big) - \mathcal{G}^{0}_{t_0,\xi}\Big(\int^{\cdot}_0v(s) \,  \d s\Big) =  X^{\epsilon,v^{\epsilon}}_{t;t_0,\xi^{\epsilon}}- X^{0,v}_{t;t_0,\xi}= Z_1^{\epsilon} + Z_2^{\epsilon} +Z_3^{\epsilon} ,
\end{equation*}
to prove the condition (S1), it suffices to show both $Z_{1}^{\epsilon}$, $Z_2^{\epsilon}$ and $Z_3^{\epsilon}$ converge to $0$ in distribution as $\epsilon\rightarrow 0$ in $\sC_{t_0}$. Clearly, it follows from Lemma \ref{lem: good rate func} that $\sup_{t\in [t_0,T]}\|Z_3^{\epsilon}(t)\|^2_H+\gamma_0\int^T_{t_0}\|Z_3^{\epsilon}(s)\|^{2}_V \, \d s$ converges to $0$  as $\epsilon\rightarrow 0$. On the other hand, by Lemma \ref{lem: controlled eq diff}, for any fixed time $T>t_0$, we have
  \begin{equation*}
    \lim_{\epsilon \rightarrow 0} \E \big[ \sup_{t\in [t_0,T]}\|Z_1^{\epsilon}(t)\|^2_H+\gamma_0\int^T_{t_0}\|Z_1^{\epsilon}(s)\|^{2}_V \, \d s \big]^{1/2} =0.
  \end{equation*}
Thus, it remains to prove that $Z_2^{\epsilon}$ converges to $0$ in distribution as $\epsilon\rightarrow 0$ in $\sC_{t_0}$. If $A^{\epsilon} \equiv A^{0}$, then this convergence is given by \cite{Liu_Large_2020,Pan_Large_2026}. If $A^{\epsilon} \neq A^{0}$, an additional term
\begin{equation*}
  2 \scalV{A^{\epsilon}(X^{0,v^{\epsilon}}_{t;t_0,\xi})-A^{0}(X^{0,v^{\epsilon}}_{t;t_0,\xi}),X^{\epsilon,v^{\epsilon}}_{t;t_0,\xi}-X^{0,v^{\epsilon}}_{t;t_0,\xi}}
\end{equation*}
will appear in the estimate of $Z_2^{\epsilon}$. Using the condition {\rm [A3]} and the uniform estimate of $X^{\epsilon,v^{\epsilon}}$ given in Lemma \ref{lem: skeleton}, the above term is uniformly less than $O(\epsilon) \| Z_2^{\epsilon} \|_V$. Thus, the convergence of $Z_2^{\epsilon}$ also holds and condition (S1) is verified. We get the desired result.
\end{proof}

Now, we state the result of uniform LDP on compact sets for the solution of Eq. \eqref{eq: SPDE}.
\begin{theorem}\label{thm: ULDP}
Under conditions {\rm [A1-A5]}, if $\gamma_0  \geq  \tilde{\gamma}_0$, for each $t_0\in \mathbb{R}$ and any $T>t_0$, the solution $X^{\epsilon}_{\cdot;t_0,\xi}$ of Eq. \eqref{eq: SPDE} satisfies the uniform LDP on the space $C([t_0,T];H)\cap L^2([t_0,T];V)$ with the rate function $\mathcal{I}_{t_0,\xi}$ defined by \eqref{def: rate func 1}.
\end{theorem}
\begin{proof}
Based on Proposition \ref{prp-1}, we only need to verify conditions (S1)-(S2). The condition (S2) is verified by Lemma \ref{lem: lsc uniform LDP}. Moreover, Lemma \ref{lem: S1 condition} implies the condition (S1) holds. We complete the proof.
\end{proof}

\section{LDP for stationary solutions and invariant measures} \label{sec: LDP for stationary solution}

Analogous to the proof of Theorem \ref{thm: stationary solution}, using  Lemma \ref{lem: well-posedness controlled eq} and \ref{lem: controlled eq diff}, we can obtain that for any $v^{\epsilon} \in \mathcal{A}_{M}$ and for any $N \in \N$, the sequence of solutions $\{X^{\epsilon,v^\epsilon}_{\cdot;-n,\xi}\}_{n \geq N}$ to Eq. \eqref{eq: controlled equation truncated} is a Cauchy sequence in $L^{1}(\Omega;\sC_{-N})$, and the limit element $\cY^{\epsilon,v^{\epsilon}} \in L^{\infty}(\R; L^{\infty}(\Omega;H))$ satisfies
\begin{align}
   \scalV{\cY^{\epsilon,v^{\epsilon}}(t),u} & = \scalV{ \cY^{\epsilon}(t_0),u} + \int_{t_0}^{t} \scalV{A^{\epsilon}(\cY^{\epsilon,v^{\epsilon}}(s)),u} \, \d s \nonumber \\
   & \quad + \int_{t_0}^{t} \scalH{B(\cY^{\epsilon,v^{\epsilon}}(s)) v^{\epsilon}(s),u} \, \d  s  + \sqrt{\epsilon} \int_{t_0}^{t} \scalH{B(\cY^{\epsilon,v^{\epsilon}}(s)) \, \d W_s,u}, \label{eq: stationary controlled}
\end{align}
for any $u \in V$ and $-\infty<t_0<t<+\infty$. By Girsanov theorem, we know that $\cY^{\epsilon,v^{\epsilon}} = \mathcal{G}_{-\infty}^{\epsilon} (W (\cdot) +\frac{1}{\sqrt{\epsilon}}\int^{\cdot}_0v^{\epsilon}(s) \, \d s)$. Analogous to the proof of Theorem \ref{thm: stationary solution}, $\cY^{\epsilon,v^{\epsilon}}$ is the unique solution satisfying Eq. \eqref{eq: stationary controlled} in the space $L^{\infty}(\R; L^{\infty}(\Omega;H))$. According to \cite[Theorem 5]{Budhiraja_Large_2008}, the following conditions are sufficient
to ensure the LDP for stationary solution $\cY^{\epsilon}$.
\begin{proposition}\label{prp-2}
The stationary solution $\cY^{\epsilon}$ satisfeis LDP on the space $\sC_{-\infty}$ with rate function $\mathcal{I}_{-\infty}$ given by (\ref{def: rate func 2}), if the following items hold:
 \begin{description}
   \item[(C1)] For any $M>0$, let $\{v^{\epsilon}: \epsilon>0\}\subset \mathcal{A}_M$, if $v^{\epsilon}$ converges to $v$ in distribution as $S_M-$valued random elements, then
\begin{equation*}
  \mathcal{G}^{\epsilon}_{-\infty}\Big(W(\cdot)+\frac{1}{\sqrt{\epsilon}}\int^{\cdot}_0v^{\epsilon}(s) \, \d s\Big) \rightarrow \mathcal{G}^{0}_{-\infty}\Big(\int^{\cdot}_0v(s) \,  \d s\Big) \quad \mbox{on} \quad \sC_{-\infty},
 \end{equation*}
in distribution as $\epsilon\rightarrow 0$.
  \item[(C2)] For each $M>0$, the set $\Gamma_{M; -\infty}:=\{\mathcal{G}^{0}_{-\infty}(\int^{\cdot}_0v(s) \,  \d s)\in\sC_{-\infty}: v\in S_M\}$ is compact.
 \end{description}
\end{proposition}

Based on Lemma \ref{lem: good rate func}, Lemma \ref{lem: controlled eq diff} and  Proposition \ref{prp-2}, we achieve the following result.
\begin{theorem}\label{thm: LDP for stationary solution}
 Assume that the conditions {\rm [A1-A5]} hold and $ \gamma_0 \geq  \tilde{\gamma}_0$. The stationary solution $\{\cY^{\epsilon}\}_{\epsilon>0}$ of  \eqref{eq: SPDE} satisfies LDP on $\sC_{-\infty}$ with the good rate function $\mathcal{I}_{-\infty}$ given by (\ref{def: rate func 2}).

 \end{theorem}
\begin{proof}
 Clearly, (C2) is verified by
Lemma \ref{lem: good rate func}. We focus on the proof of (C1).

For any bounded continuous function $F$ from $\sC_{-\infty}$ to $\R$, we have
\begin{align}\notag
    & F \Big( \mathcal{G}_{-\infty}^{\epsilon} \Big(W(\cdot)+\frac{1}{\sqrt{\epsilon}} \int_{0}^{\cdot} v^{\epsilon}(s) \, \d s\Big) \Big) - F \Big( \mathcal{G}_{-\infty}^{0} \Big(\int_{0}^{\cdot} v(s) \, \d s\Big)  \Big) \\
    \notag
    & =  \bigg( F \Big( \mathcal{G}_{-\infty}^{\epsilon} \Big(W(\cdot)+\frac{1}{\sqrt{\epsilon}} \int_{0}^{\cdot} v^{\epsilon}(s) \, \d s\Big) \Big) - F \Big(  \mathcal{G}^{\epsilon}_{-n,\xi}\Big(W(\cdot)+\frac{1}{\sqrt{\epsilon}}\int^{\cdot}_0v^{\epsilon}(s) \, \d s\Big)  \Big) \bigg) \\
    \notag
    & \quad + \bigg( F \Big(  \mathcal{G}^{\epsilon}_{-n,\xi}\Big(W(\cdot)+\frac{1}{\sqrt{\epsilon}}\int^{\cdot}_0v^{\epsilon}(s) \, \d s\Big) \Big) - F \Big( \mathcal{G}_{_{-n,\xi}}^{0} \Big(\int_{0}^{\cdot} v(s) \, \d s\Big)  \Big) \bigg) \\
    & \quad + \bigg( F \Big( \mathcal{G}_{_{-n,\xi}}^{0} \Big(\int_{0}^{\cdot} v(s) \, \d s\Big)  \Big) - F \Big( \mathcal{G}_{-\infty}^{0} \Big(\int_{0}^{\cdot} v(s) \, \d s\Big)  \Big)  \bigg)     \label{eq: split}
\end{align}
We denote the three terms of the right hands of the above identity as $I_1$, $I_2$ and $I_3$, respectively. Using Lemma \ref{lem: skeleton}, we have $\lim_{n \rightarrow +\infty} I_3=0$.

For any $t_0 \in \R$, denote $F_{t_0} (\Phi):= F(\Phi(t) \ind_{t \geq t_0 } + \Phi(t_0) \ind_{t <t_0}) $ for any $\Phi \in \sC_{-\infty}$, then $F_{t_0}$ is a bounded continuous function from $\sC_{t_0}$ to $\R$. Using the condition (S1), which has already been verified in Theorem \ref{thm: ULDP}, it holds
\begin{align*}
 \lim_{\epsilon \rightarrow 0} I_2 & = \lim_{\epsilon \rightarrow 0} \bigg| F \Big(  \mathcal{G}^{\epsilon}_{-n,\xi}\Big(W(\cdot)+\frac{1}{\sqrt{\epsilon}}\int^{\cdot}_0v^{\epsilon}(s) \, \d s\Big) \Big) - F \Big( \mathcal{G}_{_{-n,\xi}}^{0} \Big(\int_{0}^{\cdot} v(s) \, \d s\Big)  \Big) \bigg| \\
  & = \lim_{\epsilon \rightarrow 0} \bigg| F_{-n} \Big(  \mathcal{G}^{\epsilon}_{-n,\xi}\Big(W(\cdot)+\frac{1}{\sqrt{\epsilon}}\int^{\cdot}_0v^{\epsilon}(s) \, \d s\Big) \Big) - F_{-n} \Big( \mathcal{G}_{_{-n,\xi}}^{0} \Big(\int_{0}^{\cdot} v(s) \, \d s\Big)  \Big) \bigg|=0.
\end{align*}

As we mentioned in the beginning of this section, $\cY^{\epsilon,v^{\epsilon}}$ satisfies Eq. \eqref{eq: stationary controlled}. Thus, $\mathcal{Y}^{\epsilon,v^{\epsilon}}(t)=X^{\epsilon,v^{\epsilon}}_{t;-n,\mathcal{Y}^{\epsilon,v^{\epsilon}}(-n)}$. Using Lemma \ref{lem: controlled eq diff}, we have
\begin{align*}
   & \max \Big\{ \E \big[\sup_{s\in [-n,t]} e^{- \gamma (t-s)} \| X^{\epsilon,v^{\epsilon}}_{s;-n,\xi}-\mathcal{Y}^{\epsilon,v^{\epsilon}}(s)\|_{H}^2 \big]^{\frac{1}{2}}, \E \big[ (\gamma_0 -\tilde{\gamma}_0) \int_{-n}^{t} e^{- \gamma (t-s)}   \| X^{\epsilon,v^{\epsilon}}_{s;-n,\xi}-\mathcal{Y}^{\epsilon,v^{\epsilon}}(s) \|_{V}^2 \, \d s \big]^{\frac{1}{2}} \Big\}  \nonumber \\
    & \leq 2 (\E \| \xi- \mathcal{Y}^{\epsilon,v^{\epsilon}}(-n) \|_{H}^2 )^{\frac{1}{2}}  e^{\frac{  C_{\rho_1}}{\gamma_0} \big( \| \xi \|_{H}^{2} + \| \xi \|_{H}^{2+\beta} + \big(2+\frac{\beta}{2} \big) \frac{ 2 C_{A,\rho} + 4C_B M+ \beta C_{B} \tilde{\epsilon}  }{\lambda_1 \gamma_0} \big)+ \frac{C_B M}{\lambda_1(\gamma_0-\tilde{\gamma}_0)}} e^{-\frac{\gamma }{2}(t+n)},
\end{align*}
for any $\epsilon<\frac{\tilde{\epsilon}}{4} \vee \epsilon_0$ and $\gamma \in [0,\frac{\lambda_1 (\gamma_0-\tilde{\gamma}_0) }{2}]$. Using estimate \eqref{est: energy controlled eq beta H}, we know that the limit element $\cY^{\epsilon,v^{\epsilon}}$ of the Cauchy sequence $\{X^{\epsilon,v^\epsilon}_{\cdot;-n,\xi}\}_{n \geq N}$ is uniformly bounded, i.e. $\sup_{\epsilon} \sup_{t \in \R} \E \| \cY^{\epsilon,v^{\epsilon}} (t) \|_{H}^2<+\infty$. Thus, we obtain that $\lim_{n \rightarrow +\infty} \sup_{\epsilon} I_1=0$.

Combining the above estimates, the condition (C1) is verified. Thus, we obtain LDP for stationary solutions $\{ \cY^{\epsilon}\}_{\epsilon>0}$.
\end{proof}

Denote by $\nu^{\epsilon}:=\mathbb{P}\circ (\cY^{\epsilon}(0))^{-1}$. From the definition of stationary solution, we know the fact $\mathcal{Y}^{\epsilon}(t)=X^{\epsilon}_{t;0,\cY^{\epsilon}(0)}$ for any $t \geq 0$. Thus, we have
\begin{align*}
 \big( P_{t} \, \nu^{\epsilon} \big) (A)=\int_{H}\mathbb{P}(X^{\epsilon}_{t;0,\zeta}\in A) \, \nu^{\epsilon}(d\zeta)= \mathbb{P}(\cY^{\epsilon}(0)\in A)=\nu^{\epsilon}(A) ,
\end{align*}
for any Borel set $A\in \mathcal{B}(H)$ and $t\geq 0$. Thus, $\nu^{\epsilon}$ is an invariant measure of the solution map of \eqref{eq: SPDE}. \cite{Arnold_Random_2003} shows every invariant measure corresponds to a stationary solution, but it is
defined in an extended probability space. Thus, the uniqueness of the stationary solutions ensures the uniqueness of the invariant measure. Then $\nu^{\epsilon}$ is ergodic measure.

\begin{theorem}\label{thm: LDP for invariant measures}
Assume that the conditions {\rm [A1-A5]} hold and $\gamma \geq  \tilde{\gamma}_0$. The family of invariant measure $\{\nu^{\epsilon}\}_{\epsilon>0}$ of  \eqref{eq: SPDE} satisfies the LDP on $H$ with the good rate function $\mathcal{I}^{\prime} $ given by \eqref{rrr-12}.
\end{theorem}
\begin{proof}
  The result is a consequence of the contraction principle and Theorem \ref{thm: LDP for stationary solution}.
\end{proof}

\section{Application to examples}\label{sec: application}
For the reader's convenience, we first recall parameter $\tilde{\gamma}_0$ used in Theorems \ref{thm: LDP for stationary solution intro} and \ref{thm: LDP for invariant measures intro}. It is defined in \eqref{def: gamma} and takes the form:
\begin{equation*}
  \tilde{\gamma}_0= \frac{C_{\rho_2}}{\lambda_1} \vee \big( \frac{(4+\beta) C_{\rho_1} C_{A,\rho}}{\lambda_1} \big)^{1/2},
\end{equation*}
where $C_{A, \rho}$ is given by condition {\rm [A5]}. Lemma \ref{lem-1} gives an upper bound of $C_{A, \rho}$ as follows
\[
C_{A,\rho,\epsilon_0}:= \max \Big\{ \frac{ \sup_{0 \leq \epsilon \leq \epsilon_0} \|A^{\epsilon}(0)\|_{V^{\ast}}^2}{4 \delta(\epsilon_0)}, \frac{2}{2+\beta} \Big(\frac{\beta}{\lambda_1 (2+\beta)} \Big)^{\frac{\beta}{2}} \, \delta(\epsilon_0) \Big( \frac{ \sup_{0 \leq \epsilon \leq \epsilon_0} \|A^{\epsilon}(0)\|_{V^{\ast}}^2}{4 \delta(\epsilon_0)^2} \Big)^{1+\frac{\beta}{2}}   \Big\} + \epsilon_0 C_{B},
\]
for some $0 < \epsilon_0 < \frac{\lambda_1 \gamma_0-C_{\rho_2}}{2 C_B + \lambda_1 L_B}$. According to Theorems \ref{thm: LDP for stationary solution intro} and \ref{thm: LDP for invariant measures intro}, to obtain the LDP for stationary solutions and invariant measures of (\ref{eq: SPDE}), we need to verify the conditions {\rm{[A1]-[A5]}} and $ \tilde{\gamma}_0<\gamma_0$. When $\lambda_1 \gamma_0 > C_{\rho_2}$, Lemma \ref{lem-1} shows condition {\rm [A5]} can be obtained by {\rm [A1-A3]}. Thus, it is sufficient to check the conditions {\rm{[A1]-[A4]}}  and the following two conditions:
\begin{itemize}
  \item \textbf{(D1):} $\lambda_1 \gamma_0 > C_{\rho_2}$, where $\lambda_1$ is given in the embedding inequality \eqref{eq: embedding}.
  \item \textbf{(D2) :} $\lambda_1 \gamma_0^2 > (4+\beta) C_{\rho_1} C_{A,\rho,\epsilon_0} $ holds for some small enough $\epsilon_0>0$.
\end{itemize}

\subsection{2D stochastic Navier-Stokes equations on a torus}
Consider the two-dimensional Navier-Stokes equation defined on a torus $\mathbb{T}^2:=[0,2\pi]^2$, perturbed by a small noise,
\begin{equation} \label{eq: 2D NS}
  \begin{cases}
    \frac{\partial u(t,x)}{\partial t}=\chi\Delta u(t,x)-(u(t,x)\cdot \nabla)u(t,x)-\nabla p(t,x)+\sqrt{\epsilon}\eta(t,x) , &  \\
      ({\rm{div}}\ u)(t,x)=0 , &   \\
     u(0,x)=u_0(x) , &
  \end{cases}
\end{equation}
where $u(t,x)\in \mathbb{R}^2$ denotes the velocity field, $\chi>0$ is the viscosity, $p(t,x)$ denotes the pressure, $0<\epsilon\ll1$ and $\eta(t,x)$ is a Gaussian random field, white in time and colored in space.

Define $H:=\{f\in L^2(\mathbb{T}^2,\mathbb{R}^2): \int_{\mathbb{T}^2}f(x)dx=0,\  {\rm{div}}\ f=0\}$. Let $\Pi$ be the Leray-Helmholtz projection from $L^2(\mathbb{T}^2;\mathbb{R}^2)$ onto $H$. Setting
\begin{equation*}
  A:=\Pi(-\Delta),\quad D(A)= H \cap H^2(\mathbb{T}^2;\mathbb{R}^2)\ {\rm{and}}\ F(u,u)=\Pi((u\cdot \nabla)u).
\end{equation*}
The interpolation spaces between $D(A)$ and $H$ is $D(A^{\alpha})=[H,D(A)]_{\alpha}$, $\alpha\in (0,1)$. Let $V=D(A^{\frac{1}{2}})$ and $V^{*}$ is the dual space of $V$. Thus, $V\subset H\subset V^{*}$ be the Gelfand triple.

Applying the Leray-Helmholtz projection to \eqref{eq: 2D NS}, it can be written as in the abstract form
\begin{equation} \label{eq: 2D NS project}
  \begin{cases}
     \d u+\chi A u \, \d t+ F(u,u) \, \d t=\sqrt{\epsilon} Q(u) \, \d W(t),   \\
      u(0)=u_0,
  \end{cases}
\end{equation}
where $Q: H\rightarrow L_2(U;H)$ and $W(t)$ is a cylindrical Wiener process on a separable Hilbert space $U$.
When $Q\equiv Q(u)$ is an isomorphism of $H$ onto $D(A^{\frac{\alpha}{2}})$ for some $\alpha>1$, it was proved by \cite{Flandoli_Dissipativity_1994} that for all $\epsilon\in (0,1]$ and $u_0\in H$,
the Markov process on $H$ generated by
\eqref{eq: 2D NS project} has an invariant measure $\nu_{\epsilon}$, which is also unique and ergodic (see \cite{Ferrario_1999_Stochastic}). Furthermore, through a series of refined a priori estimates for the skeleton equation corresponding to \eqref{eq: 2D NS project}, in the periodic boundary condition, \cite[Theorem 7.1]{Brzezniak_Quasipotential_2015} proved that its quasi-potential admits an explicit expression
\begin{align}\label{rate}
   U(\phi) :=\begin{cases}
    |\phi|^2_V, & \phi\in V, \\
    +\infty, & \phi\in H\backslash V.
  \end{cases}
\end{align}
Building upon it,
by establishing two types of uniform LDP, \cite{Brzezniak_Large_2017} proved that the LDP of invariant measures $\{\nu_{\epsilon}\}_{\epsilon>0}$ with rate function defined by formula \eqref{rate}.

In the following, we will verify our framework {\rm{[A1]-[A5]}} for the equation \eqref{eq: 2D NS project}. Firstly, the embedding inequality \eqref{eq: embedding} holds for some $\lambda_1>0$. Denote by
\begin{align*}
  A_{F}(u):=-\chi Au-F(u,u).
\end{align*}
The hemicontinuity {\rm{[A1]}} is due to $A$ is a linear operator and $F$ is a bilinear
map. By the incompressible condition, the interpolation inequality and the Young inequality, we have
\begin{align}\notag
  2 _{V^{\ast}}\<A_{F}(v_1)-A_{F}(v_2)  , v_1 - v_2 \>_V
=& -2\chi\|v_1-v_2\|^2_V+\langle F(v_1-v_2,v_1), v_1-v_2\rangle\\ \notag
 \leq& -2\chi\|v_1-v_2\|^2_V+C\|v_1\|_V\|v_1-v_2\|^2_{L^4(\mathbb{T}^2)}\\
\label{r-2}
\leq& -\chi\|v_1-v_2\|^2_V+C(\chi)\|v_1\|^2_V\|v_1-v_2\|^2_H.
\end{align}
Moreover, referring to \cite[Example 5.1.10]{Liu_Stochastic_2015} and \cite[Section 4]{Liu_Large_2020}, one has
\begin{align}\label{r-4}
   \|A_{F}(v)\|_{V^*}\leq \chi\|v\|_V+C\|v\|_H\|v\|_V.
\end{align}
We next present and discuss our results separately for the cases of noises.

\noindent \textbf{\emph{Case 1} (Additive noise)}: The operator $Q$ is an additive noise satisfying $\|Q\|^2_{L_2(U;H)}\leq C_B$ for some constant $C_B>0$. Combining (\ref{r-2}) with (\ref{r-4}), it follows that Eq. \eqref{eq: 2D NS project} satisfies the conditions {\rm{[A1]-[A4]}} with $\beta=0$, $C_{\rho_2}=0$, $\gamma_0=\frac{\chi}{2}$ and $\kappa=2$. Clearly, condition (D1) holds. Note that $\|A_{F}(0)\|_{V^{\ast}}=0$, we have $C_{A, \rho,\epsilon_0}=\epsilon_0 C_B\rightarrow 0$ as $\epsilon_0 \rightarrow 0$, thus condition (D2) is valid.

\noindent \textbf{\emph{Case 2} (Multiplicative noise)}: The operator $Q: V\rightarrow L_2(U;H)$ satisfies the conditions {\rm{[A3]-[A4]}}. Based on (\ref{r-2}) and condition {\rm{[A3]}}, it follows that for $0<\epsilon<\frac{\lambda_1 \chi}{ 2C_B+ 2 \lambda_1 L_B}$,
\begin{align*}
  &2 _{V^{\ast}}\<A_{F}(v_1)-A_{F}(v_2) , v_1 - v_2 \>_V+\epsilon \, \| Q(v_1) - Q(v_2) \|_2^2\\
\leq& -\chi \|v_1-v_2\|^2_V+C(\chi)\|v_1\|^2_V\|v_1-v_2\|^2_H+\epsilon C_B\|v_1-v_2\|^2_H+\epsilon L_B\|v_1-v_2\|^2_V\\
\leq& -\chi \|v_1-v_2\|^2_V+C(\chi)\|v_1\|^2_V\|v_1-v_2\|^2_H+\epsilon \Big(\frac{C_B}{\lambda_1}+ L_B\Big)\|v_1-v_2\|^2_V\\
\leq& -\frac{\chi }{2}\|v_1-v_2\|^2_V+C(\chi,C_B,L_B,\lambda_1)\|v_1\|^2_V\|v_1-v_2\|^2_H.
\end{align*}
It implies that
$\gamma_0=\frac{\chi}{4}$, $C_{\rho_2}=0$ and $\beta=0$ in the condition {\rm{[A2]}}. Due to the same reason as in case 1, the conditions (D1) and (D2) hold. Combining the above two cases, we reach
\begin{theorem}[Ex: 2D SNS] \label{thm-r-1}
Assume that either $Q$ is an additive noise satisfying $\|Q\|^2_{L_2(U;H)}\leq C_B$, or
 $Q: V\rightarrow L_2(U;H)$ is a multiplicative noise satisfting the conditions {\rm{[A3]-[A4]}} with $\beta=0$.
Then for any small enough $\epsilon_0>0$ and all $0<\epsilon<\epsilon_0$, Theorem \ref{thm: stationary solution intro} shows that Eq. \eqref{eq: 2D NS project} admits a pathwise unique stationary solution. Furthermore, the LDP results stated in Theorems \ref{thm: LDP for stationary solution intro} and \ref{thm: LDP for invariant measures intro}  hold for the stationary solutions and invariant measures of Eq. \eqref{eq: 2D NS project}.
\end{theorem}
\begin{remark}
In \cite{Brzezniak_Large_2017}, the requirement $Q\in L_2(U;D(A^{\frac{\alpha}{2}}))$ with $\alpha>1$ is imposed to guarantee the exponential estimates for the invariant measure. Consequently, our result in Theorem \ref{thm-r-1} relaxes the condition imposed on $Q$. Moreover, the conditions imposed on $Q$ in case 2 are quite natural. For instance, we may take $Q(v)=a(v)B_0$, where $B_0\in L_2(U;H)$ and $a: H\rightarrow \mathbb{R}$ is a global Lipschitz continuous and uniformly bounded scalar function.
\end{remark}


Finally, we remark that the result of Theorem \ref{thm-r-1} remains valid for the equation \eqref{eq: 2D NS project} perturbed by an additional Kraichnan noise of transport type, namely
\begin{equation}\label{r-12}
  \begin{cases}
     \d u+[\chi Au +F(u,u)]\d t=\sqrt{\epsilon} Q(u) \, \d W(t)+\sqrt{\epsilon} \sum_{k\in \mathbb{Z}^2_0}  \Pi(\sigma_{k}\cdot \nabla u) \circ \,  \d B^{k}_t,   \\
      u(0)=u_0,
  \end{cases}
\end{equation}
where $\circ \d$ means the Stratonovich integral, $\mathbb{Z}^2_0$ is the set of non-zero lattice points and $\{\sigma_{k}\}$ are divergence free vector fields. The sequence $\{B^{k}\}$ consists of independent Brownian motions, which are independent of $W(t)$. We assume that $L_{B}:= \sum_{k\in \mathbb{Z}^2_0} \| \sigma_{k} \|_{L^{\infty}}^2 <+\infty$.

Clearly, Eq. \eqref{r-12} can be written in the form \eqref{eq: SPDE} with $A^{\epsilon} u := A_{F} + \frac{\epsilon}{2} \sum_{k} \Pi (\sigma_{k} \cdot \nabla \Pi (\sigma_{k} \cdot \nabla u ) )$ and $B u:= \sum_{k} \< Q(u) , \mathcal{U}_{k}^1 \>_{U_1} \mathcal{U}_{k}^1 + \sum_{k}\Pi(\sigma_{k}\cdot \nabla u) \, \mathcal{U}_{k}^{2}$, where $\{\mathcal{U}_{k}^{i}\}_{k}$ is an orthonormal basis of Hilbert space $U_{i}$ for $i=1,2$, respectively. Let $U = U_1 \times U_2$, if $Q \in L_2(U_1; H) $, then $B \in L_2(U; H)$. The Kraichnan noise preserves the energy estimates for differences, so the condition {\rm{[A2]}} coincides with that for the equation driven solely by $\sqrt{\epsilon} Q(u)\d W(t)$. A straightforward calculation shows that the last two items in the condition {\rm{[A3]}} holds. Moreover, with $L_{B}<+\infty$, the first two items in the condition {\rm{[A3]}} will also be satisfied by incorporating the extra terms $L_{B} \| v \|_{V}^2$ or $L_{B} \| v_1 -v_2 \|_{V}^2$, respectively. Due to the divergence free property of the Kraichnan noise, the condition {\rm{[A4]}} remains unchanged. Thus, the results of Theorem \ref{thm-r-1} are valid for Eq. \eqref{r-12}. Note that the operators $B$ and $Q$ are different, thus the rate functions defined by \eqref{def: rate func 1} and \eqref{def: rate func 2} are different.

\subsection{Stochastic 2D magneto-hydrodynamic equations}
Let $D\subset \mathbb{R}^2$ be a bounded open domain with smooth boundary and $T>0$. Consider stochastic 2D magneto-hydrodynamic equations (MHD) in the form
\begin{align}\label{MHD}
  \left\{
    \begin{array}{ll}
      \partial_t v=\chi\Delta v-(v\cdot \nabla)v+\frac{{{M}}^2\chi}{{{Rm}}}(b\cdot \nabla)b
-\nabla (P+\frac{{{M}}^2\chi}{{{Rm}}}\frac{|b|^2}{2})+\sqrt{\epsilon}\eta_1(t,x), &  \\
     \partial_t b=\chi_1\Delta b-(v\cdot \nabla)b+(b\cdot \nabla)v +\sqrt{\epsilon}\eta_2(t,x),   \\
     {\rm{div}}\ v=0,\  {\rm{div}}\ b=0, \quad {\rm{in}}\ D, &  \\
     v=0,\ b\cdot n=0,\ \partial_1b_2-\partial_2b_1=0, \quad {\rm{on}}\ \partial D, &
    \end{array}
  \right.
\end{align}
where $v=(v_1(x,t),v_2(x,t))$ is velocity field of fluid, $b=(b_1(x,t),b_2(x,t))$ is magnetic field, $P$ is pressure, $n$ denotes the outward normal to $\partial D$, $\frac{1}{\chi}>0$ is the Reynold number, $\chi_1>0$ is the magnetic resistivity and ${{Rm}}, M$ correspond to the magnetic Reynolds number and the Hartman number, respectively. Moreover, $\eta_i(t,x)$, $i=1,2$ are
random perturbations. Let $p=P+\frac{M^2\chi}{Rm}\frac{|b|^2}{2}$, it stands for the total pressure.
Boundary conditions on $b$ express the physical hypothesis that the boundary is perfectly conductive. Without loss of generality, we can assume that $\chi=\chi_1$ and $\frac{{{M}}^2\chi}{{{Rm}}}=1$. System \eqref{MHD} models the interaction between a conductor fluid or plasma and the magnetic field $b$ in the presence of random perturbations. There is a vast literature on magneto-hydrodynamic equations, see \cite{Barbu_2007_Existence,Sermange_1983_Mathematical} and references therein.

To further describe the system \eqref{MHD}, we need some notations.
The space $C^{\infty}_0(D;\mathbb{R}^2)$ consists of all smooth functions from $D$ to $\mathbb{R}^2$ with compact support. $W^{1,2}_0(D;\mathbb{R}^2)$ is the closure of $C^{\infty}_0(D;\mathbb{R}^2)$ with respect to the norm $\|u\|_{W^{1,2}_0(D;\mathbb{R}^2)}=\int_{D}| \nabla u(x)|^2 dx$.
Define the space
\begin{align*}
V_1=\{v\in W^{1,2}_0(D;\mathbb{R}^2): \nabla\cdot v=0\ in\ D\},
\end{align*}
whose norm is $\|v\|_{V_1}:=\Big(\int_D|\nabla v|^2dx\Big)^{1/2}$. Let $H_1$ be the closure of $V_1$ in the norm $\|v\|_{H_1}:=\Big(\int_D|v|^2dx\Big)^{1/2}$. Define the following space with the norm $\|b\|_{V_2}:=\Big(\int_D|\nabla b|^2dx\Big)^{1/2}$:
\begin{align*}
V_2=\{b\in W^{1,2}(D;\mathbb{R}^2): \nabla\cdot b=0\ in\ D,\ b\cdot n=0\ on\ \partial D\},
\end{align*}
Denote by $H:=H_1\times H_1$, $V:=V_1\times V_2$ and $V^{\ast}$ is the dual space with $V$. Then $V \subset H \subset V^{\ast}$ is a Gelfand triple and the embedded inequality \eqref{eq: embedding} holds.

 For any $k\in \mathbb{N}$, set $W^{k,2}(D;\mathbb{R}^2)=\{v\in L^2(D;\mathbb{R}^2): D^{\alpha}v\in L^2(D;\mathbb{R}^2)\}$. Let $\Pi_{H_1}$ be the Helmhotz-Hodge projection from $L^2(D; \mathbb{R}^2)$ to $H_1$. Define
\begin{align*}
  A_1v &=\chi \Pi_{H_1}\Delta v,\quad v\in W^{2,2}(D;\mathbb{R}^2)\cap V_1, \\
  A_2b &=\chi \Delta b,\quad  \quad \quad  b\in \{u\in W^{2,2}(D;\mathbb{R}^2)\cap V_2: \partial_1b_2-\partial_2b_1=0\ {\rm{on}}\ \partial D\}.
\end{align*}

We set $b(y,z,w)= \int_{D}  (y \cdot \nabla z)  \cdot w \, \d x$, for $y,z,w\in V$. Define the operators $F_1: V_1\times V_1\rightarrow V^*_1$, $F_2: V_2\times V_2\rightarrow V^*_1$, $F_3: V_1\times V_2\rightarrow V^*_2$ and $F_4: V_2\times V_1\rightarrow V^*_2$:
\begin{align*}
  & (F_1(y,y),w)=\Pi_{H_1}b(y,y,w),\quad y,w\in V_1, \\
  & (F_2(z,z),w)=\Pi_{H_1}b(z,z,w),\quad z\in V_2,\ w\in V_1, \\
  &(F_3(y,z),w)=b(y,z,w),\quad  \quad  \quad y\in V_1, z\in V_2, w\in V_2,\\
 &(F_4(z,y),w)=b(z,y,w),\quad \quad \quad y\in V_1, z\in V_2,\ w\in V_2.
\end{align*}
Let $X=(v,b)$, set
\begin{align*}
  AX=(A_1v, A_2b)^{\top},\ F(X,X)=(-F_1(v,v)+F_2(b,b), -F_3(v,b)+F_4(b,v))^{\top}.
\end{align*}

With the above operators in hand, by applying $\Pi_{H_1}$ to the system \eqref{MHD}, we obtain
\begin{align}\label{MHD-1}
  dX=AXdt+F(X,X)dt+\sqrt{\epsilon}Q(X)dW(t),
\end{align}
where $Q: H\rightarrow L_2(U;H)$ and $W(t)$ is a cylindrical Wiener process on a separable Hilbert space $U$. Denote by $A_F(X)=AX+F(X,X)$. In the following, we will verify conditions {\rm{[A1]-[A3]}}.

The hemicontinuity {\rm{[A1]}} follows from the fact that $A$ is a linear operator and $F$ is a bilinear map. Similar to 2D NS, it follows from the H\"{o}lder, interpolation and Young inequalities that
\begin{align*}
2 _{V^{\ast}}\<A_F(X_1)-A_F(X_2)  , X_1 - X_2 \>_V
\leq& -\frac{3\chi}{2} \|X_1 - X_2\|^2_V+C(\chi)\|X_2\|^4_{L^4(D)}\|X_1 - X_2\|^2_H \nonumber \\
\leq& -\frac{3\chi}{2} \|X_1 - X_2\|^2_V+C(\chi)\|X_2\|^2_{H}\|X_2\|^2_{V}\|X_1 - X_2\|^2_H.
 \end{align*}
If $B$ satisfies the conditions {\rm{[A3]-[A4]}} with $\beta=2$, then for all $0<\epsilon<\frac{\lambda_1 \chi}{2C_B+2 \lambda_1 L_B}$,
\begin{align*}
& 2 _{V^{\ast}}\<A_F(X_1)-A_F(X_2)  , X_1 - X_2 \>_V  + \epsilon \| Q(v_1)-Q(v_2) \|_2^2  \\
\leq& -\chi \|X_1 - X_2\|^2_V+C(\chi)\|X_2\|^2_{H}\|X_2\|^2_{V}\|X_1 - X_2\|^2_H.
 \end{align*}
It implies that the condition {\rm{[A2]}} with $\gamma_0=\frac{\chi}{2}$, $C_{\rho_2}=0$ and $\beta=2$, thereby condition (D1) holds. Moreover, we readily to check that
\begin{align*}
   \|A_F(X)\|^2_{V^*}\leq C(1+\|X\|^2_H)(1+\|X\|^2_V),
\end{align*}
for some $C>0$ and condition {\rm{[A3]}} holds for $\kappa=2$. Since $\|A_F(0)\|_{V^{\ast}}=0$, we have $C_{A, \rho,\epsilon_0}=\epsilon_0 C_B\rightarrow 0$ as $\epsilon_0 \rightarrow 0$, thus condition (D2) is valid. We conclude that
\begin{theorem}(Ex: 2D MHD equations)
Assume that $Q$ satisfies the conditions {\rm{[A3]-[A4]}} with $\beta=2$. Then for any small enough $\epsilon_0>0$ and all $0<\epsilon<\epsilon_0$, Theorem \ref{thm: stationary solution intro} shows that Eq. \eqref{MHD-1} admits a pathwise unique stationary solution. Furthermore, the LDP results stated in Theorems \ref{thm: LDP for stationary solution intro} and \ref{thm: LDP for invariant measures intro}  hold for the stationary solutions and invariant measures of Eq. \eqref{MHD-1}.
\end{theorem}
The noise defined by $B(X)=\frac{\sigma_0}{1+\|X\|^2_H}\cdot I_U$ satisfy the conditions {\rm{[A3]-[A4]}} with $\beta=2$, where $I_U: U\rightarrow H$ is any Hilbert-Schmidt operator and constant $\sigma_0>0$.

\subsection{Stochastic hyper-dissipative 3D Navier-Stokes equations}
Consider the 3D stochastic hyper-dissipative
Navier-Stokes equations defined on $\mathbb{T}^3=[0,2\pi]^3$,
\begin{align}\label{rrr-20}
  du=\Big(-\chi(-\Delta)^{\alpha}u+(u\cdot \nabla)u+\nabla p \Big)dt+\sqrt{\epsilon}\eta(t,x),\quad \nabla\cdot u=0, \, u(0)=u_0,
\end{align}
where $u$ is velocity field, $\chi>0$ is the viscosity, $p$ is pressure and $\eta(t,x)$ is a Gaussian random field white in time and colored in space. The fractional Laplacian $(-\Delta)^{\alpha}$ is defined through the Fourier transform $\widehat{(-\Delta)^{\alpha}u}(k)=|2\pi k|^{2\alpha}\widehat{u}(k)$. When $\alpha=1$, Eq. \eqref{rrr-20} is reduced to the classical 3D Navier-Stokes equation. It is well-known that the Leray weak solution of 3D Navier-Stokes equations is non-uniqueness for both deterministic and stochastic cases, see \cite{Buckmaster_2019_Nonuniqueness}. Dissipation corresponding to a fractional Laplacian arises from modeling real physical phenomena. In order to explore the long-time behavior, we consider the hyper-dissipative critical case, $\alpha=\frac{5}{4}$.

Denote by $H:=\{u\in L^2(\mathbb{T}^3): {\rm{div}}\ u=0, \int_{\mathbb{T}^3}u dx=0\}$, whose norm is $\|u\|^2_H=\|u\|^2_{L^2(\mathbb{T}^3)}$. Let $V= H \cap  H^{\frac{5}{4}}(\mathbb{T}^3)$ and $V^{\ast}$ be the dual space of $V$, then $V\subset H\subset V^* $ is a Gelfand triple. Clearly, the embedding inequality \eqref{eq: embedding} holds for some $\lambda_1>0$. Let $\Pi$ be the Leray-Helmholtz projection from $L^2(\mathbb{T}^3)$ to $H$. Define the operator $A^{\alpha}$ by $A^{\alpha}u:=\chi\Pi (-\Delta)^{\alpha} u$ and the bilinear map $F(u,v)=\Pi (u\cdot \nabla) v$.
With the projection $\Pi$, Eq. \eqref{rrr-20} can be rewritten as
\begin{align}\label{rrr-21}
  du= A_{F}(u)dt+\sqrt{\epsilon}Q(u) \, \d W :=-\chi A^{\frac{5}{4}}udt-F(u,u)dt+\sqrt{\epsilon}Q(u) \, \d W,
\end{align}
where $Q: H\rightarrow L_2(U;H)$ and $W(t)$ is a cylindrical Wiener process on a Hilbert space $U$.

The hemicontinuity {\rm{[A1]}} is obvious, since $A^{\frac{5}{4}}$ is a linear operator and $F$ is a bilinear
map. It follows from the H\"{o}lder and Young inequalities that
\begin{align}
&2 _{V^{\ast}}\<A_{F}(v_1)-A_{F}(v_2)  , v_1 - v_2 \>_V \nonumber \\
\leq& -2\chi \|A^{\frac{5}{4}}(v_1-v_2)\|^2_H+\|A (v_1-v_2)\|_{L^{\frac{12}{5}}(\mathbb{T}^3)} \|v_1\|_{L^{12}(\mathbb{T}^3)}\|v_1-v_2\|_H \nonumber\\
  \leq& -\frac{3}{2}\chi \|v_1-v_2\|^2_V+C\|v_1\|^2_V\|v_1-v_2\|^2_H, \label{r-3}
\end{align}
where we have used the embedding inequalities in $\T^3$:
\begin{equation*}
  \|v_1\|_{L^{12}(\T^3)} \leq C \|A^{\frac{5}{4}}v_1\|_H= C \|v_1\|_{V} \, \mbox{and} \, \|A (v_1-v_2)\|_{L^{\frac{12}{5}}(\mathbb{T}^3)} \leq C \|A^{\frac{5}{4}}(v_1-v_2)\|_H = C \|v_1-v_2\|_{V}.
\end{equation*}
Moreover, it follows from (2.7) in \cite{Qiu_2023_Large} that
\begin{align}\label{r-5}
   \|A_{F}(u)\|_{V^*}\leq \chi\|u\|_V+C\|u\|_H\|u\|_V,
 \end{align}
which implies that condition {\rm{[A3]}} holds for $\kappa=2$. Based on (\ref{r-3}) and (\ref{r-5}), we conclude from  the discussion in the above of Theorem \ref{thm-r-1} that
\begin{theorem}(Ex: Hyper-dissipative 3D NS equations)
Assume that either $Q$ is an additive noise satisfying $\|Q\|^2_{L_2(U;H)}\leq C_B$, or $Q: V\rightarrow L_2(U;H)$ satisfies the conditions {\rm{[A3]-[A4]}} with $\beta=0$. Then for any small enough $\epsilon_0>0$ and all $0<\epsilon<\epsilon_0$, Theorem \ref{thm: stationary solution intro} shows that Eq. \eqref{rrr-21} admits a pathwise unique stationary solution. Furthermore, the LDP results stated in Theorems \ref{thm: LDP for stationary solution intro} and \ref{thm: LDP for invariant measures intro}  hold for the stationary solutions and invariant measures of Eq. \eqref{rrr-21}.
\end{theorem}
\begin{remark}
Qiu et al. \cite{Qiu_2023_Large} investigated Eq. \eqref{rrr-21} driven by additive noise and established its ergodicity. By proving Freidlin-Wentzell ULDP over bounded sets, they further claimed the LDP for invariant measures. However, in their proof of the lower bound for this LDP, the Dembo--Zeitouni ULDP over bounded sets was used without justification. We establish the LDP for invariant measures by proving the corresponding LDP for stationary solutions, thereby avoiding the need to establish Dembo--Zeitouni ULDP over bounded sets.
\end{remark}

\subsection{Stochastic semilinear equations}\label{subsec-1}
Let $\Lambda$ be an open bounded domain in $\mathbb{R}^d$ with smooth boundary. Let $C^{\infty}_0(\Lambda)$ denote the set of all infinitely differentiable real-valued functions on $\Lambda$ with compact support. Define $H^{1,p}_0(\Lambda)$ as the completion of $C^{\infty}_0(\Lambda)$ with norm $\|u\|_{H^{1,2}_0}:=\Big(\int_{\Lambda}|\nabla u(x)|^2 \d x\Big)^{\frac{1}{2}}$. Define the Gelfand triple
\begin{align*}
  V:=H^{1,2}_0(\Lambda)\subset H:=L^2(\Lambda)\subset (H^{1,2}_0(\Lambda))^*=V^*.
\end{align*}
The Poincar inequality gives the embedded inequality between $H$ and $V$ like \eqref{eq: embedding}.

Consider the stochastic semilinear equation
\begin{align}\label{rrr-16}
  dX(t)=\Big(\chi\Delta X(t)+\sum^d_{i=1}f_i(X(t)) \frac{\partial}{\partial x_{i}} X(t)+g(X(t))\Big)dt+\sqrt{\epsilon}B(X(t)) \, \d W(t),
\end{align}
where $X(t,x)=0$ for $x\in \partial \Lambda$, $X(0)=X_0\in H$, $W_t$ is a cylindrical wiener process, $\{f_i\}$ are Lipschitz functions on $\mathbb{R}$, and $g$ is a continuous function on $\mathbb{R}$ with $g(0)=0$ and satisfies
\begin{align}
  &|g(\xi)|\leq C (|\xi|^r+1),  \quad \xi \in \mathbb{R}, \label{rrr-22}\\
  &(g(\xi)-g(\zeta))(\xi-\zeta)\leq C_{g} (1+|\xi|^s)(\xi-\zeta)^2,\quad \xi, \zeta\in \mathbb{R},   \label{r-13}
\end{align}
where $C,C_{g},s$ are nonnegative constants and $r \geq 1$. Let
\begin{align*}
A_g(u):=\chi\Delta u+f(u)\partial_{x}u+g(u)=:A(u)+g(u).
\end{align*}
The hemicontinuity {\rm{[A1]}} follows from the above conditions on $f$ and $g$.

In what follows, we consider two distinct cases and provide the associated results as well as illustrative examples, including stochastic 1D Burgers equation,  stochastic 1D Ginzburg--Landau equation and stochastic reaction-diffusion equations.

\subsubsection{\texorpdfstring{Case 1: \ $d=1$, $\{f_i\}_i$ are Lipschitz, $r\leq 3$ and $s=0$.}{Case 1 : d=1.}}
Using $H^{1,2}_0(\Lambda)\subset L^{\infty}(\Lambda)$ continuously and Young inequality, it was proved by (5.15) in \cite{Liu_Stochastic_2015}
\begin{align}
  &2 _{V^{\ast}}\<A(v_1)-A(v_2), v_1 - v_2 \>_V \nonumber \\
  =&-2\chi\|v_1 - v_2\|^{2}_{V}+Lip(f)\Big(\|v_1\|_{L^{\infty}(\Lambda)}\|v_1 - v_2\|_{H} \|v_1 - v_2\|_{V}+\|v_1 \|_{V}\|v_1 - v_2\|^2_{L^4(\Lambda)}\Big) \nonumber \\
\leq&  -\frac{3\chi}{2}  \|v_1 - v_2\|_{V}^{2} + C\|v_1\|^2_V \| v_1 - v_2 \|_H^2, \label{rrr-24}
\end{align}
for any $v_1,v_2\in V$. The condition \eqref{r-13} yields for $v_1, v_2\in V$,
\begin{equation} \label{rrr-25}
  2 _{V^{\ast}}\<g(v_1)-g(v_2)  , v_1 - v_2 \>_V
  =2\int_{\Lambda}(g(v_1)-g(v_2))(v_1-v_2)dx\leq 2 C_{g} \| v_1 - v_2 \|_H^2.
\end{equation}
Assume that the operator $B$ satisfies the conditions {\rm{[A3]-[A4]}} with $\beta=0$, we deduce from {\rm{[A3]}}, estimates \eqref{rrr-25} and \eqref{rrr-24} that for $0<\epsilon<\frac{\lambda_1 \chi}{2C_B+2 \lambda_1 L_B}$,
\begin{align}
  &2 _{V^{\ast}}\<A_g(v_1)-A_g(v_2)  , v_1 - v_2 \>_V +\epsilon\|B(v_1)-B(v_2)\|^2_2 \nonumber \\
  \leq & -\frac{3\chi}{2}   \|v_1 - v_2\|_{V}^{2} +(C \|v_1\|^2_V +2C_{g}) \| v_1 - v_2 \|_H^2+\epsilon C_B\| v_1 - v_2 \|_H^2+\epsilon L_B\| v_1 - v_2 \|_V^2 \nonumber \\
  \leq & -\chi   \|v_1 - v_2\|_{V}^{2} + ( C \|v_1\|^2_V + 2C_{g} ) \| v_1 - v_2 \|_H^2. \label{r-6}
\end{align}
Thus, condition {\rm{[A2]}} holds for $\gamma_0=\chi/2$, $C_{\rho_2}=2 C_{g}$ and $\beta=0$. Thus, when $\chi> 4C_{\rho}/ \lambda_1 $, condition (D1) holds. Since $r\leq 3$, using $\|u\|_{L^r(\Lambda)}\leq\|u\|_{L^3(\Lambda)}\leq \|u\|^{\frac{2}{3}}_H\|u\|^{\frac{1}{3}}_V$, we obtain
\begin{align*}
|_{V^{\ast}}\<g(u)  , v \>_V|\leq C\int_{\Lambda}(1+|u|^r)|v|d\xi\leq C\|v\|_{L^{\infty}}(1+\|u\|^3_{L^3})\leq C\|v\|_V(1+\|u\|_V\|u\|^2_H).
\end{align*}
Then, we reach the condition {\rm{ [A3]}} with the parameter $\kappa=4$:
\begin{align}\label{r-7}
  \|A_g(u)\|_{V^*}\leq \chi\|u\|_V+C\|u\|_H\|u\|_V+C(1+\|u\|_V\|u\|^2_H).
\end{align}
Note that $\|A_g(0)\|_{V^{\ast}}=0$, we have $C_{A, \rho,\epsilon_0}=\epsilon_0 C_B\rightarrow 0$ as $\epsilon_0 \rightarrow 0$, thus condition (D2) is valid. Thus, as we discussed in the beginning of this section, we conclude that
\begin{theorem}\label{thm-6}
Assume that $d=1$, $\{f_i\}_i$ are Lipschitz continuous, $g$ satisfies \eqref{rrr-22}-\eqref{r-13} with $r\leq 3$ and $s=0$. Let $\chi> 4C_{g}/ \lambda_1 $ and assume that either $B$ is an additive noise satisfying $\|B\|^2_{L_2(U;H)}\leq C_B$, or $B: V\rightarrow L_2(U;H)$ satisfies the conditions {\rm{[A3]-[A4]}} with $\beta=0$. Then for any small enough $\epsilon_0>0$ and all $0<\epsilon<\epsilon_0$, Theorem \ref{thm: stationary solution intro} shows that Eq. \eqref{rrr-16} admits a pathwise unique stationary solution. Furthermore, the LDP results stated in Theorems \ref{thm: LDP for stationary solution intro} and \ref{thm: LDP for invariant measures intro}  hold for the stationary solutions and invariant measures of Eq. \eqref{rrr-16}.
\end{theorem}

The 1D stochastic viscous Burgers equation:
  \begin{align}\label{rrr-15}
  dX(t)=\Big(\chi \frac{\partial^2}{\partial x^2 } X(t)+X(t) \frac{\partial}{\partial x} X(t)\Big)dt+\sqrt{\epsilon}B(X(t)) \, \d W(t),
\end{align}
is a typical example satisfying Theorem \ref{thm-6} with $C_{g}=0$.
\begin{remark}
Bai et al.\ \cite{Bai_2026_Large} studied equation \eqref{rrr-15} on $[0,1]$ with $B(X(t)) = \sqrt{Q_{\delta(\epsilon)}}$, where $Q_{\delta(\epsilon)}$ is a trace class operator converging to $(-\Delta)^{\alpha/2}$ with $\alpha < 1$ as $\epsilon \to 0$. Under a joint scaling regime in which $(\epsilon, \delta(\epsilon)) \to 0$, they established the LDP for invariant measures. As $\epsilon \to 0$, the operator $Q_{\delta(\epsilon)}$ becomes increasingly singular, the uniform estimate with $\epsilon>0$ of solutions and the well-posedness of associated skeleton equation relies on the convolution structure, which is natural in the mild solution framework but difficult to exploit within the variational framework. Consequently, this result lies outside the scope of the present work.
\end{remark}

The other typical example with $C_{g}=0$ is the 1D stochastic Ginzburg--Landau equation:
 \begin{align*}
 dX(t)=\Big(\chi\Delta X(t)+f(X(t))\partial_{x}X(t)-\alpha X(t)-c|X(t)|^2X(t)\Big)dt+\sqrt{\epsilon}B(X(t)) \, \d W(t),
\end{align*}
where $\alpha$ and $c$ are positive constants.

Finally, we give an example satisfying Theorem \ref{thm-6} with $C_{g} \geq 0$,
\begin{align}\label{r-14}
  du=(\chi\Delta u-\alpha u+h(u))dt+\sqrt{\epsilon}B(u) \, \d W(t),\quad u(0)=u_0:
\end{align}
where $\alpha > 0$ and $h$ is Lipschitz continuous fulfilling $h(0)=0$. Thus, $g(u)=-\alpha u + h(u)$ naturally satisfies \eqref{r-13} with $C_{g}=(Lip(h)-\alpha) \vee 0$ and $s=0$. When $Lip(h)<\alpha$, Zhang \cite[Theorem 6.2]{Zhang_2012_Large} has established the LDP of invariant measures of Eq. \eqref{r-14} with two reflecting walls. Thus, our result Theorem \ref{thm-6} recovers and generalizes their results to a broader setting, as it removes the requirement of reflecting walls and only needs $\alpha + \frac{\lambda_1 \chi}{4} >Lip(h)$.

\subsubsection{\texorpdfstring{Case 2:\ $d=3$, $f_i=0$, $r\leq\frac{7}{3}$ and s=0.}{Case 2: d=3.}}
Estimate \eqref{rrr-25} also holds and it gives
\begin{align} \label{rrr-27}
  2 _{V^{\ast}}\<A_g(v_1)-A_{g}(v_2), v_1 - v_2 \>_V
  \leq -2\chi\|v_1 - v_2\|^{2}_{V}+2C_g\|v_1 - v_2\|^{2}_{H}.
\end{align}
For any $v\in V$, we have
\begin{align*}
   _{V^{\ast}}\<g(u) , v \>_V\leq C\int_{\Lambda}(1+|u|^r)|v|d\xi
   \leq C(\|v\|_{L^1(\Lambda)}+\|u\|^r_{L^{\frac{6r}{5}}(\Lambda)}\|v\|_{L^{6}(\Lambda)}).
\end{align*}
By using $H^{1,2}_0(\Lambda)\subset L^{6}(\Lambda)$ continuously, $r\leq \frac{7}{3}$ and the interpolation inequality, we reach
\begin{align*}
  \|g(u)\|_{V^*}\leq 1+\|u\|^r_{L^{\frac{6r}{5}}(\Lambda)}
  \leq 1+\|u\|_{L^{6}(\Lambda)}\|u\|^{r-1}_{L^{\frac{(r-1)3}{2}(\Lambda)}}.
\end{align*}
 Since $r\leq \frac{7}{3}$, we obtain that $\|g(u)\|_{V^*}\leq 1+\|u\|_V\|u\|^{\frac{4}{3}}_H$ and
\begin{align}\label{r-8}
  \|A_g(u)\|_{V^*}\leq 1+\chi \|u\|_V+\|u\|_V\|u\|^{\frac{4}{3}}_H.
\end{align}
It implies that the parameter $\kappa=\frac{8}{3}$ in the condition {\rm{[A3]}}. Based on \eqref{rrr-27} and \eqref{r-8}, it follows from the discussion in the above of Theorem \ref{thm-r-1} that
\begin{theorem}
Assume that $d=3$, $f_i=0$ for all $i =1 ,2 ,3$, $g$ satisfies \eqref{rrr-22}-\eqref{r-13} with $r \leq \frac{7}{3}$ and $s=0$. If $B$ is an additive noise satisfying $\|B\|^2_{L_2(U;H)}\leq C_B$, or $B: V\rightarrow L_2(U;H)$ satisfies the conditions {\rm{[A3]-[A4]}} with $\beta=0$, then the LDP results stated in Theorems \ref{thm: LDP for stationary solution intro} and \ref{thm: LDP for invariant measures intro} hold for the stationary solutions and invariant measures of Eq. \eqref{rrr-16}.
\end{theorem}

A typical example is 3D stochastic reaction-diffusion equation
\begin{align}\label{r-9}
  \d u(t)+(\lambda u(t)-\Delta u(t)+c|u|^{\frac{4}{3}}u) \d t=\sqrt{\epsilon}B(u) \d W_t,
\end{align}
where $\lambda, c$ are positive constants.
\begin{remark}
As established in \cite[Theorem 4.1]{Liu_Large_2020}, for reaction-diffusion equations, the general polynomial growth nonlinearity cannot be addressed due to constraints inherent in the variational framework. Within this framework, the growth order must be tied to the spatial dimension $d$. By contrast, the results of \cite{Cerrai_Stochastic_2003,Cerrai_Large_2005} are obtained in the space $C(\bar{\mathcal{O}};\mathbb{R}^r)$ using mild solution, which impose no such restriction on the growth order.
\end{remark}

\appendix
\section{Analysis tools} \label{sec: appendix}

Kurtz in \cite[Theorem 3.14]{Kurtz_YamadaWatanabeEngelbert_2007} shows the general Yamada--Watanabe--Engelbert Theorem for the compatible solution. Following this result and \cite{Rockner_YamadaWatanabe_2008}, we can get a similar result for $\cF_{t}$-adapted solution to \eqref{eq: stationary solution}, which are driven by double-side Wiener process $W$. Recall that $\W$ is the Wiener space, which the double-side wiener process $W$ lies on, see \eqref{def: Wiener space}.
\begin{lemma}\label{lem: Yamada-Watanabe}
Let $\mathcal{P}_{W}(\sC_{-\infty} \times \W)$ denote the Borel probability space on $\sC_{-\infty} \times \W$ satisfying the marginal distribution on $\W$ is $\P \circ W^{-1}$. Let $S_{\Gamma} \subset \mathcal{P}_{W}(\sC_{-\infty} \times \W)$ denote the collection of $\cF_{t}$-adapted solution measures with the constraint $\Gamma$. If $S_{\Gamma} $ is a convex non-empty set, then the following are equivalent:
 \begin{itemize}
  \item [i)] Pathwise uniqueness holds for $\cF_{t}$-adapted martingale solutions. This means that if there are two $\cF_{t}$-adapted process $X^1_t$ and $X^2_t$ defined on the same probability space such that $\nu_{1}:=\P \circ (W, X^1)^{-1}$ and $\nu_{2}:=\P \circ (W, X^2)^{-1}$ both belong to $S_{\Gamma}$, then $X^1=X^2$, $\P$-a.s..
  \item [ii)] Joint uniqueness in law holds, and there exists a $\cF_{t}$-adapted strong solution. This means that $S_{\Gamma}$ contains at most one measure, denoted by $\nu$, and $\nu$ corresponds to an $\mathcal{F}_t$-adapted process $X_{t}=\mathcal{G}(W)(t)$ such that the pair $(\omega, \mathcal{G}(W(\omega))$ satisfies the constraint $\Gamma$, where $\mathcal{G}: \W \rightarrow \sC_{-\infty}$ is a Borel measurable map.
 \end{itemize}
\end{lemma}
\begin{proof}
 ii) implies i) is immediate, so we focus on obtaining ii) from i).

Let $\zeta_1 $ and $\zeta_2$ uniformly distributed on $[0,1]$ and $\zeta_1$ and $\zeta_2$ are independent. Then by \cite[Lemma A.1]{Kurtz_YamadaWatanabeEngelbert_2007}, for any $\nu_1,\nu_2 \in S_{\Gamma}$, there exists Borel measurable map $F_1: \W \times [0,1]$ and $F_2: \W \times [0,1]$ such that $(F_{1}(W, \zeta_1),W)$ and $(F_{2}(W, \zeta_2),W)$ has the distribution $\nu_1$ and $\nu_2$, respectively. Let $\mathcal{B}_{t}(\sC_{-\infty})$ be the $\sigma$-algebra generated by all maps $\{ \pi_{s} \}_{s \leq t}$, where $(\pi_{s} z) (\tau):= z(\tau) \ind_{\tau \leq s}$, $z \in \sC_{-\infty}$. Next, we will show for any $\mathcal{B}_{t}(\sC_{-\infty})$ measurable map $f : \sC_{-\infty} \rightarrow \R$, map
 \begin{equation*}
  (\W,[0,1]) \ni (y,s) \mapsto f(F_{1}(y, s))= \int_{\sC_{-\infty}} f(z) K_{F_1}(y,s, \d z)
 \end{equation*}
 is $\mathcal{B}_{t}(\W) \otimes \mathcal{B}([0,1])$ (defined similar as $\mathcal{B}_{t}(\sC_{-\infty})$) measurable, where $K_{F_1}(y,s,z)$ is a transition kernel function. It is sufficient to prove that for any $B \in \mathcal{B}_{t}(\sC_{-\infty})$, $A_1 \in \mathcal{B}_{t}(\W)$ and
 \begin{align*}
  A_2 \in \sigma \big( & \{ w(t_1)-w(t) \in D_1, w(t_2)-w(t_1) \in D_2, \cdots, w(t_k)-w(t_{k-1}) \in D_k \big. \\
  & \quad \big. : \, t \leq t_1 < t_2 < \cdots <t_k, \, D_1,D_2, \cdots D_k \in \mathcal{B}(U), \, w \in \W \} \big)
 \end{align*}
 the following identity holds
 \begin{align} \label{eq: adapted}
  & \int_{\W \times [0,1]} \mathrm{1}_{ A_1 \cap A_2}( y) \, K_{F_1}(y,s,B) \, \P \circ W^{-1} (\d y) \, \d s \nonumber \\
  &= \int_{\W \times [0,1]} \mathrm{1}_{ A_1 \cap A_2}( y) \, \E [K_{F_1}(\cdot,B) | \mathcal{B}_t(\W) \otimes \mathcal{B}([0,1])] \, \P \circ W^{-1} (\d y) \, \d s,
 \end{align}
since all $A_1 \cap A_2$ generate $\cF$. The left-hand side of \eqref{eq: adapted} is equal to
 \begin{equation*}
  \int_{\Omega} \mathrm{1}_{ A_1 \cap A_2}(W) \mathrm{1}_{B}(Z) \d \P= \P \circ W^{-1} (A_2) \int_{\Omega} \mathrm{1}_{ A_1 }(W) \mathrm{1}_{B}(Z) \d \P,
 \end{equation*}
 where $(Z,W)$ is a $\cF_t$-adapted martingale solution and the last step used $\mathrm{1}_{A_2}(W)$ is independent of $\mathcal{F}_{t}$. Due to $A_2$ is independent $\mathcal{B}_{t}(\W)$, the right-hand side of \eqref{eq: adapted} is equal to
 \begin{align*}
  & \int_{\W } \mathrm{1}_{ A_2}( y) \, \P \circ W^{-1} (\d y) \times \int_{\W \times [0,1]} \mathrm{1}_{ A_1 }( y) \, \E [K_{F_1}(\cdot,B) | \mathcal{B}_t(\W) \otimes \mathcal{B}([0,1])] \, \P \circ W^{-1} (\d y) \, \d s \\
  & \quad = \P \circ W^{-1} (A_2) \times \int_{\W \times [0,1]} \mathrm{1}_{ A_1 }( y) \, K_{F_1}(\cdot,B) \, \P \circ W^{-1} (\d y) \, \d s \\
  & \quad = \P \circ W^{-1} (A_2) \int_{\Omega} \mathrm{1}_{ A_1 }(W) \mathrm{1}_{B}(Z) \d \P.
 \end{align*}
 Thus, \eqref{eq: adapted} holds, and furthermore $f(F_1(W,\zeta_1))$ is $\mathcal{F}_{t} \otimes \sigma(\zeta_1)$ measurable. The similar result holds for $F_2(W,\zeta_2)$.
 By pathwise uniqueness and \cite[Lemma A.2]{Kurtz_YamadaWatanabeEngelbert_2007}, there exists Borel measurable $F : \W \rightarrow \sC_r$ such that $\P$-a.s. $F(W)=F_{1}(W, \zeta_1)=F_{2}(W, \zeta_2)$, then $F(W)$ is the strong solution. By the measurability properties of $F_1(W,\zeta_1)$, we obtain that $F(Y)$ is $\mathcal{F}_{t}$-adapted.
\end{proof}

\noindent{\bf  Acknowledgements:}
Yong Liu is supported by NSFC (Nos. 12231002) and Center for Statistical Sciences, PKU. Rangrang Zhang is supported by Open Foundation of the State Key Laboratory of Mathematical Sciences (Grant No. 09), Beijing Institute of Technology Research Fund Program for Young Scholars and MIIT Key Laboratory of Mathematical Theory and Computation in Information Security.  Rangrang Zhang is the corresponding author.

\noindent{\bf Author Contributions:} Three authors contributed equally. 

\printbibliography

@article{Bakhtin_Existence_2003,
  title = {Existence and Uniqueness of a Stationary Solution of a Nonlinear Stochastic Differential Equation with Memory},
  author = {Bakhtin, Yuri},
  date = {2003-01},
  journaltitle = {Theory of Probability \& Its Applications},
  shortjournal = {Theory Probab. Appl.},
  volume = {47},
  number = {4},
  pages = {684--688},
  publisher = {{Society for Industrial and Applied Mathematics}},
  doi = {10.1137/S0040585X97980051}
}

@article{Bakhtin_Stationary_2005,
  title = {Stationary Solutions of Stochastic Differential Equations with Memory and Stochastic Partial Differential Equations},
  author = {Bakhtin, Yuri and Mattingly, Jonathan C.},
  date = {2005-10},
  journaltitle = {Communications in Contemporary Mathematics},
  shortjournal = {Commun. Contemp. Math.},
  volume = {07},
  number = {05},
  pages = {553--582},
  publisher = {World Scientific Publishing Co.},
  doi = {10/d45cwb}
}

@article{Boue_variational_1998,
  title = {A Variational Representation for Certain Functionals of Brownian Motion},
  author = {Bou\'e, Michelle and Dupuis, Paul},
  date = {1998-10-01},
  journaltitle = {The Annals of Probability},
  shortjournal = {Ann. Probab.},
  volume = {26},
  number = {4},
  pages = {1641--1659},
  doi = {10.1214/aop/1022855876}
}

@article{Budhiraja_Large_2008,
  title = {Large Deviations for Infinite Dimensional Stochastic Dynamical Systems},
  author = {Budhiraja, Amarjit and Dupuis, Paul and Maroulas, Vasileios},
  date = {2008-07-01},
  journaltitle = {The Annals of Probability},
  shortjournal = {Ann. Probab.},
  volume = {36},
  number = {4},
  pages = {1390--1420},
  doi = {10/fssfx7}
}

@article{Budhiraja_variational_2000,
  title = {A Variational Representation for Positive Functionals of Infinite Dimensional Brownian Motion},
  author = {Budhiraja, Amarjit and Dupuis, Paul},
  date = {2000},
  journaltitle = {Probability and Mathematical Statistics},
  shortjournal = {Probab. Math. Statist.},
  volume = {20},
  number = {1},
  pages = {39--61},
  mrnumber = {1785237}
}

@online{Gao_Large_2022,
  title = {Large Deviations Principle for Stationary Solutions of Stochastic Differential Equations with Multiplicative Noise},
  author = {Gao, Peipei and Liu, Yong and Sun, Yue and Zheng, Zuohuan},
  date = {2022-06-06},
  eprint = {2206.02356},
  eprinttype = {arXiv},
  eprintclass = {math},
  doi = {10.48550/arXiv.2206.02356}
}

@article{Kurtz_YamadaWatanabeEngelbert_2007,
  title = {The Yamada-Watanabe-Engelbert Theorem for General Stochastic Equations and Inequalities},
  author = {Kurtz, Thomas},
  date = {2007-01},
  journaltitle = {Electronic Journal of Probability},
  volume = {12},
  pages = {951--965},
  publisher = {{Institute of Mathematical Statistics and Bernoulli Society}},
  doi = {10/fzsm9s}
}

@article{Liu_Large_2020,
  title = {Large Deviation Principle for a Class of SPDE with Locally Monotone Coefficients},
  author = {Liu, Wei and Tao, Chunyan and Zhu, Jiahui},
  date = {2020-06-01},
  journaltitle = {Science China Mathematics},
  shortjournal = {Sci. China Math.},
  volume = {63},
  number = {6},
  pages = {1181--1202},
  doi = {10/gs95fh}
}

@article{Liu_Representation_2009,
  title = {Representation of Pathwise Stationary Solutions of Stochastic Burgers' Equations},
  author = {Liu, Yong and Zhao, Huaizhong},
  date = {2009-12},
  journaltitle = {Stochastics and Dynamics},
  shortjournal = {Stoch. Dyn.},
  volume = {09},
  number = {04},
  pages = {613--634},
  doi = {10.1142/S0219493709002798}
}

@article{Rockner_YamadaWatanabe_2008,
  title = {{Yamada--Watanabe} Theorem for Stochastic Evolution Equations in Infinite Dimensions},
  author = {R{\"o}ckner, Michael and Schmuland, Byron and {Xicheng Zhang}},
  year = 2008,
  journal = {Condensed Matter Physics},
  volume = {11},
  number = {2},
  eid = {247},
  issn = {1607324X},
  doi = {10/gtbpcb},
  langid = {english}
}

@article{Scheutzow_stochastic_2013,
  title = {A Stochastic Gronwall Lemma},
  author = {Scheutzow, Michael},
  date = {2013},
  journaltitle = {Infinite Dimensional Analysis, Quantum Probability and Related Topics},
  shortjournal = {Infin. Dimens. Anal. Quantum Probab. Relat. Top.},
  volume = {16},
  number = {2},
  pages = {1350019, 4},
  doi = {10.1142/S0219025713500197},
  mrnumber = {3078830}
}

@book{Liu_Stochastic_2015,
  title = {Stochastic Partial Differential Equations: An Introduction},
  author = {Liu, Wei and R\"ockner, Michael},
  date = {2015},
  series = {Universitext},
  publisher = {Springer International Publishing},
  location = {Cham},
  doi = {10.1007/978-3-319-22354-4},
  isbn = {978-3-319-22353-7}
}

@article{Mohammed_Stable_2008,
  title = {The Stable Manifold Theorem for Semilinear Stochastic Evolution Equations and Stochastic Partial Differential Equations},
  author = {Mohammed, Salah-Eldin A. and Zhang, Tusheng and Zhao, Huaizhong},
  year = {2008},
  journal = {Memoirs of the American Mathematical Society},
  volume = {196},
  number = {917},
  doi = {10.1090/memo/0917}
}

@online{Liu_Large_2025,
  title = {Large Deviation Principle for the Stationary Solutions of Stochastic Functional Differential Equations with Infinite Delay},
  author = {Liu, Yong and Tang, Bin},
  date = {2025-01-13},
  eprint = {2501.07325},
  eprinttype = {arXiv},
  eprintclass = {math},
  pubstate = {prepublished}
}

@article{Flandoli_Dissipativity_1994,
  title = {Dissipativity and Invariant Measures for Stochastic Navier-Stokes Equations},
  author = {Flandoli, Franco},
  date = {1994-12},
  journaltitle = {Nonlinear Differential Equations and Applications NoDEA},
  shortjournal = {NoDEA},
  volume = {1},
  number = {4},
  pages = {403--423},
  doi = {10.1007/BF01194988}
}

@article{Cerrai_Stochastic_2003,
  title = {Stochastic Reaction-Diffusion Systems with Multiplicative Noise and Non-Lipschitz Reaction Term},
  author = {Cerrai, Sandra},
  date = {2003-02-01},
  journaltitle = {Probability Theory and Related Fields},
  shortjournal = {Probab Theory Relat Fields},
  volume = {125},
  number = {2},
  pages = {271--304},
  doi = {10.1007/s00440-002-0230-6}
}

@article{Brzezniak_Quasipotential_2015,
  title = {Quasipotential and Exit Time for 2D Stochastic Navier-Stokes Equations Driven by Space Time White Noise},
  author = {Brze{\'z}niak, Zdzis{\l}aw and Cerrai, Sandra and Mark, Freidlin},
  year = {2015},
  journal = {Probability Theory and Related Fields},
  volume = {162},
  number = {3},
  pages = {739--793},
  doi = {10.1007/s00440-014-0584-6}
}

@article{Cerrai_Large_2005,
  title = {Large Deviations for Invariant Measures of Stochastic Reaction--Diffusion Systems with Multiplicative Noise and Non-Lipschitz Reaction Term},
  author = {Cerrai, Sandra and R{\"o}ckner, Michael},
  year = {2005},
  journal = {Annales de l'Institut Henri Poincare (B) Probability and Statistics},
  volume = {41},
  number = {1},
  pages = {69--105},
  doi = {10.1016/j.anihpb.2004.03.001},
  copyright = {https://www.elsevier.com/tdm/userlicense/1.0/}
}

@article{Wang_Large_2024,
  title = {Large Deviations of Invariant Measures of Stochastic Reaction--Diffusion Equations on Unbounded Domains},
  author = {Wang, Bixiang},
  year = {2024},
  journal = {Journal of Statistical Physics},
  volume = {191},
  number = {96},
  doi = {10.1007/s10955-024-03316-6}
}

@article{Bai_2026_Large,
  title = {Large Deviation Principle for Invariant Measures of Stochastic {{Burgers}} Equations},
  author = {Bai, Rui and Feng, Chunrong and Zhao, Huaizhong},
  date = {2026-03-01},
  journaltitle = {Journal of Functional Analysis},
  shortjournal = {Journal of Functional Analysis},
  volume = {290},
  number = {5},
  eid = {111284},
  issn = {0022-1236},
  doi = {10.1016/j.jfa.2025.111284}
}

@article{Pan_Large_2026,
  title = {Large Deviations for Fully Local Monotone Stochastic Partial Differential Equations Driven by Gradient-Dependent Noise},
  author = {Pan, Tianyi and Shang, Shijie and Zhai, Jianliang and Zhang, Tusheng},
  date = {2026-02},
  journaltitle = {Bernoulli},
  volume = {32},
  number = {1},
  pages = {249--273},
  publisher = {{Bernoulli Society for Mathematical Statistics and Probability}},
  doi = {10.3150/25-BEJ1857}
}

@article{Kager_Generation_1997,
author = {Gerald Kager and Michael Scheutzow},
title = {{Generation of One-Sided Random Dynamical Systems by Stochastic Differential Equations}},
volume = {2},
journal = {Electronic Journal of Probability},
number = {none},
publisher = {Institute of Mathematical Statistics and Bernoulli Society},
pages = {1 -- 17},
year = {1997},
doi = {10.1214/EJP.v2-22},
URL = {https://doi.org/10.1214/EJP.v2-22}
}

@article{Brzezniak_2014_Strong,
  title = {Strong Solutions for {{SPDE}} with Locally Monotone Coefficients Driven by {{L{\'e}vy}} Noise},
  author = {Brze{\'z}niak, Zdzis{\l}aw and Liu, Wei and Zhu, Jiahui},
  date = {2014-06},
  journaltitle = {Nonlinear Analysis: Real World Applications},
  shortjournal = {Nonlinear Analysis: Real World Applications},
  volume = {17},
  pages = {283--310},
  issn = {14681218},
  doi = {10.1016/j.nonrwa.2013.12.005},
  langid = {english},
}

@article{Krylov_1981_Stochastic,
  title = {Stochastic Evolution Equations},
  author = {Krylov, Nicolai V. and Rozovskii, Boris L.},
  date = {1981},
  journaltitle = {Journal of Soviet Mathematics},
  shortjournal = {J Math Sci},
  volume = {16},
  number = {4},
  pages = {1233--1277},
  issn = {0090-4104, 1573-8795},
  doi = {10.1007/BF01084893},
  langid = {english},
}

@article{Kumar_2024_Wellposedness,
  title = {Well-Posedness of a Class of Stochastic Partial Differential Equations with Fully Monotone Coefficients Perturbed by {{L{\'e}vy}} Noise},
  author = {Kumar, Ankit and Mohan, Manil T.},
  date = {2024-06},
  journaltitle = {Analysis and Mathematical Physics},
  shortjournal = {Anal.Math.Phys.},
  volume = {14},
  number = {44},
  issn = {1664-2368, 1664-235X},
  doi = {10.1007/s13324-024-00898-y},
  langid = {english},
}

@article{Liu_2013_Wellposedness,
  title = {Well-Posedness of Stochastic Partial Differential Equations with {{Lyapunov}} Condition},
  author = {Liu, Wei},
  date = {2013-08},
  journaltitle = {Journal of Differential Equations},
  shortjournal = {Journal of Differential Equations},
  volume = {255},
  number = {3},
  pages = {572--592},
  issn = {00220396},
  doi = {10.1016/j.jde.2013.04.021},
  langid = {english},
}

@article{Neelima_2020_Coercivity,
  title = {Coercivity Condition for Higher Moment a Priori Estimates for Nonlinear {{SPDEs}} and Existence of a Solution under Local Monotonicity},
  author = {{Neelima} and {\v{S}}i{\v{s}}ka, David},
  date = {2020-07-03},
  journaltitle = {Stochastics},
  shortjournal = {Stochastics},
  volume = {92},
  number = {5},
  pages = {684--715},
  issn = {1744-2508, 1744-2516},
  doi = {10.1080/17442508.2019.1650043},
  langid = {english},
}

@article{Nguyen_2021_Nonlinear,
  title = {Nonlinear Stochastic Parabolic Partial Differential Equations with a Monotone Operator of the {{Ladyzenskaya-Smagorinsky}} Type, Driven by a {{L{\'e}vy}} Noise},
  author = {Nguyen, Phuong and Tawri, Krutika and Temam, Roger},
  date = {2021-10},
  journaltitle = {Journal of Functional Analysis},
  shortjournal = {Journal of Functional Analysis},
  volume = {281},
  number = {8},
  eid = {109157},
  issn = {00221236},
  doi = {10.1016/j.jfa.2021.109157},
  langid = {english},
}

@article{Pardouxt_1980_Stochastic,
  title = {Stochastic Partial Differential Equations and Filtering of Diffusion Processes},
  author = {Pardoux, {\'E}tienne},
  date = {1980-01},
  journaltitle = {Stochastics},
  shortjournal = {Stochastics},
  volume = {3},
  number = {1--4},
  pages = {127--167},
  issn = {0090-9491},
  doi = {10.1080/17442507908833142},
  langid = {english}
}

@article{Ren_2007_Stochastic,
  title = {Stochastic Generalized Porous Media and Fast Diffusion Equations},
  author = {Ren, Jiagang and R{\"o}ckner, Michael and Wang, Feng-Yu},
  date = {2007-07},
  journaltitle = {Journal of Differential Equations},
  shortjournal = {Journal of Differential Equations},
  volume = {238},
  number = {1},
  pages = {118--152},
  issn = {00220396},
  doi = {10.1016/j.jde.2007.03.027},
  langid = {english}
}

@article{Rockner_2024_Wellposedness,
  title = {Well-Posedness of Stochastic Partial Differential Equations with Fully Local Monotone Coefficients},
  author = {R\"ockner, Michael and Shang, Shijie and Zhang, Tusheng},
  date = {2024-11},
  journaltitle = {Mathematische Annalen},
  shortjournal = {Math. Ann.},
  volume = {390},
  number = {3},
  pages = {3419--3469},
  issn = {0025-5831, 1432-1807},
  doi = {10.1007/s00208-024-02836-6},
  langid = {english},
}

@article{Pardoux_1974_Equations,
  title = {{\'E}quations aux d{\'e}riv{\'e}es partielles stochastiques de type monotone},
  author = {Pardoux, {\'E}tienne},
  date = {1974/1975},
  journaltitle = {S{\'e}minaire sur les {\'e}quations aux d{\'e}riv{\'e}es partielles},
  number = {3},
  pages = {1--10},
  issn = {2743-0588},
  url = {https://www.numdam.org/item/?id=SJL_1974-1975___3_A2_0},
  urldate = {2026-01-23},
  langid = {french}
}

@book{Arnold_Random_2003,
  title = {Random Dynamical Systems},
  author = {Arnold, Ludwig},
  year = 2003,
  series = {Springer Monographs in Mathematics},
  edition = {Corr. 2. printing},
  publisher = {Springer},
  address = {Berlin Heidelberg},
  isbn = {978-3-540-63758-5},
  langid = {english}
}

@article{Varadhan_1966_Asymptotic,
  title = {Asymptotic Probabilities and Differential Equations},
  author = {Varadhan, SR Srinivasa},
  year = 1966,
  journal = {Communications on Pure and Applied Mathematics},
  volume = {19},
  number = {3},
  pages = {261--286},
  issn = {1097-0312},
  doi = {10.1002/cpa.3160190303},
  copyright = {Copyright \copyright{} 1966 Wiley Periodicals, Inc., A Wiley Company},
  langid = {english}
}

@book{Freidlin_2012_Random,
  title = {Random {{Perturbations}} of {{Dynamical Systems}}},
  author = {Freidlin, Mark I. and Wentzell, Alexander D.},
  year = 2012,
  series = {Grundlehren Der Mathematischen {{Wissenschaften}}},
  volume = {260},
  publisher = {Springer},
  address = {Berlin, Heidelberg},
  doi = {10.1007/978-3-642-25847-3},
  isbn = {978-3-642-25846-6},
  langid = {english}
}

@book{Budhiraja_2019_Analysis,
  title = {Analysis and {{Approximation}} of {{Rare Events}}: {{Representations}} and {{Weak Convergence Methods}}},
  shorttitle = {Analysis and {{Approximation}} of {{Rare Events}}},
  author = {Budhiraja, Amarjit and Dupuis, Paul},
  year = 2019,
  series = {Probability {{Theory}} and {{Stochastic Modelling}}},
  volume = {94},
  publisher = {Springer US},
  address = {New York, NY},
  doi = {10.1007/978-1-4939-9579-0},
  copyright = {http://www.springer.com/tdm},
  isbn = {978-1-4939-9577-6},
  langid = {english}
}

@book{Dupuis_1997_Weak,
  title = {A {{Weak Convergence Approach}} to the {{Theory}} of {{Large Deviations}}},
  author = {Dupuis, Paul and Ellis, Richard S.},
  year = 1997,
  series = {Wiley {{Series}} in {{Probability}} and {{Statistics}}},
  edition = {1},
  publisher = {Wiley},
  doi = {10.1002/9781118165904},
  isbn = {978-0-471-07672-8},
  langid = {english}
}

@article{Freidlin_1988_Random,
  title = {Random {{Perturbations}} of {{Reaction-Diffusion Equations}}: {{The Quasi-Deterministic Approximation}}},
  shorttitle = {Random {{Perturbations}} of {{Reaction-Diffusion Equations}}},
  author = {Freidlin, Mark I.},
  year = 1988,
  journal = {Transactions of the American Mathematical Society},
  volume = {305},
  number = {2},
  eprinttype = {jstor},
  pages = {665--697},
  publisher = {American Mathematical Society},
  issn = {0002-9947},
  doi = {10.2307/2000884},
  langid = {american}
}

@article{Sowers_1992_Large,
  title = {Large Deviations for the Invariant Measure of a Reaction-Diffusion Equation with Non-{{Gaussian}} Perturbations},
  author = {Sowers, Richard},
  year = 1992,
  journal = {Probability Theory and Related Fields},
  volume = {92},
  number = {3},
  pages = {393--421},
  issn = {1432-2064},
  doi = {10.1007/BF01300562},
  langid = {english}
}

@article{Brzezniak_Large_2017,
  title = {Large Deviations Principle for the Invariant Measures of the {{2D}} Stochastic {{Navier}}--{{Stokes}} Equations on a Torus},
  author = {Brze{\'z}niak, Zdzis{\l}aw and Cerrai, Sandra},
  year = 2017,
  journal = {Journal of Functional Analysis},
  volume = {273},
  number = {6},
  pages = {1891--1930},
  issn = {0022-1236},
  doi = {10.1016/j.jfa.2017.05.008}
}

@article{Cerrai_2022_Large,
  title = {Large Deviations Principle for the Invariant Measures of the {{2D}} Stochastic {{Navier}}--{{Stokes}} Equations with Vanishing Noise Correlation},
  author = {Cerrai, Sandra and Paskal, Nicholas},
  year = 2022,
  journal = {Stochastics and Partial Differential Equations: Analysis and Computations},
  volume = {10},
  number = {4},
  pages = {1651--1681},
  issn = {2194-041X},
  doi = {10.1007/s40072-021-00219-5},
  langid = {english}
}

@online{Klose_2024_Large,
  title = {Large Deviations of the $\Phi^4_3$ Measure via Stochastic Quantisation},
  author = {Klose, Tom and Mayorcas, Avi},
  year = {2024},
  eprint = {2402.00975},
  eprinttype = {arXiv},
  eprintclass = {math-ph}
}

@article{Zhang_2012_Large,
  title = {Large Deviations for Invariant Measures of {{SPDEs}} with Two Reflecting Walls},
  author = {Zhang, Tusheng},
  year = 2012,
  journal = {Stochastic Processes and their Applications},
  volume = {122},
  number = {10},
  pages = {3425--3444},
  issn = {0304-4149},
  doi = {10.1016/j.spa.2012.06.003}
}

@article{Ito_Stationary_1964,
  title = {On Stationary Solutions of a Stochastic Differential Equation},
  author = {It{\^o}, Kiyosi and Nisio, Makiko},
  year = 1964,
  journal = {Kyoto Journal of Mathematics},
  volume = {4},
  number = {1},
  issn = {2156-2261},
  doi = {10.1215/kjm/1250524705}
}

@article{Arnold_Perfect_1995,
  title = {Perfect Cocycles through Stochastic Differential Equations},
  author = {Arnold, Ludwig and Scheutzow, Michael},
  year = 1995,
  journal = {Probability Theory and Related Fields},
  volume = {101},
  number = {1},
  pages = {65--88},
  issn = {1432-2064},
  doi = {10/cws3n7},
  langid = {english}
}

@article{Bakhtin_2014_Spacetime,
  title = {Space-Time Stationary Solutions for the {{Burgers}} Equation},
  author = {Bakhtin, Yuri and Cator, Eric and Khanin, Konstantin},
  year = 2014,
  journal = {Journal of the American Mathematical Society},
  volume = {27},
  number = {1},
  pages = {193--238},
  issn = {0894-0347, 1088-6834},
  doi = {10.1090/S0894-0347-2013-00773-0},
  langid = {english}
}

@article{Dunlap_2021_Stationary,
  title = {Stationary {{Solutions}} to the {{Stochastic Burgers Equation}} on the {{Line}}},
  author = {Dunlap, Alexander and Graham, Cole and Ryzhik, Lenya},
  year = 2021,
  journal = {Communications in Mathematical Physics},
  volume = {382},
  number = {2},
  pages = {875--949},
  issn = {1432-0916},
  doi = {10/gs4rz5},
  langid = {english}
}

@article{Breit_2019_Stationary,
  title = {Stationary Solutions to the Compressible {{Navier}}--{{Stokes}} System Driven by Stochastic Forces},
  author = {Breit, Dominic and Feireisl, Eduard and Hofmanov{\'a}, Martina and Maslowski, Bohdan},
  year = 2019,
  journal = {Probability Theory and Related Fields},
  volume = {174},
  number = {3-4},
  pages = {981--1032},
  issn = {0178-8051, 1432-2064},
  doi = {10.1007/s00440-018-0875-4},
  langid = {english}
}

@article{Flandoli_1999_Weak,
  title = {Weak {{Solutions}} and {{Attractors}} for {{Three-Dimensional Navier-Stokes Equations}} with {{Nonregular Force}}},
  author = {Flandoli, Franco and Schmalfu{\ss}, Bj{\"o}rn},
  year = 1999,
  journal = {Journal of Dynamics and Differential Equations},
  volume = {11},
  number = {2},
  pages = {355--398},
  issn = {10407294},
  doi = {10.1023/A:1021937715194}
}

@article{Jiang_2023_Global,
  title = {Global Stability of Stationary Solutions for a Class of Semilinear Stochastic Functional Differential Equations with Additive White Noise},
  author = {Jiang, Jifa and Lv, Xiang},
  year = 2023,
  journal = {Journal of Differential Equations},
  volume = {367},
  pages = {890--921},
  issn = {0022-0396},
  doi = {10.1016/j.jde.2023.05.035}
}

@article{Odasso_2008_Exponential,
  title = {Exponential Mixing for Stochastic {{PDEs}}: The Non-Additive Case},
  shorttitle = {Exponential Mixing for Stochastic {{PDEs}}},
  author = {Odasso, Cyril},
  year = 2008,
  journal = {Probability Theory and Related Fields},
  volume = {140},
  number = {1},
  pages = {41--82},
  issn = {1432-2064},
  doi = {10.1007/s00440-007-0057-2},
  langid = {english}
}

@article{Gess_2025_Stabilization,
  title = {Stabilization by Transport Noise and Enhanced Dissipation in the {{Kraichnan}} Model},
  author = {Gess, Benjamin and Yaroslavtsev, Ivan},
  year = 2025,
  journal = {Journal of Evolution Equations},
  volume = {25},
  number = {42},
  issn = {1424-3202},
  doi = {10.1007/s00028-025-01066-w},
  langid = {english}
}

@article{Flandoli_2024_Quantitative,
  title = {Quantitative Convergence Rates for Scaling Limit of {{SPDEs}} with Transport Noise},
  author = {Flandoli, Franco and Galeati, Lucio and Luo, Dejun},
  year = 2024,
  journal = {Journal of Differential Equations},
  volume = {394},
  pages = {237--277},
  issn = {0022-0396},
  doi = {10.1016/j.jde.2024.02.053}
}

@article{Galeati_2023_Ldp,
  title = {{{LDP}} and {{CLT}} for {{SPDEs}} with Transport Noise},
  author = {Galeati, Lucio and Luo, Dejun},
  year = 2023,
  journal = {Stochastics and Partial Differential Equations: Analysis and Computations},
  issn = {2194-041X},
pages = {736--793},
  doi = {10/gr7vzq}
}

@article{Ferrario_1999_Stochastic,
  title = {Stochastic {{Navier-Stokes}} Equations: {{Analysis}} of the Noise to Have a Unique Invariant Measure},
  shorttitle = {Stochastic {{Navier-Stokes}} Equations},
  author = {Ferrario, Benedetta},
  year = 1999,
  journal = {Annali di Matematica Pura ed Applicata},
  volume = {177},
  number = {1},
  pages = {331--347},
  issn = {1618-1891},
  doi = {10.1007/BF02505916},
  langid = {english}
}

@article{Buckmaster_2019_Nonuniqueness,
  title = {Nonuniqueness of Weak Solutions to the {{Navier-Stokes}} Equation},
  author = {Buckmaster, Tristan and Vicol, Vlad},
  year = 2019,
  journal = {Annals of Mathematics},
  volume = {189},
  number = {1},
pages = {101--144},
  issn = {0003-486X},
  doi = {10.4007/annals.2019.189.1.3}
}

@online{Qiu_2023_Large,
  title = {Large Deviations of Invariant Measure for the 3D Stochastic Hyperdissipative {Navier-Stokes} Equations},
  author = {Qiu, Zhaoyang and Liu, Hui and Sun, Chengfeng},
  year = {2023},
  eprint = {2307.04271},
  eprinttype = {arXiv},
  eprintclass = {math}
}

@article{Sermange_1983_Mathematical,
  title = {Some Mathematical Questions Related to the Mhd Equations},
  author = {Sermange, Michel and Temam, Roger},
  year = 1983,
  journal = {Communications on Pure and Applied Mathematics},
  volume = {36},
  number = {5},
  pages = {635--664},
  issn = {1097-0312},
  doi = {10.1002/cpa.3160360506},
  copyright = {Copyright \copyright{} 1983 Wiley Periodicals, Inc., A Wiley Company},
  langid = {english}
}

@article{Barbu_2007_Existence,
  title = {Existence and {{Ergodicity}} for the {{Two-Dimensional Stochastic Magneto-Hydrodynamics Equations}}},
  author = {Barbu, Viorel and Da Prato, Giuseppe},
  year = 2007,
  journal = {Applied Mathematics and Optimization},
  volume = {56},
  number = {2},
  pages = {145--168},
  issn = {1432-0606},
  doi = {10.1007/s00245-007-0882-2},
  langid = {english}
}

\end{document}